\documentclass[namedate,webpdf,imanum]{ima-authoring-template}

\usepackage{mathrsfs}
\usepackage{mathtools}
\usepackage{xfrac}
\usepackage{booktabs}
\usepackage[table]{xcolor}

\newcommand{\Real}{\mathbb{R}}
\newcommand{\N}{\mathbb{N}}
\newcommand{\Z}{\mathbb{Z}}
\newcommand{\F}{\mathbb{F}}

\newcommand{\Matrix}[1]{\mathbf{#1}}
\newcommand{\trans}{\intercal}

\newcommand{\diag}[2][]{\mathrm{diag}#1( #2 #1)}

\newcommand{\Norm}[3][]{#1\| #2 #1\|_{\mathrm{#3}}}

\newcommand{\hap}{\mathbin{\odot}}
\newcommand{\krp}{\mathbin{\otimes}}
\newcommand{\vect}{\mathrm{vec}}
\newcommand{\Vect}[2][]{\vect#1( #2 #1)}
\newcommand{\One}{\Matrix{1}}

\newcommand{\Layer}[1]{\mathsf{#1}}
\newcommand{\Id}{\Layer{Id}}

\newcommand{\dist}[3][]{\rho_{\mathrm{#2}} #1( #3 #1)}
\newcommand{\cond}[3][]{\kappa_{\mathrm{#2}} #1( #3 #1)}

\newcommand{\condabs}[3][]{\tilde{\kappa}_{\mathrm{#2}} #1( #3 #1)}
\newcommand{\fl}[2][]{\mathsf{fl}#1( #2 #1)}
\newcommand{\uh}{\mathsf{u}}
\newcommand{\Oh}[2][]{\mathcal{O}#1( #2 #1)}

\newcommand{\set}[3][]{#1\{ #2 : #3 #1\}}
\newcommand{\supp}[2][]{\mathrm{supp}#1( #2 #1)}

\newtheorem{theorem}{Theorem}[section]
\newtheorem{proposition}[theorem]{Proposition}%

\newtheorem{lemma}[theorem]{Lemma}

\newtheorem{model}[theorem]{Model}

\numberwithin{equation}{section}

\begin{document}

\newif\ifpreprint
\preprinttrue

\ifpreprint
\input{preprint_maketitle}
\else
\fi

\DOI{DOI HERE}
\copyrightyear{2021}
\vol{00}
\pubyear{2021}
\access{Advance Access Publication Date: Day Month Year}
\appnotes{Paper}
\copyrightstatement{Published by Oxford University Press on behalf of the Institute of Mathematics and its Applications. All rights reserved.}
\firstpage{1}


\title[Numerical stability analysis of large language models]{Numerical stability analysis of large language models}

\author{Stanislav Budzinskiy \ORCID{0000-0001-7983-7900}
\address{\orgdiv{Faculty of Mathematics}, \orgname{University of Vienna}, \orgaddress{\street{Kolingasse 14-16}, \postcode{1090}, \state{Vienna}, \country{Austria}}}}

\author{Wenyi Fang and Longbin Zeng
\address{\orgname{Huawei Technologies}}}

\author{Philipp Petersen* \ORCID{0000-0003-3566-1020}
\address{\orgdiv{Faculty of Mathematics and Research Network Data Science @ Uni Vienna}, \orgname{University of Vienna}, \orgaddress{\street{Kolingasse 14-16}, \postcode{1090}, \state{Vienna}, \country{Austria}}}}

\corresp[*]{Corresponding author: \href{email:philipp.petersen@univie.ac.at}{philipp.petersen@univie.ac.at}}

\received{Date}{0}{Year}
\revised{Date}{0}{Year}
\accepted{Date}{0}{Year}



\abstract{
Transformers are the state-of-the-art architecture for large language models, and a key to their scalability is the strategic usage of low-precision arithmetic.
We develop a mixed-precision analysis of transformer inference, deriving bounds for the condition numbers and forward error of the architecture's constituent parts.
Notably, we compare the numerical stability of LayerNorm and RMSNorm in the massive-outlier regime, tighten the error bound of softmax in the presence of attention sinks, and quantify the impact of its shifted evaluation on the sensitivity to perturbations.
Furthermore, we derive novel sequence-length-independent bounds on the local Lipschitz constant of self-attention.
Our worst-case error bound for transformer inference suggests that its numerical stability is determined by the interplay between weight magnitude and the growth of the residual stream.
Crucially, and as validated by experiments with GPT-2, our analysis establishes that the scaling of residual-projection weights preserves the propagation of the relative rounding error unless it forces a qualitative transition in the dynamics of the residual stream.
}

\keywords{deep learning; transformer; inference; numerical stability; mixed precision; condition number.}

\maketitle


\section{Introduction}
\emph{Transformers} \citep{vaswani2017attention} constitute a family of neural-network architectures that serves as the backbone for \emph{large language models} (LLMs).
While specific transformer architectures have evolved over time (see Section~\ref{subsec:transformer_architecture}), their core design principle persists: this is a deep composition of alternating \emph{attention} and \emph{feedforward} (FF) blocks.
Originally tailored for natural language processing \citep{devlin2019bert}, transformers have been successfully adapted to tasks ranging from computer vision \citep{dosovitskiy2021image} to computational biology \citep{ying2021transformers, lin2023evolutionary} and mathematical reasoning \citep{frieder2023mathematical}.
This versatility stems largely from the scaling of models to colossal numbers of parameters.

The practical realisation of massive LLMs is enabled by the strategic usage of low-precision number formats for both training and inference \citep{micikevicius2018mixed, nagel2021white}.
Consider model deployment.
First, weight-only quantisation, implemented as quantisation-aware training \citep{jacob2018quantization} or post-training quantisation \citep{frantar2023gptq, lin2024awq}, reduces the disk-storage requirements of the deployed model.
Second, weight-activation quantisation, using zero-shot \citep{dettmers2022gpt3, rouhani2023microscaling} or calibration-based \citep{xiao2023smoothquant} methods, optimises the runtime accelerator-memory demands.
Third, mixed-precision \emph{general matrix multiplication} (GEMM) kernels exploit hardware parallelism and cache locality to maximise inference throughput, relying on quantised low-precision operands \citep{sze2017efficient, markidis2018nvidia}.

Despite these substantial efficiency gains, quantisation introduces numerical perturbations into the GEMM operands, which propagate through the entire transformer architecture and are compounded by the rounding errors of \emph{floating-point} (FP) operations.
The established practical heuristic, exemplified by frontier LLMs \citep{liu2024deepseek}, prescribes that while quantisation can be aggressive, both intermediate GEMM accumulation and the evaluation of non-linear functions require higher bit-widths.
The extended dynamic range of wider number formats prevents swamping and overflows \citep{micikevicius2018mixed, colbert2024accumulator}, and their increased precision preserves the training and predictive accuracies \citep{kim2021bert, shah2024flashattention, qi2025defeating}.

While the impact of quantisation and rounding errors in transformers has been a subject of extensive empirical studies, theoretical analyses remain scarce---our work addresses this deficit.

\subsection{Contributions and outline}
We develop a numerical analysis of transformer inference, bounding the condition numbers and forward error of its constituent parts.
Our rounding error analysis is deterministic \citep{higham2002accuracy} and mixed-precision, distinguishing between quantisation, accumulation, residual, and working precisions.

We provide a numerical-stability perspective on the transition from \emph{LayerNorm} \citep{ba2016layer} to \emph{RMSNorm} \citep{zhang2019root} in modern LLMs, focusing on the regime of \emph{massive outliers}.
While this architectural shift was introduced to reduce the computational complexity of normalisation, a comparison of condition numbers shows that neither function is strictly superior in terms of sensitivity.
Meanwhile, the forward error of LayerNorm is governed by the magnitude of the outlier's background entries, whereas RMSNorm exhibits unconditional forward stability.
See Section~\ref{sec:layernorm}.

The rounding error analysis of FF blocks in Section~\ref{sec:ffn} shows that the FP evaluation of both \emph{GELU} \citep{hendrycks2016gaussian} and \emph{SwiGLU} \citep{shazeer2020glu} mechanisms is forward stable.
We bound the condition numbers for general FF blocks with smooth ReLU-like activation (or gating) functions.

In Section~\ref{sec:softmax}, we focus on the practical aspects of evaluating \emph{softmax}.
Our analysis of the unshifted and shifted algorithms shows that the shift, while preventing overflow, at most quadruples the sensitivity to perturbations.
For concentrated softmax probability distributions, we prove a rounding-error bound that takes absorption effects into account and scales with the \emph{effective} sequence length in long-context scenarios; this is particularly relevant in the presence of \emph{attention sinks} \citep{xiao2024efficient}.

Section~\ref{sec:attention} is devoted to \emph{self-attention}.
We derive a collection of bounds on its local Lipschitz constant that are independent of the input sequence length, improving on the mean-field bound from \cite{castin2024smooth} in the discrete case.\footnote{We independently developed these bounds during the revision of our paper. After finalising the analysis, we became aware of concurrent work \cite{emadi2026exact} deriving a special case.}
An even smaller bound holds in the presence of a massive outlier.

The forward error for the entire deep transformer architecture is analysed in Section~\ref{sec:transformer}.
Our bound suggests that the numerical stability of inference is determined not only by the magnitude of weights---depth-dependent scaling at initialisation has become standard since GPT-2 \citep{radford2019language}---but also by the growth dynamics of the residual stream.
Crucially, the relative-error bound is invariant to the scaling of residual-projection weights when the residual stream responds linearly to this scaling, a mathematical property we validate empirically via experiments with GPT-2.

Notation and background material on transformers, condition numbers, and rounding error analysis are provided in Section~\ref{sec:background}.
Appendix~\ref{appendix:llm_comparison} describes the evolution of architectural details across generations of LLMs, and Appendix~\ref{appendix:layernorm-cond-comparison} contains auxiliary formulas for Section~\ref{sec:layernorm}.

\subsection{Limitations}
First, our derivation focuses on a representative transformer architecture.
Although the analysis is more broadly applicable,\footnote{For example, all bounds hold for grouped-query and multi-query attention \citep{shazeer2019fast, ainslie2023gqa}.} we introduce simplifying structural assumptions to maintain clarity of exposition: we exclude parallel architectures \citep{chowdhery2023palm, malartic2024falcon2} in favour of a standard sequential architecture; exclude all bias terms; exclude sparse \citep{beltagy2020longformer} and latent \citep{liu2024deepseek} attention mechanisms in favour of standard causal self-attention; and exclude key-value caching \citep{dai2019transformer}.
The architecture we study is a decoder-only deep transformer with pre-normalisation, GELU or SwiGLU FF mechanism, and causal multi-head self-attention with \emph{RoPE} \citep{su2024roformer}.
Find a detailed description of the architecture in Section~\ref{subsec:transformer_architecture}.

Second, we adopt deterministic rounding \citep{higham2002accuracy}.
Although the worst-case error bounds it yields are conservative compared to average-case bounds of stochastic rounding \citep{croci2022stochastic}, the deterministic setting allows for tractable error analysis of non-trivial compositions.
Rigorous stochastic-rounding error analysis of nonlinear functions remains in the early stages of development \citep{el2024bounds}, necessitating auxiliary probabilistic assumptions to make it feasible \citep{beuzeville2026deterministic}.

Third, we treat quantisation errors strictly as deterministic rounding errors to abstract away specific quantisation schemes \citep{gholami2022survey}---such as outlier-aware integer quantisation \citep{dettmers2022gpt3} or micro-scaled FP quantisation \citep{rouhani2023microscaling}---and thereby facilitate the analysis.
A more fine-grained analysis could be performed for any particular quantisation scheme using specialised techniques \citep{abdelfattah2025analysis}.

Fourth, we assume that the four precisions are fixed throughout the entire transformer.
That is, each mixed-precision GEMM rounds all operand entries to the same quantisation precision and accumulates every inner product in the same accumulation precision; likewise, each nonlinearity is evaluated in the same working precision and every residual update is rounded to the same residual precision.
While this is a standard theoretical assumption, we note that there exist featurewise mixed-precision quantisation schemes \citep{dettmers2022gpt3} in practice, whereas the precisions used to accumulate inner products and evaluate nonlinearities remain fixed in inference kernels.\footnote{In a follow-up paper \citep{budzinskiy2026lamp}, which appeared after the first version of the present work was made public, we investigate adaptive mixed-precision accumulation for transformer inference.}

\subsection{Related work}
The generalisation properties of a neural network and its robustness with respect to input perturbations can be characterised via the Lipschitz constants of its layers: the smaller, the better \citep{bartlett2017spectrally, cisse2017parseval, weng2018evaluating, gouk2021regularisation}.
Meanwhile, there are investigations into the robustness-expressivity trade-off, as the representation power of a neural network with small Lipschitz constant may be limited in some cases \citep{tsipras2018robustness, anil2019sorting, bethune2022pay}.

Upper bounds for the local Lipschitz constant of self-attention were derived in \cite{kim2021lipschitz, castin2024smooth, yudin2025pay}.
For \emph{residual} architectures, including transformers, the (upper bound of the) Lipschitz constant of the neural network necessarily grows with depth as the product of Lipschitz constants of its residual blocks.
Modified residual connections can mitigate the exponential growth of the Lipschitz constant, e.g., by proper weight normalisation \citep{bachlechner2021rezero, wang2024deepnet, newhouse2025training}.
When the residual connections are excluded from a transformer, its Lipschitz constant decays rapidly with depth, leading to rank collapse and a loss of expressivity \citep{dong2021attention}.
Similarly, improper weight scaling, or lack thereof, not only causes the Lipschitz constant to explode exponentially but can also lead to the same rank collapse \citep{noci2022signal}.

In the context of adversarial stability, backward error analysis of FF networks was used to design adversarial attacks \citep{beuzeville2021adversarial, beerens2024adversarial}.
This rounding error analysis of inference was further extended to stochastic rounding in \cite{beuzeville2026deterministic} and motivated a mixed-precision accumulation strategy in \cite{el2025mixed}.

The propagation of quantisation errors, treated as random variables, was analysed theoretically in \cite{lin2016fixed, sakr2017analytical} with a focus on signal-to-noise ratio and classification-mismatch probability, respectively.
The propagation was explicitly taken into account for layer-wise post-training quantisation in \cite{nagel2020up, arai2025quantization}.
In \cite{malinovskii2025higgs}, the expected perplexity of a neural network was bounded via the expected relative quantisation errors of its weights.
\section{Background and preliminaries}
\label{sec:background}

\subsection{Notation}
\label{subsec:notation}
We denote vectors and matrices by $\Matrix{x} \in \Real^n$ and $\Matrix{Y} \in \Real^{m \times n}$.
The $j$th entry of $\Matrix{x}$ will be written as $x_j$, the $j$th column of $\Matrix{Y}$ as $\Matrix{y}_j$, and the $(i,j)$th entry of $\Matrix{Y}$ as $y_{i,j}$.
For an index set $J \subseteq \{ 1, \ldots, n \}$, we denote by $\Matrix{x}_{J}$ and $\Matrix{Y}_{J}$ the restriction of $\Matrix{x}$ and $\Matrix{Y}$ to the entries and columns indexed by $J$, respectively, preserving their natural order.
We will write $[i,j] = \{ i, i + 1, \ldots, j \}$ for $i \leq j$.
We denote by $\Matrix{e}_1, \ldots, \Matrix{e}_m \in \Real^m$ the columns of the identity matrix $\Matrix{I}_m \in \Real^{m \times m}$ and let $\One = \sum_{i = 1}^{m} \Matrix{e}_i$ be the vector of all ones (the lengths of $\Matrix{e}_i$ and $\One$ will be deducible from the context). 
The absolute value of a vector or matrix is understood in the entrywise sense.
The Hadamard product of matrices is denoted by $\hap$ and the Kronecker delta by $\delta_{i,j}$.

For each $1 \leq p,q \leq \infty$, the $\ell_p$ vector norm will be written as $\Norm{\cdot}{p}$ and the matrix $\ell_p \to \ell_q$ operator norm as $\Norm{\cdot}{p,q}$.
Recall that $\Norm{\Matrix{Y}}{2,2}$ is the largest singular value of $\Matrix{Y}$. 
Less known are the equalities
\begin{equation*}
    \Norm{\Matrix{Y}}{1,q} = \max_{\Matrix{x} \in \{ \Matrix{e}_1, \ldots, \Matrix{e}_n\} } \Norm{\Matrix{Y} \Matrix{x}}{q} = \max_{1 \leq j \leq n} \Norm{\Matrix{y}_j}{q}, \quad \Norm{\Matrix{Y}}{\infty,q} = \max_{\Matrix{x} \in \{ -1,1 \}^n} \Norm{\Matrix{Y} \Matrix{x}}{q},
\end{equation*}
which follow from maximising a convex function on a polytope.
It holds that $\Norm{\Matrix{Y}^\trans}{p,q} = \Norm{\Matrix{Y}}{q*, p*}$ with $1/q + 1/q* = 1/p + 1/p* = 1$ \cite[Theorem~5.6.35]{horn2012matrix}.
With an abuse of notation, we define $\Norm{\Matrix{x}}{-\infty} = \min_{1 \leq i \leq d} |x_i|$.
The Kronecker product $\Matrix{Y} \krp \Matrix{Z} \in \Real^{mk \times nl}$ with $\Matrix{Z} \in \Real^{k \times l}$ is defined as
\begin{equation*}
    \Matrix{Y} \krp \Matrix{Z} = \begin{bmatrix}
        y_{1,1} \Matrix{Z} & \cdots & y_{1,n} \Matrix{Z} \\
        \vdots & \ddots & \vdots \\
        y_{m,1} \Matrix{Z} & \cdots & y_{m,n} \Matrix{Z} \\
    \end{bmatrix},
\end{equation*}
and it holds that $\Norm{\Matrix{Y} \krp \Matrix{Z}}{p,q} = \Norm{\Matrix{Y}}{p,q} \Norm{\Matrix{Z}}{p,q}$ \cite[Theorem~8]{lancaster1972norms}.
For conformable matrices, the following mixed-product property holds \cite[Lemma~4.2.10]{horn1994topics}:
\begin{equation*}
    (\Matrix{Y}_1 \krp \Matrix{Z}_1)(\Matrix{Y}_2 \krp \Matrix{Z}_2) = \Matrix{Y}_1\Matrix{Y}_2 \krp \Matrix{Z}_1\Matrix{Z}_2.
\end{equation*}

For a set $G$, let $G^\ast = \bigcup_{n \in \N} G^{n}$ be the set of $G$-valued finite-length sequences.
For a function $f: G \to H$, we denote by $f^\ast : G^\ast \to H^\ast$ its extension to sequences that acts entrywise as
\begin{equation*}
    f^\ast : (g_1, \ldots, g_n) \mapsto (f(g_1), \ldots, f(g_n)).
\end{equation*}
There are two main applications of such extensions for our purposes.
First, when $G = H = \Real$ and $\Matrix{x} \in \Real^{d}$, we can apply $f$ to $\Matrix{x}$ entrywise as $f^\ast(\Matrix{x}) \in \Real^{d}$.
Next, if $G = \Real^d$ and $H = \Real^D$, then the isomorphism $\Real^{d \times n} \cong (\Real^d)^n$ allows us to apply $f$ to $\Matrix{X} \in \Real^{d \times n}$ columnwise as $f^\ast(\Matrix{X}) \in \Real^{D \times n}$.

\subsection{Transformer architectures}
\label{subsec:transformer_architecture}
The fundamental task performed by generative LLMs is the prediction of the next \emph{token} in a sequence.
A token is an element of a finite vocabulary set $\Omega \cong \{ 1, \ldots, |\Omega| \}$, which is fixed for a given LLM.
In the notation introduced above, an LLM maps sequences from $\Omega^\ast$ to probability distributions on $\Omega$.

Denote by $N$ the length of the input sequence of tokens $(t_1, \ldots, t_N) \in \{ 1, \ldots, |\Omega| \}^N$.
First, the tokens are embedded into the vector space $\Real^{d}$ via a look-up in the \emph{embedding} matrix $\Matrix{W}_{E} \in \Real^{d \times |\Omega|}$, that is, via a matrix product $\Matrix{X} = \Matrix{W}_{E}~[\Matrix{e}_{t_1}~\ldots~\Matrix{e}_{t_N}] \in \Real^{d \times N}$.
Then, the embedded vectors pass through a transformer $\Layer{T} : (\Real^d)^\ast \to (\Real^d)^\ast$ to become $\Layer{T}(\Matrix{X}) \in \Real^{d \times N}$. 
Next, these vectors return to their original dimensions using a \emph{language-model} matrix $\Matrix{W}_{LM} \in \Real^{|\Omega| \times d}$, i.e, $\Matrix{W}_{LM} \Layer{T}(\Matrix{X}) \in \Real^{|\Omega| \times N}$.
Finally, the \emph{softmax} function
\begin{equation}
\label{eq:softmax}
    \Layer{S} : \Real^\ast \to \Real^\ast, \quad 
    \Layer{S}(\Matrix{y}) = \frac{1}{\sum_{i = 1}^n \exp(y_i)} \begin{bmatrix}
        \exp(y_1) & \cdots & \exp(y_n)
    \end{bmatrix}^\trans, \quad \Matrix{y} \in \Real^n,
\end{equation}
is applied columnwise to produce the output matrix $\Matrix{P} = \Layer{S}^\ast(\Matrix{W}_{LM} \Layer{T}(\Matrix{X})) \in \Real^{|\Omega| \times N}$.
Every column of $\Matrix{P}$ is a probability distribution on $\Omega$, and when the transformer $\Layer{T}$ is \emph{decoder-only}---the case considered in our paper---the $j$th output column $\Matrix{p}_j$ depends only on the first $j$ input columns $\Matrix{x}_1, \ldots, \Matrix{x}_j$ and corresponds to the prediction of the $(j+1)$th token.
For the specific task of inference, or next-token generation, only the last output column $\Matrix{p}_N$ is used.

We focus our analysis exclusively on the transformer $\Layer{T}$ itself.
A decoder-only sequential transformer of depth $L$ with pre-normalisation is a deep composition of residual blocks
\begin{equation}
\label{eq:transformer-LN}
    \Layer{T} = (\Id + \Layer{F}_L^\ast \circ \Layer{LN}_{L,2}^\ast) \circ (\Id + \Layer{A}_L \circ \Layer{LN}_{L,1}^\ast) \circ \cdots \circ (\Id + \Layer{F}_1^\ast \circ \Layer{LN}_{1,2}^\ast) \circ (\Id + \Layer{A}_1 \circ \Layer{LN}_{1,1}^\ast),
\end{equation}
where $\Id$ is the identity map.
Layer normalisation $\Layer{LN}_{l,i} : \Real^d \to \Real^d$, FF mechanism $\Layer{F}_l : \Real^d \to \Real^d$, and self-attention mechanism $\Layer{A}_l : (\Real^d)^\ast \to (\Real^d)^\ast$ will be defined below.
We demonstrate the practical relevance of the architecture \eqref{eq:transformer-LN} by comparing it with industrial LLMs in Appendix~\ref{appendix:llm_comparison}.

\emph{Layer normalisation} \citep{ba2016layer} is traditionally defined with a positive stabilisation parameter:
\begin{equation}
\label{eq:layernorm}
    \Layer{LN}(\Matrix{x}) = \frac{\Matrix{g} \hap (\Matrix{I}_d - \tfrac{1}{d} \One \One^\trans)\Matrix{x}}{\sqrt{\Norm{(\Matrix{I}_d - \tfrac{1}{d} \One \One^\trans)\Matrix{x}}{2}^2 + \epsilon}}.
\end{equation}
It subtracts the sample mean from $\Matrix{x}$, normalises it to `unit' norm, and scales with an entrywise non-zero \emph{gain} $\Matrix{g} \in \Real^d$, whose `default' value is $\sqrt{d} \One$.
Omitting the subtraction of mean, we obtain a function
\begin{equation}
\label{eq:rms-layernorm}
    \Layer{RN}(\Matrix{x}) = \frac{\Matrix{g} \hap \Matrix{x}}{\sqrt{\Norm{\Matrix{x}}{2}^2 + \epsilon}}
\end{equation}
called \emph{root-mean-square} layer normalisation \citep{zhang2019root}.
In defining the two functions, we have omitted an additive bias term, which can be learned together with the gain.
As Table~\ref{tab:LLM-LN} shows, the choice of normalisation and the inclusion of bias vary across generations of LLMs.
Following the more recent trend, we omit bias and, in addition to \eqref{eq:transformer-LN}, consider an architecture with $\Layer{RN}$:
\begin{equation}
\label{eq:transformer-RN}
    \Layer{T} = (\Id + \Layer{F}_L^\ast \circ \Layer{RN}_{L,2}^\ast) \circ (\Id + \Layer{A}_L \circ \Layer{RN}_{L,1}^\ast) \circ \cdots \circ (\Id + \Layer{F}_1^\ast \circ \Layer{RN}_{1,2}^\ast) \circ (\Id + \Layer{A}_1 \circ \Layer{RN}_{1,1}^\ast).
\end{equation}

Transformer architectures also differ in where the normalisation is placed within each residual block \citep{xiong2020layer}.
The original approach proposed in \cite{vaswani2017attention} is \emph{post-normalisation}: the transformer blocks are organised as $(\Layer{LN}^\ast \circ (\Id + \Layer{F}^\ast)) \circ (\Layer{LN}^\ast \circ (\Id + \Layer{A}))$ with normalisation applied after the residual connection.
\emph{Pre-normalisation} is predominantly used in recent LLMs, and we adopt it in the architectures \eqref{eq:transformer-LN} and \eqref{eq:transformer-RN}.
A different approach was taken in \cite{walsh20252}, where the outputs of residual branches themselves are normalised before addition, i.e., $(\Id + \Layer{LN}^\ast \circ \Layer{F}^\ast) \circ (\Id + \Layer{LN}^\ast \circ \Layer{A})$, and \cite{riviere2024gemma} combines this with pre-normalisation.
See Table~\ref{tab:LLM-LN}.

Consider FF mechanisms next.
We use the standard two-layer mechanism with a componentwise activation function $\sigma : \Real \to \Real$.
Specifically, a two-layer mechanism is a composition
\begin{equation}
\label{eq:ffn-gelu}
    \Layer{F}(\Matrix{x}) = \Matrix{W}_{down} \sigma^\ast(\Matrix{W}_{up} \Matrix{x})
\end{equation}
with \emph{up-projection} and \emph{down-projection} matrices $\Matrix{W}_{up} \in \Real^{D \times d}$ and $\Matrix{W}_{down} \in \Real^{d \times D}$ ($D \geq d$).
In general, up-projection and down-projection can be affine operators with an additive bias term, though we omit it in accordance with recent trends (Table~\ref{tab:LLM-FFN}).
The typical choices for the activation function are ReLU, $\sigma(x) = \max\{ 0, x \}$, and GELU, $\sigma(x) = x \Phi(x)$, defined via the standard Gaussian cumulative distribution function $\Phi$ \citep{hendrycks2016gaussian}.
We will use smooth activations in the analysis (e.g., GELU).

Recent models adopt a \emph{gated} FF mechanism that consists of \emph{three} linear layers.
It is defined as
\begin{equation}
\label{eq:ffn-swiglu}
    \Layer{FG}(\Matrix{x}) = \Matrix{W}_{down} \Big( \sigma^\ast(\Matrix{W}_{gate}\Matrix{x}) \hap \Matrix{W}_{up} \Matrix{x} \Big),
\end{equation}
where we omit biases again and introduce the \emph{gate-projection} matrix $\Matrix{W}_{gate} \in \Real^{D \times d}$ \citep{shazeer2020glu}.
The gating function $\sigma$ is typically the \emph{Swish} function, $\sigma(x) = x / (1 + \exp(-\beta x))$.
Most modern LLMs select $\beta = 1$, reducing Swish to the \emph{SiLU} function and yielding the SwiGLU mechanism (Table~\ref{tab:LLM-FFN}).
We will analyse this prevailing architecture as a modification of \eqref{eq:transformer-LN} and \eqref{eq:transformer-RN}:
\begin{equation}
\label{eq:transformer-SwiGLU}
    \Layer{T} = (\Id + \Layer{FG}_L^\ast \circ \Layer{RN}_{L,2}^\ast) \circ (\Id + \Layer{A}_L \circ \Layer{RN}_{L,1}^\ast) \circ \cdots \circ (\Id + \Layer{FG}_1^\ast \circ \Layer{RN}_{1,2}^\ast) \circ (\Id + \Layer{A}_1 \circ \Layer{RN}_{1,1}^\ast).
\end{equation}

The two FF mechanisms described above are \emph{dense}, meaning the same weights are applied to every input.
In contrast, the sparse \emph{mixture-of-experts} maintains multiple dense FF mechanisms, or \emph{experts} \citep{shazeer2017outrageously}.
For each input, a gating mechanism selects a subset of experts and computes a weighted sum of their outputs.
In this work, we focus on architectures with dense FF mechanisms.

\emph{Causal multi-head self-attention} \citep{vaswani2017attention} is the only part of a transformer that `mixes' tokens; specifically, the $j$th column of $\Layer{A}(\Matrix{X})$ is determined by the first $j$ columns of $\Matrix{X} \in \Real^{d \times N}$.
Attention depends on three weight matrices $\Matrix{W}_{Q}, \Matrix{W}_{K}, \Matrix{W}_{V} \in \Real^{d \times d}$, which are used to compute the \emph{query} matrix $\Matrix{Q} = \Matrix{W}_{Q} \Matrix{X}$, the \emph{key} matrix  $\Matrix{K} = \Matrix{W}_{K} \Matrix{X}$, and the \emph{value} matrix $\Matrix{V} = \Matrix{W}_{V} \Matrix{X}$.
These are then split into blocks of $d_{head}$ rows across $n_{head}$ attention \emph{heads} with $d = n_{head} d_{head}$ and, e.g., $\Matrix{Q}_h \in \Real^{d_{head} \times N}$ for the $h$th head.
Attention heads compute $\Matrix{A}_h = \Layer{AH}(\Matrix{Q}_h, \Matrix{K}_h, \Matrix{V}_h) \in \Real^{d_{head} \times N}$ column by column according to
\begin{equation}
\label{eq:sh-attention}
    \Matrix{a}_{h, n} = \Matrix{V}_{h, [1,n]} \Layer{S}\bigg( \frac{(\Matrix{K}_{h, [1,n]})^\trans \Matrix{q}_{h,n}}{\sqrt{d_{head}}} \bigg) \in \Real^{d_{head}}, \quad 1 \leq n \leq N,
\end{equation}
where $\Matrix{V}_{h, [1,n]}, \Matrix{K}_{h, [1,n]} \in \Real^{d_{head} \times n}$ denotes the `causal' restriction of values and keys.
Finally, individual attention heads are combined via the \emph{output} weight matrix $\Matrix{W}_O \in \Real^{d \times d}$ to yield
\begin{equation}
\label{eq:mh-attention}
    \Layer{A}(\Matrix{X}) = \Matrix{W}_O \begin{bmatrix}
        \Matrix{A}_1^\trans & \cdots & \Matrix{A}_{n_{head}}^\trans
    \end{bmatrix}^\trans \in \Real^{d \times N}.
\end{equation}
The linear operators induced by the weights can be extended to affine operators by adding bias vectors. However, Table~\ref{tab:LLM-ATTN} shows that many modern LLMs omit them, and so do we to preserve clarity.

As described above, each input column $\Matrix{x}_n$ is transformed into key and query vectors in the same way.
\emph{Positional encodings} (PE) add a dependence on $n$ to these transformations.
\emph{Absolute} PE were used in earlier models and amount to adding a vector before weight multiplication, e.g., $\Matrix{q}_{n} = \Matrix{W}_{Q} (\Matrix{x}_n + \Matrix{p}_n)$.
The ALiBi method introduced in \cite{press2021train} injects positional information into the key-query product $(\Matrix{K}_{h, [1,n]})^\trans \Matrix{q}_{h,n}$ by adding a fixed vector to it.
The standard approach adopted in modern LLMs, as Table~\ref{tab:LLM-ATTN} demonstrates, is \emph{rotary} PE (RoPE, \cite{su2024roformer}). 
RoPE relies on a pre-defined, fixed multiplicative group of rotation matrices $\{ \Matrix{R}_n \}_{n \in \Z}$ and applies them to both keys and queries:
\begin{equation*}
    \Matrix{q}_{n} = \Matrix{R}_n \Matrix{W}_{Q} \Matrix{x}_n, \quad \Matrix{k}_{n} = \Matrix{R}_n \Matrix{W}_{K} \Matrix{x}_n.
\end{equation*}
Following the modern convention, we will study the effects of RoPE on our bounds.

Another major trend in recent models is the use of \emph{grouped-query} attention \citep{shazeer2019fast, ainslie2023gqa}.
This modification involves reusing the same key and value matrices, $\Matrix{K}_h$ and $\Matrix{V}_h$, across multiple attention heads to reduce runtime memory requirements.
For the purposes of our analysis, grouped-query attention requires no separate treatment, and hence we omit it without loss of generality.
The \emph{latent} attention of \cite{liu2024deepseek} relies on a different technique to compress keys and values, which we do not take into consideration to keep the analysis more focused.
For the same reason, we do not explicitly study \emph{sliding-window} attention \citep{jiang2023mistral, jiang2024mixtral, riviere2024gemma}.

\subsection{Condition numbers}
\label{subsec:condition_numbers}
Let us recall the basics of differentiable functions and their sensitivity to perturbations of the arguments.
Consider normed linear spaces $V, Z$ over $\Real$ and a function $f : V \to Z$ defined on an open domain.
The function $f$ is said to be \emph{Fr\'echet differentiable} at $v \in V$ if there exists a continuous linear operator $\Layer{D}_f(v)$ from $V$ to $Z$ such that
\begin{equation*}
    \lim_{\Norm{\delta}{V} \to 0} \frac{\Norm{f(v + \delta) - f(v) - \Layer{D}_f(v)\delta}{Z}}{\Norm{\delta}{V}} = 0.
\end{equation*}
The operator $\Layer{D}_f(v)$ is unique, if exists, and is called the \emph{Fr\'echet derivative} of $f$ at $v$.
When $V = \Real^n$ and $Z = \Real^m$, the Fr\'echet derivative is an $m \times n$ matrix of partial derivatives.
In what follows, we consider only finite-dimensional $V$ and $Z$.

In the finite-dimensional setting, the definition above is independent of the choice of norms.
Given bases $v_1, \ldots, v_{\dim V} \in V$ and $z_1, \ldots, z_{\dim Z} \in Z$, we can represent the Fr\'echet derivative as a $\dim Z \times \dim V$ matrix.
Denote by $\Layer{C}_V : V \to \Real^{\dim V}$ the linear operator that evaluates the expansion coefficients of its argument in the fixed basis and consider a function $g : \Real^{\dim V} \to \Real^{\dim Z}$ defined by $g(\Matrix{x}) = \Layer{C}_Z f(\Layer{C}_V^{-1} \Matrix{x})$.
By the chain rule, its Fr\'echet derivative is a matrix $\Layer{D}_g(\Matrix{x}) = \Layer{C}_Z \circ \Layer{D}_f(\Layer{C}_V^{-1} \Matrix{x}) \circ \Layer{C}_V^{-1}$.
This matrix is the \emph{Jacobian} of $f$ at $v = \Layer{C}_V^{-1} \Matrix{x}$ in the selected bases, and we denote it by $\Matrix{J}_f(v) \in \Real^{\dim Z \times \dim V}$.
The normed spaces we will encounter in this article are spaces of vectors and matrices, and we shall use the standard orthonormal bases $\Matrix{e}_1, \ldots, \Matrix{e}_n \in \Real^n$ and $\Matrix{E}_{1,1}, \ldots, \Matrix{E}_{m,n} \in \Real^{m \times n}$ for them.

Computations can become confusing when spaces of matrices are involved \citep{magnus2019matrix, magnus2024matrix}.
To circumvent the confusion, we define the \emph{vectorisation} of $\Matrix{X} \in \Real^{m \times n}$ as
\begin{equation*}
    \Vect{\Matrix{X}} = \begin{bmatrix}
        \Matrix{x}_1^\trans & \cdots & \Matrix{x}_n^\trans
    \end{bmatrix}^\trans \in \Real^{mn},
\end{equation*}
and the vectorisation of a matrix function $f : \Real^{m \times n} \to \Real^{p \times q}$ as
\begin{equation*}
    f_{\vect} : \Real^{mn} \to \Real^{pq}, \quad f_{\vect} = \vect \circ f \circ \vect^{-1}.
\end{equation*}
Note that $\Layer{C}_{\Real^n} \Matrix{x} = \Matrix{x}$ and $\Layer{C}_{\Real^{m \times n}} \Matrix{X} = \Vect{\Matrix{X}}$, so that the Jacobian of a matrix function $f$ is the Fr\'echet derivative of its vectorisation, or $\Matrix{J}_{f}(\Matrix{X}) = \Matrix{J}_{f_{\vect}}(\Vect{\Matrix{X}})$.\footnote{The matrix $\Matrix{J}_{f_{\vect}}(\Vect{\Matrix{X}})$ is also known as the Kronecker form of the Fr\'echet derivative $\Layer{D}_f(\Matrix{X})$ \cite[\S 3]{higham2008functions}.}
As a useful example, consider the GEMM function $f : \Real^{m \times k} \times \Real^{k \times n} \to \Real^{m \times n}$ defined by $f(\Matrix{X}, \Matrix{Y}) = \Matrix{X} \Matrix{Y}$.
Its Jacobian equals
\begin{equation*}
    \Matrix{J}_{f}(\Matrix{X}, \Matrix{Y}) = \begin{bmatrix}
        \Matrix{Y}^\trans \krp \Matrix{I}_m & \Matrix{I}_n \krp \Matrix{X}
    \end{bmatrix} \in \Real^{mn \times (mk + kn)}.
\end{equation*}

The Jacobian of a smooth function determines its local sensitivity to input perturbations \citep{rice1966theory}, and it is important to specify how the input and output perturbations are measured.
Let $v \in V$.
We define the \emph{relative normwise} and \emph{relative componentwise} distances from $\hat{v} \in V$ to $v$ as
\begin{equation*}
    \dist{V}{\hat{v}, v} = \frac{\Norm{\hat{v} - v}{V}}{\Norm{v}{V}}, \quad \dist{c}{\hat{v}, v} = \max_{1 \leq i \leq \dim V} \left| \frac{(\Layer{C}_V \hat{v})_i - (\Layer{C}_V v)_i}{(\Layer{C}_V v)_i} \right|,
\end{equation*}
where $x / 0 = \infty$ for $x \neq 0$ and $0 / 0 = 0$ by convention.
It follows from the definition that $\dist{c}{\hat{v}, v} < \infty$ if and only if $\supp{\Layer{C}_V \hat{v}} \subseteq \supp{\Layer{C}_V v}$, while $\dist{V}{\hat{v}, v} = \infty$ if and only if $v = 0$ and $\hat{v} \neq 0$.
(Relative) condition numbers reflect the sensitivity of relative output perturbations to relative input perturbations \citep{gohberg1993mixed}.
For a function $f : V \to Z$ and $v \in V$, the condition numbers of $f$ at $v$ are
\begin{equation*}
    \cond{\theta_{in},\theta_{out}}{f,v} = \lim_{\epsilon \to 0} \sup \set[\Bigg]{\frac{\dist{\theta_{out}}{f(\hat{v}), f(v)}}{\dist{\theta_{in}}{\hat{v}, v}}}{\dist{\theta_{in}}{\hat{v}, v} \leq \epsilon},
\end{equation*}
where $\theta_{\mathrm{in}}$ and $\theta_{\mathrm{out}}$ range through possible combinations of distance used to measure input and output perturbations.
Note that $\hat{v} = v$ is allowed in the definition above; together with the conventions regarding the division by zero, this lets us omit the standard assumptions of $v \neq 0$ and $f(v) \neq 0$ by allowing the condition numbers to be infinite.
We shall say that a condition number is \emph{normwise} (or \emph{componentwise}) if both input and output perturbations are measured with normwise (or componentwise) distances; all other combinations are referred to as \emph{mixed}.

When the function $f$ is Fr\'echet differentiable at $v$, its condition numbers are expressed explicitly in terms of the Jacobian.

\begin{proposition}
\label{proposition:condition_numbers}
Let $f$ be Fr\'echet differentiable at $v \in V$, and let $\Norm{v}{p} = \Norm{\Layer{C}_V v}{p}$ and $\Norm{z}{p} = \Norm{\Layer{C}_Z z}{p}$ be norms on $V$ and $Z$ for any $1 \leq p \leq \infty$.
Then the normwise condition number of $f$ is equal to
\begin{equation*}
    \cond{V,Z}{f,v} = \Norm{\Layer{D}_f(v)}{V,Z} \frac{\Norm{v}{V}}{\Norm{f(v)}{Z}},
\end{equation*}
its mixed condition numbers are equal to
\begin{equation*}
    \cond{p,c}{f,v} = \Norm{\diag{\Layer{C}_Z f(v)}^{-1} \Matrix{J}_f(v)}{p,\infty} \Norm{v}{p}, \quad
    \cond{c,p}{f,v} = \Norm{\Matrix{J}_f(v) \diag{\Layer{C}_V v}}{\infty,p} \frac{1}{\Norm{f(v)}{p}},
\end{equation*}
and its componentwise condition number equals 
\begin{equation*}
    \cond{c,c}{f,v} = \Norm{\diag{\Layer{C}_Z f(v)}^{-1} \Matrix{J}_f(v) \diag{\Layer{C}_V v}}{\infty,\infty}.
\end{equation*}
\end{proposition}
\begin{proof}
Follows from the definition in analogy with \cite{rice1966theory, gohberg1993mixed}.
\end{proof}

The following statement rigorously describes how the condition numbers reflect the sensitivity of a function to perturbations.
It requires the function to be more than just Fr\'echet differentiable \citep[\S4.2]{zeidler1995applied}, which holds for all functions considered in our work.

\begin{proposition}
\label{proposition:sensitivity_bound}
Let $f$ be twice continuously Fr\'echet differentiable on an open set $\hat{V} \subseteq V$ and let $v \in \hat{V}$.
Let $\theta_{\mathrm{in}}$ and $\theta_{\mathrm{out}}$ describe the distances used for input and output perturbations.
Let $t > 0$ and consider a distance ball $\mathbb{B}_{\theta_{\mathrm{in}}}(v, t) = \set{\hat{v} \in V}{\dist{\theta_{in}}{\hat{v}, v} \leq t}$.
For every $t$ such that $\mathbb{B}_{\theta_{\mathrm{in}}}(v, t) \subset \hat{V}$, there exists a constant $C_{\theta_{\mathrm{in}}, \theta_{\mathrm{out}}}(v, t) \geq 0$ such that
\begin{equation*}
    \dist{\theta_{out}}{f(\hat{v}), f(v)} \leq \cond{\theta_{in}, \theta_{out}}{f,v} \dist{\theta_{in}}{\hat{v},v} + C_{\theta_{\mathrm{in}}, \theta_{\mathrm{out}}}(v, t) \dist{\theta_{in}}{\hat{v},v}^2, \quad \hat{v} \in \mathbb{B}_{\theta_{\mathrm{in}}}(v, t).
\end{equation*}
If the condition number is finite then $C_{\theta_{\mathrm{in}}, \theta_{\mathrm{out}}}(v, t) < \infty$.
If $f$ is an affine function then $C_{\theta_{\mathrm{in}}, \theta_{\mathrm{out}}}(v, t) = 0$.
\end{proposition}
\begin{proof}
Follows from Taylor's theorem \citep[\S4.5]{zeidler1995applied} since $\mathbb{B}_{\theta_{\mathrm{in}}}(v, t)$ is compact and convex.
\end{proof}

\subsection{Rounding error analysis}
\label{subsec:rounding_error_analysis}
Here, we recall the basics of FP arithmetic and classical deterministic rounding error analysis \citep{higham2002accuracy, muller2018handbook}.
Consider an FP number system $\F$ with base 2 and $\mu \in \N$ mantissa bits; for example, this could be FP64 ($\mu = 52$), FP32 ($\mu = 23$), TF32 ($\mu = 10$), or BF16 ($\mu = 7$).
The \emph{unit round-off} of $\F$ is defined as $\uh = 2^{-\mu-1}$.
Any number $x \in \Real \setminus \F$ can be rounded to the nearest FP number,\footnote{In this work, we do not take into account the dynamic range of FP number systems and the issues related to overflow.} denoted by $\fl{x} \in \F$, which satisfies the fundamental rounding property \citep[Thm.~2.2]{higham2002accuracy}:
\begin{equation}
\label{eq:correct_rounding}
    \fl{x} = x(1 + \delta), \quad |\delta| \leq \uh.
\end{equation}

As is common in numerical analysis, we will assume that arithmetic operations and the elementary functions relevant to LLMs (namely, $\sqrt{\cdot}$ and $\exp$) are \emph{correctly rounded}; that is, their computed values satisfy the relative bound \eqref{eq:correct_rounding}.
Therefore, our rounding error analysis will rely on the following theoretical model of FP arithmetic, where $\fl{\cdot}$ denotes the value of a function evaluated according to a certain algorithm in FP arithmetic.

\begin{model}
\label{model:fp_arithmetic}
For every arithmetic operation $\mathrm{op} \in \{ +, -, \times, / \}$ and every elementary function $f$,
\begin{alignat*}{3}
    &\fl{x~\mathrm{op}~y} &&= (x~\mathrm{op}~y)(1 + \delta_1), &&\quad |\delta_1| \leq \uh, \\
    &\fl{f(x)} &&= f(x) (1 + \delta_2), &&\quad |\delta_2| \leq \uh.
\end{alignat*}
\end{model}

When several consecutive elementary functions are computed, the resulting relative error is typically described by a product $\prod_{k = 1}^{n} (1 + \delta_k)$, which can be simplified as in \cite[Lem.~3.1]{higham2002accuracy}.

\begin{lemma}
\label{lemma:theta_gamma}
Let $|\delta_k| \leq \uh$ and $\rho_k = \pm 1$ for $k = 1, \ldots, n$. If $n \uh < 1$ then
\begin{equation*}
    \prod_{k = 1}^{n} (1 + \delta_k)^{\rho_k} = 1 + \theta_n, \quad |\theta_n| \leq \gamma(n,\uh) = \frac{n\uh}{1 - n\uh} = n\uh + \Oh{\uh^2}.
\end{equation*}
\end{lemma}

The purpose of our analysis is to bound the \emph{forward} error.
The most well-studied algorithms in this regard concern summation and matrix products.
When based on \emph{recursive summation}, they satisfy
\begin{equation*}
    \Big|\fl[\Big]{\sum\nolimits_{i = 1}^d x_i} - \sum\nolimits_{i = 1}^d x_i\Big| \leq \gamma(d-1,\uh) \Norm{\Matrix{x}}{1}, \quad |\fl{\Matrix{W x}} - \Matrix{W x}| \leq \gamma(d,\uh) \Matrix{|W| |x|}, \quad \Matrix{x} \in \Real^{d}, \quad \Matrix{W} \in \Real^{D \times d}.
\end{equation*}
However, modern accelerators used in deep learning do not compute GEMMs via recursive summation.
To account for this discrepancy without sacrificing the clarity of presentation, we introduce a theoretical model of numerical GEMMs that abstracts away the specific algorithm.
In addition, the model below explicitly incorporates the quantisation of GEMM inputs (cf. \cite{blanchard2020mixed}).

\begin{model}
\label{model:fp_gemm}
Let $\Matrix{W} \in \Real^{D \times d}$ and $\Matrix{X} \in \Real^{d \times N}$, and let $\Matrix{Y} \in \Real^{D \times N}$ be stored in precision $\uh_{a}$. Then
\begin{equation*}
    |\fl{\Matrix{Y} + \Matrix{W X}} - (\Matrix{Y} + \Matrix{W X})| \leq \uh_{a} |\Matrix{Y}| + \Big( \gamma_{MM}(d, \uh_{a}) + \gamma(2, \uh_{q}) + \gamma_{MM}(d, \uh_{a})\gamma(2, \uh_{q}) \Big) \Matrix{|W| |X|}.
\end{equation*}
\end{model}

Model~\ref{model:fp_gemm} contains two precisions: the quantisation precision $\uh_{q}$ and the accumulation precision $\uh_{a}$.
Our analysis will also include a working precision $\uh_{w}$ used to compute nonlinear functions and a residual precision $\uh_{r}$ used to store the state variable $\Matrix{X}$ in the residual stream between blocks.
We assume that these precisions satisfy $\uh_{w} \leq \uh_{a} \leq \uh_{r} \ll \uh_{q}$ and will write $\gamma_{MM}(d) = \gamma_{MM}(d, \uh_a)$ for brevity.

Let us note that we assume \emph{fused accumulation} for residual updates $\Matrix{X} + f(\Matrix{X})$.
That is, the output of $f(\Matrix{X})$---which concludes with a GEMM for both FF and attention blocks---is not rounded to residual precision $\uh_{r}$ for intermediate storage.
Instead, it is kept in accumulation precision $\uh_{a}$ and immediately added to $\Matrix{X}$ (via a so-called GEMM epilogue).
Only then is $\Matrix{X} + f(\Matrix{X})$ rounded to residual precision $\uh_{r}$.
\section{Analysis of layer normalisation}
\label{sec:layernorm}

\subsection{Condition numbers: Derivation}
We begin by studying the condition numbers of the $\Layer{RN}$ normalisation function defined in \eqref{eq:rms-layernorm} as
\begin{equation*}
    \Layer{RN}(\Matrix{x}) = \frac{\Matrix{g} \hap \Matrix{x}}{\sqrt{\Norm{\Matrix{x}}{2}^2 + \epsilon}}.
\end{equation*}

\begin{theorem}
\label{theorem:cond-rms-layernorm}
Let $\Matrix{x} \in \Real^d$.
Normwise condition numbers of $\Layer{RN}$ satisfy
\begin{gather*}
    \cond{2,2}{\Layer{RN}, \Matrix{x}} \leq \frac{\Norm{\Matrix{g}}{\infty} \Norm{\Matrix{x}}{2}}{\Norm{\Matrix{g} \hap \Matrix{x}}{2}}, \quad
    \cond{2,\infty}{\Layer{RN}, \Matrix{x}} = \max_{1 \leq i \leq d} |g_i| \sqrt{1 - \bigg( 1 + \frac{\epsilon}{\Norm{\Matrix{x}}{2}^2 + \epsilon} \bigg) \frac{x_i^2}{\Norm{\Matrix{x}}{2}^2 + \epsilon}} \frac{\Norm{\Matrix{x}}{2}}{\Norm{\Matrix{g} \hap \Matrix{x}}{\infty}}, \\
    \cond{\infty,2}{\Layer{RN}, \Matrix{x}} \leq \max_{\Matrix{v} \in \{ -1,1 \}^d} \sqrt{d - \bigg( 1 + \frac{\epsilon}{\Norm{\Matrix{x}}{2}^2 + \epsilon} \bigg) \frac{(\Matrix{v}^\trans \Matrix{x})^2}{\Norm{\Matrix{x}}{2}^2 + \epsilon}} \frac{\Norm{\Matrix{g}}{\infty} \Norm{\Matrix{x}}{\infty}}{\Norm{\Matrix{g} \hap \Matrix{x}}{2}},\\
    \cond{\infty,\infty}{\Layer{RN}, \Matrix{x}} = \max_{1 \leq i \leq d} |g_i| \left( 1 + \frac{ |x_i| ( \Norm{\Matrix{x}}{1} - 2 |x_i| )}{\Norm{\Matrix{x}}{2}^2 + \epsilon} \right) \frac{\Norm{\Matrix{x}}{\infty}}{\Norm{\Matrix{g} \hap \Matrix{x}}{\infty}},
\end{gather*}
its mixed condition numbers satisfy
\begin{gather*}
    \cond{c,\infty}{\Layer{RN}, \Matrix{x}} = 2 \max_{1 \leq i \leq d} |g_i x_i| \bigg( 1 - \frac{x_i^2 + \epsilon/2}{\Norm{\Matrix{x}}{2}^2 + \epsilon} \bigg) \frac{1}{\Norm{\Matrix{g} \hap \Matrix{x}}{\infty}}, \\
    \cond{\infty,c}{\Layer{RN}, \Matrix{x}} = \left( \frac{1}{\Norm{\Matrix{x}}{-\infty}} + \frac{ \Norm{\Matrix{x}}{1} - 2 \Norm{\Matrix{x}}{-\infty} }{\Norm{\Matrix{x}}{2}^2 + \epsilon} \right) \Norm{\Matrix{x}}{\infty}, \\
    \cond{c,2}{\Layer{RN}, \Matrix{x}} \leq \max_{\Matrix{v} \in \{ -1,1 \}^d} \sqrt{\Norm{\Matrix{x}}{2}^2 - \bigg( 1 + \frac{\epsilon}{\Norm{\Matrix{x}}{2}^2 + \epsilon} \bigg) \frac{(\Matrix{v}^\trans (\Matrix{x} \hap \Matrix{x}))^2}{\Norm{\Matrix{x}}{2}^2 + \epsilon}} \frac{\Norm{\Matrix{g}}{\infty}}{\Norm{\Matrix{g} \hap \Matrix{x}}{2}}, \\
    \cond{2,c}{\Layer{RN}, \Matrix{x}} = \sqrt{\frac{1}{\Norm{\Matrix{x}}{-\infty}^2} - \bigg( 1 + \frac{\epsilon}{\Norm{\Matrix{x}}{2}^2 + \epsilon} \bigg) \frac{1}{\Norm{\Matrix{x}}{2}^2 + \epsilon}} \Norm{\Matrix{x}}{2},
\end{gather*}
and its componentwise condition number equals
\begin{equation*}
    \cond{c,c}{\Layer{RN}, \Matrix{x}} = 2 \bigg( 1 - \frac{\Norm{\Matrix{x}}{-\infty}^2 + \epsilon/2}{\Norm{\Matrix{x}}{2}^2 + \epsilon} \bigg).
\end{equation*}
The inequalities become equalities when $\Matrix{g} = g \One$.
\end{theorem}
\begin{proof}
Taking the partial derivatives of $\Layer{RN}$, we get the Jacobian and denote
\begin{equation*}
    \Matrix{J}_{\Layer{RN}}(\Matrix{x}) = \frac{\diag{\Matrix{g}}}{\sqrt{\Norm{\Matrix{x}}{2}^2 + \epsilon}} \bigg( \Matrix{I}_d - \frac{\Matrix{x} \Matrix{x}^\trans}{\Norm{\Matrix{x}}{2}^2 + \epsilon} \bigg) \in \Real^{d \times d}, \quad \Matrix{K} = \Matrix{J}_{\Layer{RN}}(\Matrix{x}) \sqrt{\Norm{\Matrix{x}}{2}^2 + \epsilon}.
\end{equation*}

\emph{(Normwise)}
To bound $\cond{2,2}{\Layer{RN}, \Matrix{x}}$ via Proposition~\ref{proposition:condition_numbers}, we bound the spectral norm of
\begin{equation*}
    \Norm{\Matrix{K}}{2,2} \leq \Norm{\Matrix{g}}{\infty} \Norm[\bigg]{\Matrix{I}_d - \frac{\Matrix{x} \Matrix{x}^\trans}{\Norm{\Matrix{x}}{2}^2 + \epsilon}}{2,2} = \Norm{\Matrix{g}}{\infty}.
\end{equation*}
To compute $\cond{\infty,\infty}{\Layer{RN}, \Matrix{x}}$, we consider
\begin{align*}
    \Norm{\Matrix{K}}{\infty,\infty} = \max_{i} |g_i| \sum_{j} \bigg|\delta_{i,j} - \frac{x_i x_j}{\Norm{\Matrix{x}}{2}^2 + \epsilon}\bigg| &= \max_{i} |g_i| \left[ \bigg( 1 - \frac{x_i^2}{\Norm{\Matrix{x}}{2}^2 + \epsilon} \bigg) + \sum_{j \neq i} \frac{|x_i x_j|}{\Norm{\Matrix{x}}{2}^2 + \epsilon} \right] \\
    &= \max_{i} |g_i| \left[ \bigg( 1 - \frac{x_i^2}{\Norm{\Matrix{x}}{2}^2 + \epsilon} \bigg) + \frac{|x_i| (\Norm{\Matrix{x}}{1} - |x_i|)}{\Norm{\Matrix{x}}{2}^2 + \epsilon} \right].
\end{align*}
Similarly for $\cond{2,\infty}{\Layer{RN}, \Matrix{x}}$, we have
\begin{align*}
    \Norm{\Matrix{K}}{2,\infty} &= \max_{i} |g_i| \sqrt{\sum_{j} \bigg|\delta_{i,j} - \frac{x_i x_j}{\Norm{\Matrix{x}}{2}^2 + \epsilon}\bigg|^2} = \max_{i} |g_i| \sqrt{\bigg( 1 - \frac{x_i^2}{\Norm{\Matrix{x}}{2}^2 + \epsilon} \bigg) + \frac{x_i^2}{\Norm{\Matrix{x}}{2}^2 + \epsilon} \frac{\Norm{\Matrix{x}}{2}^2 - |x_i|^2}{\Norm{\Matrix{x}}{2}^2 + \epsilon}},
\end{align*}
and the following yields the bound on $\cond{\infty,2}{\Layer{RN}, \Matrix{x}}$:
\begin{align*}
    \Norm{\Matrix{K}}{\infty,2} = \max_{\Matrix{v} \in \{ -1,1 \}^d} \Norm{\Matrix{K} \Matrix{v}}{2} &\leq \Norm{\Matrix{g}}{\infty} \max_{\Matrix{v}} \Norm[\bigg]{\Matrix{v} - \Matrix{x} \frac{\Matrix{x}^\trans \Matrix{v}}{\Norm{\Matrix{x}}{2}^2 + \epsilon}}{2} \\
    &= \Norm{\Matrix{g}}{\infty} \max_{\Matrix{v}} \sqrt{\Norm{\Matrix{v}}{2}^2 - 2 \frac{(\Matrix{x}^\trans \Matrix{v})^2}{\Norm{\Matrix{x}}{2}^2 + \epsilon} + \frac{\Norm{\Matrix{x}}{2}^2 (\Matrix{x}^\trans \Matrix{v})^2}{(\Norm{\Matrix{x}}{2}^2 + \epsilon)^2}}.
\end{align*}

\emph{(Mixed)}
The expression for $\cond{\infty,c}{\Layer{RN}, \Matrix{x}}$ is obtained with the same argument as $\cond{\infty,\infty}{\Layer{RN}, \Matrix{x}}$.
To compute $\cond{c,\infty}{\Layer{RN}, \Matrix{x}}$, consider
\begin{align*}
    \Norm{\Matrix{K}\diag{\Matrix{x}}}{\infty,\infty} &= \max_{i} |g_i| \sum_{j} \bigg|\delta_{i,j} - \frac{x_i x_j}{\Norm{\Matrix{x}}{2}^2 + \epsilon}\bigg| |x_j| = \max_{i} |g_i x_i| \left[ \bigg( 1 - \frac{x_i^2}{\Norm{\Matrix{x}}{2}^2 + \epsilon} \bigg) + \frac{\Norm{\Matrix{x}}{2}^2 - x_i^2}{\Norm{\Matrix{x}}{2}^2 + \epsilon} \right].
\end{align*}
The bound on $\cond{c,2}{\Layer{RN}, \Matrix{x}}$ is derived along the same lines as in the case of $\cond{\infty,2}{\Layer{RN}, \Matrix{x}}$, and likewise $\cond{2,c}{\Layer{RN}, \Matrix{x}}$ is akin to $\cond{2,\infty}{\Layer{RN}, \Matrix{x}}$.

\emph{(Componentwise)}
The derivation of $\cond{c,c}{\Layer{RN}, \Matrix{x}}$ is similar to the case of $\cond{\infty,c}{\Layer{RN}, \Matrix{x}}$.
\end{proof}

Next, we bound the condition numbers of the $\Layer{LN}$ normalisation function defined in \eqref{eq:layernorm} as
\begin{equation*}
    \Layer{LN}(\Matrix{x}) = \Layer{RN}(\Matrix{y}), \quad \Matrix{y} = \left(\Matrix{I}_d - \tfrac{1}{d} \One \One^\trans \right) \Matrix{x}.
\end{equation*}

\begin{theorem}
\label{theorem:cond-layernorm}
Let $\Matrix{x} \in \Real^d$.
Normwise condition numbers of $\Layer{LN}$ satisfy
\begin{gather*}
    \cond{2,2}{\Layer{LN}, \Matrix{x}} \leq \frac{\Norm{\Matrix{g}}{\infty} \Norm{\Matrix{x}}{2}}{\Norm{\Matrix{g} \hap \Matrix{y}}{2}}, \quad
    \cond{2,\infty}{\Layer{LN}, \Matrix{x}} = \max_{1 \leq i \leq d} |g_i| \sqrt{1 - \frac{1}{d} - \bigg( 1 + \frac{\epsilon}{\Norm{\Matrix{y}}{2}^2 + \epsilon} \bigg) \frac{y_i^2}{\Norm{\Matrix{y}}{2}^2 + \epsilon}} \frac{\Norm{\Matrix{x}}{2}}{\Norm{\Matrix{g} \hap \Matrix{y}}{\infty}}, \\
    \cond{\infty,2}{\Layer{LN}, \Matrix{x}} \leq \max_{\Matrix{v} \in \{ -1,1 \}^d} \sqrt{d - \frac{(\Matrix{v}^\trans \One)^2}{d} - \bigg( 1 + \frac{\epsilon}{\Norm{\Matrix{y}}{2}^2 + \epsilon} \bigg) \frac{(\Matrix{v}^\trans \Matrix{y})^2}{\Norm{\Matrix{y}}{2}^2 + \epsilon}} \frac{\Norm{\Matrix{g}}{\infty} \Norm{\Matrix{x}}{\infty}}{\Norm{\Matrix{g} \hap \Matrix{y}}{2}}, \\
    \cond{\infty,\infty}{\Layer{LN}, \Matrix{x}} = \max_{1 \leq i \leq d} |g_i| \left(  1 - \frac{1}{d} - \frac{y_i^2}{\Norm{\Matrix{y}}{2}^2 + \epsilon} + \sum_{j \neq i} \bigg|\frac{1}{d} + \frac{y_i y_j}{\Norm{\Matrix{y}}{2}^2 + \epsilon} \bigg| \right) \frac{\Norm{\Matrix{x}}{\infty}}{\Norm{\Matrix{g} \hap \Matrix{y}}{\infty}},
\end{gather*}
its mixed condition numbers satisfy
\begin{gather*}
    \cond{c,\infty}{\Layer{LN}, \Matrix{x}} = \max_{1 \leq i \leq d} |g_i| \left(  \bigg[1 - \frac{1}{d} - \frac{y_i^2}{\Norm{\Matrix{y}}{2}^2 + \epsilon} \bigg] |x_i| + \sum_{j \neq i} \bigg|\frac{1}{d} + \frac{y_i y_j}{\Norm{\Matrix{y}}{2}^2 + \epsilon} \bigg| |x_j| \right) \frac{1}{\Norm{\Matrix{g} \hap \Matrix{y}}{\infty}}, \\
    \cond{\infty,c}{\Layer{LN}, \Matrix{x}} = \max_{1 \leq i \leq d} \left( \frac{d-1}{d |y_i|} - \frac{|y_i|}{\Norm{\Matrix{y}}{2}^2 + \epsilon} + \sum_{j \neq i} \bigg|\frac{1}{d y_i} + \frac{y_j}{\Norm{\Matrix{y}}{2}^2 + \epsilon} \bigg| \right) \Norm{\Matrix{x}}{\infty}, \\
    \cond{c,2}{\Layer{LN}, \Matrix{x}} \leq \max_{\Matrix{v} \in \{ -1,1 \}^d} \sqrt{\Norm{\Matrix{x}}{2}^2 - \frac{(\Matrix{v}^\trans \Matrix{x})^2}{d} - \bigg( 1 + \frac{\epsilon}{\Norm{\Matrix{y}}{2}^2 + \epsilon} \bigg) \frac{(\Matrix{v}^\trans (\Matrix{y} \hap \Matrix{x}))^2}{\Norm{\Matrix{y}}{2}^2 + \epsilon}} \frac{\Norm{\Matrix{g}}{\infty}}{\Norm{\Matrix{g} \hap \Matrix{y}}{2}}, \\
    \cond{2,c}{\Layer{LN}, \Matrix{x}} = \sqrt{\frac{d-1}{d \Norm{\Matrix{y}}{-\infty}^2} - \bigg( 1 + \frac{\epsilon}{\Norm{\Matrix{y}}{2}^2 + \epsilon} \bigg) \frac{1}{\Norm{\Matrix{y}}{2}^2 + \epsilon}} \Norm{\Matrix{x}}{2},
\end{gather*}
and its componentwise condition number equals
\begin{equation*}
    \cond{c,c}{\Layer{LN}, \Matrix{x}} = \max_{1 \leq i \leq d} \left(  \bigg[\frac{d-1}{d |y_i|} - \frac{|y_i|}{\Norm{\Matrix{y}}{2}^2 + \epsilon} \bigg] |x_i| + \sum_{j \neq i} \bigg|\frac{1}{d y_i} + \frac{y_j}{\Norm{\Matrix{y}}{2}^2 + \epsilon} \bigg| |x_j| \right).
\end{equation*}
The inequalities become equalities when $\Matrix{g} = g \One$.
\end{theorem}
\begin{proof}
Using the orthogonality $\One^\trans \Matrix{y} = 0$, we compute the Jacobian
\begin{equation*}
    \Matrix{J}_{\Layer{LN}}(\Matrix{x}) = \frac{\diag{\Matrix{g}}}{\sqrt{\Norm{\Matrix{y}}{2}^2 + \epsilon}} \bigg( \Matrix{I}_d - \frac{\Matrix{y} \Matrix{y}^\trans}{\Norm{\Matrix{y}}{2}^2 + \epsilon} - \frac{\One \One^\trans}{d} \bigg) \in \Real^{d \times d}, \quad \Matrix{K} = \Matrix{J}_{\Layer{LN}}(\Matrix{x}) \sqrt{\Norm{\Matrix{y}}{2}^2 + \epsilon}.
\end{equation*}

\emph{(Normwise)}
The bound on $\cond{2,2}{\Layer{LN}, \Matrix{x}}$ is derived as in Theorem~\ref{theorem:cond-rms-layernorm}.
To compute $\cond{\infty,\infty}{\Layer{LN}, \Matrix{x}}$ via Proposition~\ref{proposition:condition_numbers}, we consider
\begin{align*}
    \Norm{\Matrix{K}}{\infty,\infty} &= \max_{i} |g_i| \sum_{j} \bigg|\delta_{i,j} - \frac{y_i y_j}{\Norm{\Matrix{y}}{2}^2 + \epsilon} - \frac{1}{d} \bigg| = \max_{i} |g_i| \left( \left| 1 - \frac{1}{d} - \frac{y_i^2}{\Norm{\Matrix{y}}{2}^2 + \epsilon} \right| + \sum_{j \neq i} \bigg|\frac{1}{d} + \frac{y_i y_j}{\Norm{\Matrix{y}}{2}^2 + \epsilon} \bigg| \right)
\end{align*}
and observe that $\Matrix{y} \Matrix{y}^\trans / \Norm{\Matrix{y}}{2}^2 + \One \One^\trans / d$ is an orthogonal projection matrix, hence its diagonal entries lie between zero and one.
To derive $\cond{2,\infty}{\Layer{LN}, \Matrix{x}}$, it suffices to collect the terms in
\begin{align*}
    \Norm{\Matrix{K}}{2,\infty} = \max_{i} |g_i| \sqrt{\left( 1 - \frac{1}{d} - \frac{y_i^2}{\Norm{\Matrix{y}}{2}^2 + \epsilon} \right)^2 + \sum_{j \neq i} \bigg(\frac{1}{d} + \frac{y_i y_j}{\Norm{\Matrix{y}}{2}^2 + \epsilon} \bigg)^2}.
\end{align*}
The bound on $\cond{\infty,2}{\Layer{LN}, \Matrix{x}}$ is a result of the following, where we rely on the orthogonality $\One^\trans \Matrix{y} = 0$:
\begin{align*}
    \Norm{\Matrix{K}}{\infty,2} = \max_{\Matrix{v} \in \{ -1,1 \}^d} \Norm{\Matrix{K} \Matrix{v}}{2} &\leq \Norm{\Matrix{g}}{\infty} \max_{\Matrix{v}} \Norm[\bigg]{\Matrix{v} - \Matrix{y} \frac{\Matrix{y}^\trans \Matrix{v}}{\Norm{\Matrix{y}}{2}^2 + \epsilon} - \One \frac{\One^\trans \Matrix{v}}{d}}{2} \\
    &= \Norm{\Matrix{g}}{\infty} \max_{\Matrix{v}} \sqrt{\Norm{\Matrix{v}}{2}^2 - 2 \frac{(\Matrix{y}^\trans \Matrix{v})^2}{\Norm{\Matrix{y}}{2}^2 + \epsilon} + \frac{\Norm{\Matrix{y}}{2}^2 (\Matrix{y}^\trans \Matrix{v})^2}{(\Norm{\Matrix{y}}{2}^2 + \epsilon)^2} - \frac{(\One^\trans \Matrix{v})^2}{d}}.
\end{align*}

\emph{(Mixed and componentiwse)}
The derivation of the remaining condition numbers is analogous.
\end{proof}

\subsection{Condition numbers: Comparison}
\label{subsec:layernorm_comparison}
Let us compare the conditioning of the normalisation functions $\Layer{RN}$ and $\Layer{LN}$.
Mean subtraction can have a significant impact on the numerical stability of normalisation, which we demonstrate with a practically important example of a vector with a massive outlier---the presence of such outliers is characteristic for the inference of modern LLMs.
We assume that $\Matrix{g} = g\One$ and $\epsilon < 1/d$ and focus on four specific condition numbers.
Intermediate steps required to obtain the formulas are provided in Appendix~\ref{appendix:layernorm-cond-comparison}.

\emph{Massive outlier with zero-variance background.}
Let $\Matrix{x} = [1~\alpha~\cdots~\alpha]^\trans$ with $|\alpha| < 1$.
We shall consider an `extreme' scenario with $|\alpha| < \epsilon/2$, a `realistic' scenario with $\alpha = \beta / d$ and moderate $|\beta| \geq 1$, and another `realistic' scenario with $\alpha = \gamma / \sqrt{d}$ and moderate $|\gamma| \geq 1$.
The condition numbers of $\Layer{RN}$ and $\Layer{LN}$ are compared in Table~\ref{tab:rn_ln_zero_variance}.
In the first `realistic' case, the two functions exhibit similar conditioning behaviour overall.
Meanwhile, in the second `realistic' case, $\Layer{LN}$ is strictly better (worse) conditioned than $\Layer{RN}$ when the output errors are measured in the normwise (resp., componentwise) sense.
Note also that normwise input perturbations of the background entries lead to a blow up of their componentwise output perturbations for both normalisation functions.

\begin{table}[t]
\caption{Condition numbers of the normalisation functions $\Layer{RN}$ \eqref{eq:rms-layernorm} and $\Layer{LN}$ \eqref{eq:layernorm} in the presence of a massive outlier with zero-variance background.
The row-maximum in the definition of condition numbers marked with $\dagger$ is achieved at the outlier index ($i = 1$).}
\label{tab:rn_ln_zero_variance}
\begin{tabular}{@{}llcc@{}}
\toprule
Type & Scenario & $\Layer{RN}$ & $\Layer{LN}$ \\
\midrule
$\cond{\infty,\infty}{f,\Matrix{x}}$
& Extreme
& $1 + |\alpha| + \Oh{\epsilon|\alpha| + d \alpha^2}$
& $2(1 + \alpha) + \Oh[\Big]{\frac{1}{d^2} + \alpha^2}$ \\
& Realistic $\Big(\tfrac{\beta}{d}\Big)$
& $|\beta| - \frac{|\beta| + \beta^2(|\beta| - 1)}{d} + \Oh[\Big]{\epsilon + \frac{1}{d^2}}^\dagger$
& $2\Big(1 + \frac{\beta}{d}\Big) + \Oh[\Big]{\frac{1}{d^2}}$ \\
& Realistic $\Big(\tfrac{\gamma}{\sqrt{d}}\Big)$
& $\frac{|\gamma|\sqrt{d} + \gamma^2}{1+\gamma^2} + \Oh[\Big]{\frac{1}{\sqrt{d}}}^\dagger$
& $2\Big(1 + \frac{\gamma}{\sqrt{d}}\Big) + \Oh[\Big]{\frac{1}{d}}$ \\
\midrule
$\cond{c,\infty}{f,\Matrix{x}}$
& Extreme
& $\epsilon + 2d\alpha^2 + \Oh{\epsilon^2}^\dagger$
& $\epsilon + \frac{\epsilon}{d} + \Oh{\epsilon^2}^\dagger$ \\
& Realistic $\Big(\tfrac{\beta}{d}\Big)$
& $\epsilon + \frac{2\beta^2}{d} + \Oh[\Big]{\frac{1}{d^2}}^\dagger$
& $\frac{2|\beta| + \epsilon}{d} + \Oh[\Big]{\frac{1}{d^2}}$ \\
& Realistic $\Big(\tfrac{\gamma}{\sqrt{d}}\Big)$
& $\frac{2\gamma^2}{1+\gamma^2} - \frac{2\gamma^2}{d(1+\gamma^2)^2} + \Oh[\Big]{\epsilon + \frac{1}{d^2}}^\dagger$
& $\frac{2|\gamma|}{\sqrt{d}} + \frac{2\gamma|\gamma|}{d} + \Oh[\Big]{\frac{1}{d^{3/2}}}$ \\
\midrule
$\cond{\infty,c}{f,\Matrix{x}}$
& Extreme
& $\frac{1}{|\alpha|} + 1 + \Oh{\epsilon + d|\alpha|}$
& $2(d-1)(1 + \alpha) + \Oh[\Big]{\frac{1}{d}}$ \\
& Realistic $\Big(\tfrac{\beta}{d}\Big)$
& $\frac{d}{|\beta|} + 1 + |\beta| + \Oh[\Big]{\frac{1}{d}}$
& $2(d-1+\beta) + \Oh[\Big]{\frac{1}{d}}$ \\
& Realistic $\Big(\tfrac{\gamma}{\sqrt{d}}\Big)$
& $\left(\frac{1}{|\gamma|} + \frac{|\gamma|}{1+\gamma^2}\right)\sqrt{d} + \frac{1}{1+\gamma^2} + \Oh[\Big]{\frac{1}{\sqrt{d}}}$
& $2(d - 1 + \gamma\sqrt{d} + \gamma^2) + \Oh[\Big]{\frac{1}{\sqrt{d}}}$ \\
\midrule
$\cond{c,c}{f,\Matrix{x}}$
& Extreme
& $2 - \epsilon + \Oh{\epsilon^2}$
& $\epsilon + 2d|\alpha| + \Oh[\Big]{|\alpha| + \frac{\epsilon}{d}}$ \\
& Realistic $\Big(\tfrac{\beta}{d}\Big)$
& $2 - \epsilon + \Oh[\Big]{\frac{1}{d^2}}$
& $2|\beta| + \frac{2|\beta|(\beta-1)}{d} + \Oh[\Big]{\epsilon + \frac{1}{d^2}}$ \\
& Realistic $\Big(\tfrac{\gamma}{\sqrt{d}}\Big)$
& $2 - \frac{2\gamma^2}{d(1+\gamma^2)} + \Oh[\Big]{\epsilon + \frac{1}{d^2}}$
& $2|\gamma|\sqrt{d} + 2\gamma|\gamma| + \Oh[\Big]{\frac{1}{\sqrt{d}}}$ \\
\midrule
\end{tabular}
\end{table}

\emph{Massive outlier with zero-mean background.}
Assume that $d$ is odd and let $\Matrix{x} = [1~\alpha~-\alpha~\cdots~\alpha~-\alpha]^\trans$ with $0 \leq \alpha < 1 - 2/d$.
We consider an `extreme' scenario with $\alpha < \epsilon/2$, a `realistic' scenario with $\alpha = \beta / d$ and moderate $\beta \geq 3$, and another `realistic' scenario with $\alpha = \gamma / \sqrt{d}$ and moderate $\gamma \geq 1$.
The condition numbers are presented in Table~\ref{tab:rn_ln_zero_mean}.
They show that the the conditioning of $\Layer{RN}$ and $\Layer{LN}$ is almost identical when the output errors are measured in the normwise sense.
At the same time, the condition numbers of $\Layer{RN}$ are smaller than those of $\Layer{LN}$ in the `realistic' cases when the output errors are measured componentwise.
Again, the mixed $\ell_\infty$-to-componentwise condition numbers of both functions blow up.

\begin{table}[t]
\caption{Condition numbers of the normalisation functions $\Layer{RN}$ \eqref{eq:rms-layernorm} and $\Layer{LN}$ \eqref{eq:layernorm} in the presence of a massive outlier with zero-mean background.
The row-maximum in the definition of condition numbers marked with $\dagger$ is achieved at the outlier index ($i = 1$).}
\label{tab:rn_ln_zero_mean}
\begin{tabular}{@{}llcc@{}}
\toprule
Type & Scenario & $\Layer{RN}$ & $\Layer{LN}$ \\
\midrule
$\cond{\infty,\infty}{f,\mathbf{x}}$
& Extreme
& $1 + \alpha + \Oh{\epsilon\alpha + d \alpha^2}$
& $2 - \frac{2}{d^2} + \Oh[\Big]{\frac{\alpha}{d} + \frac{1}{d^3}}$ \\
& Realistic $\Big(\tfrac{\beta}{d}\Big)$
& $\beta - \frac{\beta + \beta^2(\beta - 1)}{d} + \Oh[\Big]{\epsilon + \frac{1}{d^2}}^\dagger$
& $\beta - \frac{\beta^2(\beta-1)}{d} + \Oh[\Big]{\epsilon + \frac{1}{d^2}}^\dagger$ \\
& Realistic $\Big(\tfrac{\gamma}{\sqrt{d}}\Big)$
& $\frac{\gamma\sqrt{d} + \gamma^2}{1+\gamma^2} + \Oh[\Big]{\frac{1}{\sqrt{d}}}^\dagger$
& $\frac{\gamma\sqrt{d} + \gamma^2}{1+\gamma^2} + \Oh[\Big]{\frac{1}{\sqrt{d}}}^\dagger$ \\
\midrule
$\cond{c,\infty}{f,\mathbf{x}}$
& Extreme
& $\epsilon + 2d\alpha^2 + \Oh{\epsilon^2}^\dagger$
& $\epsilon + d\alpha^2 + \Oh[\Big]{\frac{\epsilon}{d}}^\dagger$ \\
& Realistic $\Big(\tfrac{\beta}{d}\Big)$
& $\epsilon + \frac{2\beta^2}{d} + \Oh[\Big]{\frac{1}{d^2}}^\dagger$
& $\epsilon + \frac{2\beta^2}{d} + \Oh[\Big]{\frac{1}{d^2}}^\dagger$ \\
& Realistic $\Big(\tfrac{\gamma}{\sqrt{d}}\Big)$
& $\frac{2\gamma^2}{1+\gamma^2} - \frac{2\gamma^2}{d(1+\gamma^2)^2} + \Oh[\Big]{\epsilon + \frac{1}{d^2}}^\dagger$
& $\frac{2\gamma^2}{1+\gamma^2} + \frac{1 - \gamma^2}{(1+\gamma^2)^2}\epsilon + \Oh[\Big]{\frac{\epsilon}{d}}^\dagger$ \\
\midrule
$\cond{\infty,c}{f,\mathbf{x}}$
& Extreme
& $\frac{1}{\alpha} + 1 + \Oh{\epsilon + d|\alpha|}$
& $\frac{2(d-1)}{1-d\alpha} + \Oh[\Big]{\frac{1}{d}}$ \\
& Realistic $\Big(\tfrac{\beta}{d}\Big)$
& $\frac{d}{\beta} + 1 + \beta + \Oh[\Big]{\frac{1}{d}}$
& $\frac{2d}{\beta-1} - \frac{2}{\beta-1} + \Oh[\Big]{\frac{1}{d}}$ \\
& Realistic $\Big(\tfrac{\gamma}{\sqrt{d}}\Big)$
& $\left(\frac{1}{\gamma} + \frac{\gamma}{1+\gamma^2}\right)\sqrt{d} + \frac{1}{1+\gamma^2} + \Oh[\Big]{\frac{1}{\sqrt{d}}}$
& $\frac{2\sqrt{d}}{\gamma} + \frac{2}{\gamma^2} + \Oh[\Big]{\frac{1}{\sqrt{d}}}$ \\
\midrule
$\cond{c,c}{f,\mathbf{x}}$
& Extreme
& $2 - \epsilon + \Oh{\epsilon^2}$
& $\epsilon + 3d\alpha + \Oh[\Big]{d^2\alpha^2 + \alpha + \frac{\epsilon}{d}}$ \\
& Realistic $\Big(\tfrac{\beta}{d}\Big)$
& $2 - \epsilon + \Oh[\Big]{\frac{1}{d^2}}$
& $\frac{3\beta}{\beta-1} - \frac{3\beta}{d(\beta-1)} + \Oh[\Big]{\epsilon + \frac{1}{d^2}}$ \\
& Realistic $\Big(\tfrac{\gamma}{\sqrt{d}}\Big)$
& $2 - \frac{2\gamma^2}{d(1+\gamma^2)} + \Oh[\Big]{\epsilon + \frac{1}{d^2}}$
& $3 + \frac{3}{\gamma \sqrt{d}} + \Oh[\Big]{\frac{1}{d}}$ \\
\midrule
\end{tabular}
\end{table}

\subsection{Forward error}
Let us bound the rounding error of layer normalisation computed in FP arithmetic.
In the analysis for both $\Layer{RN}$ \eqref{eq:rms-layernorm} and $\Layer{LN}$ \eqref{eq:layernorm}, we assume that the summations required to compute the $\ell_2$ norm and the mean are calculated using the same accumulation algorithm as in Model~\ref{model:fp_gemm}, excluding the quantisation terms.
This allows us to reuse the constant $\gamma_{MM}$ for consistency and clarity.

\begin{theorem}
\label{theorem:rounding-both-layernorm}
For every $\Matrix{x} \in \Real^d$, the value of $\Layer{RN}(\Matrix{x})$ computed in FP arithmetic satisfies
\begin{equation*}
    | \fl{\Layer{RN}(\Matrix{x})} - \Layer{RN}(\Matrix{x}) | \leq \Big( \tfrac{7}{2} \uh_{w} + \tfrac{1}{2} \gamma_{MM}(d, \uh_w) + \Oh{\uh_{w}^2} \Big) |\Layer{RN}(\Matrix{x})|.
\end{equation*}
If the mean-centred $\Matrix{y} = (\Matrix{I}_d - \tfrac{1}{d} \One \One^\trans) \Matrix{x}$ is non-zero, the value of $\Layer{LN}(\Matrix{x})$ computed in FP arithmetic satisfies
\begin{multline*}
    | \fl{\Layer{LN}(\Matrix{x})} - \Layer{LN}(\Matrix{x}) | \leq \Big(\tfrac{7}{2} \uh_{w} + \tfrac{1}{2} \gamma_{MM}(d, \uh_w) + + \Oh{\uh_w^2} \Big) |\Layer{LN}(\Matrix{x})| \\
    + \cond{\infty,\infty}{\Layer{RN},\Matrix{y}} \bigg( \uh_{w} + \Big(\uh_{w} + \gamma_{MM}(d, \uh_w)\Big) \frac{\Norm{\Matrix{x}}{1}}{d \Norm{\Matrix{y}}{\infty}} \bigg) \Norm{\Layer{LN}(\Matrix{x})}{\infty} \One.
\end{multline*}
\end{theorem}
\begin{proof}
For the $\Layer{RN}$ function, we begin by accumulating the squared normalisation factor:
\begin{equation*}
    | \fl{\Norm{\Matrix{x}}{2}^2 + \epsilon} - (\Norm{\Matrix{x}}{2}^2 + \epsilon) | \leq \big(\uh_{w} + \gamma_{MM}(d,\uh_w)\big) (\Norm{\Matrix{x}}{2}^2 + \epsilon).
\end{equation*}
The square root, the Hadamard product with gain $\Matrix{g}$, and the division by the denominator are correctly rounded to precision $\uh_{w}$ according to Model~\ref{model:fp_arithmetic}, whence the bound follows.

For the $\Layer{LN}$ function, denote by $\Matrix{\hat{y}} \in \Real^d$ the computed value of $\Matrix{y}$.
By triangle inequality,
\begin{equation*}
    | \fl{\Layer{LN}(\Matrix{x})} - \Layer{LN}(\Matrix{x}) | \leq | \fl{\Layer{RN}(\Matrix{\hat{y}})} - \Layer{RN}(\Matrix{\hat{y}}) | + | \Layer{RN}(\Matrix{\hat{y}}) - \Layer{RN}(\Matrix{y}) |.
\end{equation*}
As we have just shown, the first term is bounded by
\begin{align*}
    | \fl{\Layer{RN}(\Matrix{\hat{y}})} - \Layer{RN}(\Matrix{\hat{y}}) | &\leq \Big( \tfrac{7}{2} \uh_{w} + \tfrac{1}{2} \gamma_{MM}(d, \uh_w) + \Oh{\uh_{w}^2} \Big) |\Layer{RN}(\Matrix{\hat{y}})| \\
    &\leq \Big( \tfrac{7}{2} \uh_{w} + \tfrac{1}{2} \gamma_{MM}(d, \uh_w) + \Oh{\uh_{w}^2} \Big) \Big(|\Layer{RN}(\Matrix{y})| + |\Layer{RN}(\Matrix{\hat{y}}) - \Layer{RN}(\Matrix{y})|\Big).
\end{align*}
We bound the second term using Proposition~\ref{proposition:sensitivity_bound} in $\ell_\infty$-normwise case:
\begin{equation*}
    |\Layer{RN}(\Matrix{\hat{y}}) - \Layer{RN}(\Matrix{y})| \leq \Big( \cond{\infty,\infty}{\Layer{RN},\Matrix{y}} \dist{\infty}{\Matrix{\hat{y}}, \Matrix{y}} + \Oh[\big]{\dist{\infty}{\Matrix{\hat{y}}, \Matrix{y}}^2} \Big) \Norm{\Layer{RN}(\Matrix{y})}{\infty} \One.
\end{equation*}
Finally, the computed mean-centred vector $\Matrix{\hat{y}}$ satisfies
\begin{equation*}
    | \Matrix{\hat{y}} - \Matrix{y} | \leq \uh_{w} |\Matrix{y}| + \big( \uh_{w} + \gamma_{MM}(d, \uh_w) \big) \tfrac{1}{d} \Norm{\Matrix{x}}{1} \One.
\end{equation*}
We obtain the result by combining these bounds.
\end{proof}

Since $\uh_w \leq \uh_a$ by assumption, the bounds in Theorem~\ref{theorem:rounding-both-layernorm} can be simplified to
\begin{gather*}
    | \fl{\Layer{RN}(\Matrix{x})} - \Layer{RN}(\Matrix{x}) | \lesssim \gamma_{MM}(d,\uh_w) |\Layer{RN}(\Matrix{x})|, \\
    | \fl{\Layer{LN}(\Matrix{x})} - \Layer{LN}(\Matrix{x}) | \lesssim \gamma_{MM}(d,\uh_w) |\Layer{LN}(\Matrix{x})| + \cond{\infty,\infty}{\Layer{RN},\Matrix{y}} \bigg( \uh_{w} + \gamma_{MM}(d,\uh_w) \frac{\Norm{\Matrix{x}}{1}}{d \Norm{\Matrix{y}}{\infty}} \bigg) \Norm{\Layer{LN}(\Matrix{x})}{\infty} \One.
\end{gather*}
The FP computation of $\Layer{RN}$ is unconditionally forward stable with forward error governed by the chosen summation algorithm.
The case of $\Layer{LN}$ is more subtle since $\cond{\infty,\infty}{\Layer{RN},\Matrix{y}}$ is upper bounded by
\begin{equation}
\label{eq:cond-rn-universal-bound}
    \cond{\infty,\infty}{\Layer{RN},\Matrix{y}} \leq \frac{\sqrt{d} + 1}{2} \frac{\Norm{\Matrix{g}}{\infty} \Norm{\Matrix{y}}{\infty}}{\Norm{\Matrix{g} \hap \Matrix{y}}{\infty}}.
\end{equation}
When $\epsilon = 0$ and $\Matrix{g} = \One$, this bound is attained at vectors $\Matrix{y}$ of the form
\begin{equation*}
    |y_1| = 1, \quad |y_2| = \cdots = |y_d| = \frac{1}{\sqrt{d} - 1},
\end{equation*}
which falls precisely within our second `realistic' model of massive outliers.
Assuming that $\Matrix{x}$ has zero mean and $\Matrix{x} = \Matrix{y}$, we conclude that the rounding error bound for this example is
\begin{equation*}
    | \fl{\Layer{LN}(\Matrix{x})} - \Layer{LN}(\Matrix{x}) | \lesssim \Big( \sqrt{d} \uh_{w} + \gamma_{MM}(d,\uh_w) \Big) \Norm{\Layer{LN}(\Matrix{x})}{\infty} \One.
\end{equation*}
Regardless of the summation algorithm used to subtract the mean, our bound on the rounding error scales at least as $\sqrt{d} \uh_{w}$ when $\Matrix{x}$ has a massive outlier with background oscillations of order $1 / \sqrt{d}$.

In contrast, when the background oscillations about the mean are of order $1 / d$, Table~\ref{tab:rn_ln_zero_mean} shows that the condition number satisfies $\cond{\infty,\infty}{\Layer{RN},\Matrix{y}} = \Oh{1}$, which yields
\begin{equation*}
    | \fl{\Layer{LN}(\Matrix{x})} - \Layer{LN}(\Matrix{x}) | \lesssim \gamma_{MM}(d,\uh_w) |\Layer{LN}(\Matrix{x})| + \uh_{w} \Norm{\Layer{LN}(\Matrix{x})}{\infty} \One.
\end{equation*}
\section{Analysis of feedforward mechanisms}
\label{sec:ffn}

\subsection{Two-layer feedforward mechanism}
Next, we consider the two-layer FF mechanism as defined in \eqref{eq:ffn-gelu}:
\begin{equation*}
    \Layer{F}(\Matrix{x}) = \Matrix{W}_{down} \sigma^\ast(\Matrix{W}_{up} \Matrix{x}), \quad \Matrix{W}_{up} \in \Real^{D \times d}, \quad \Matrix{W}_{down} \in \Real^{d \times D}.
\end{equation*}
Recall that $\sigma^\ast$ denotes the componentwise application of the activation function $\sigma : \Real \to \Real$, which we assume to be a \emph{smooth ReLU-like function}.
We say that a smooth function $\sigma$ is ReLU-like if for every $\epsilon > 0$ there exists $\tau_{\epsilon} > 0$ such that
\begin{equation*}
    \max\{|\sigma(x)|, |\sigma'(x)|\} < \epsilon, \quad x < -\tau_{\epsilon}, \qquad
    \max\{|\sigma(x) / x - 1|, |\sigma'(x) - 1| < \epsilon, \quad x > \tau_{\epsilon}.
\end{equation*}
The ReLU function itself satisfies this property with $\tau_{\epsilon} = 0$ for every $\epsilon$, though it is not differentiable at zero.
By definition, the derivative of every smooth ReLU-like function is uniformly bounded:
\begin{equation*}
    \Norm{\sigma'}{C(\Real)} = \sup_{x \in \Real} |\sigma'(x)| \leq \max\left\{ 1 + \epsilon, \Norm{\sigma'}{C([\pm \tau_{\epsilon}])} \right\}, \quad [\pm \tau_{\epsilon}] = [-\tau_\epsilon, \tau_\epsilon].
\end{equation*}
A specific example to have in mind is the GELU function \citep{hendrycks2016gaussian}.

\begin{theorem}
\label{theorem:cond-ffn-gelu}
Let $\sigma$ be a smooth ReLU-like function, and let $\Matrix{x} \in \Real^d$.
For every $\epsilon > 0$ and every triplet $1 \leq p,q,r \leq \infty$, normwise condition numbers of $\Layer{F}$ are bounded by
\begin{equation*}
    \cond{p,q}{\Layer{F}, \Matrix{x}} \leq \Big( \epsilon \nu_{-} + \Norm{\sigma'}{C([\pm \tau_{\epsilon}])} \nu_{0} + (1 + \epsilon) \nu_{+} \Big) \frac{\Norm{\Matrix{x}}{p}}{\Norm{\Layer{F}(\Matrix{x})}{q}},
\end{equation*}
where $\nu_{\theta} = \Norm{\Matrix{W}_{down}^{(\theta)}}{r,q} \Norm{\Matrix{W}_{up}^{(\theta)}}{p,r}$ are defined for $\theta \in \{ -, 0, + \}$ based on $\Matrix{y} = \Matrix{W}_{up} \Matrix{x}$ as
\begin{gather*}
    \Matrix{W}_{down}^{(\theta)} = \Matrix{W}_{down}(:, \Omega_{\theta}), \quad \Matrix{W}_{up}^{(\theta)} = \Matrix{W}_{up}(\Omega_{\theta}, :), \\
    \Omega_{-} = \set{i}{y_i < -\tau_{\epsilon}}, \quad \Omega_{0} = \set{i}{|y_i| \leq \tau_{\epsilon}}, \quad \Omega_{+} = \set{i}{y_i > \tau_{\epsilon}}.
\end{gather*}
Its mixed condition numbers are bounded by
\begin{gather*}
    \cond{c,q}{\Layer{F}, \Matrix{x}} \leq \Big( \epsilon \xi_{-} + \Norm{\sigma'}{C([\pm \tau_{\epsilon}])} \xi_{0} + (1 + \epsilon) \xi_{+} \Big) \frac{1}{\Norm{\Layer{F}(\Matrix{x})}{q}}, \\
    \cond{p,c}{\Layer{F}, \Matrix{x}} \leq \Big( \epsilon \eta_{-} + \Norm{\sigma'}{C([\pm \tau_{\epsilon}])} \eta_{0} + (1 + \epsilon) \eta_{+} \Big) \Norm{\Matrix{x}}{p},
\end{gather*}
where
\begin{equation*}
    \xi_{\theta} = \Norm{\Matrix{W}_{down}^{(\theta)}}{r,q} \Norm{\Matrix{W}_{up}^{(\theta)} \diag{\Matrix{x}}}{\infty,r}, \quad \eta_{\theta} = \Norm{\diag{\Layer{F}(\Matrix{x})}^{-1} \Matrix{W}_{down}^{(\theta)}}{r,\infty} \Norm{\Matrix{W}_{up}^{(\theta)}}{p,r}.
\end{equation*}
The componentwise condition number is bounded by
\begin{equation*}
    \cond{c,c}{\Layer{F}, \Matrix{x}} \leq \Big( \epsilon \chi_{-} + \Norm{\sigma'}{C([\pm \tau_{\epsilon}])} \chi_{0} + (1 + \epsilon) \chi_{+} \Big),
\end{equation*}
where $\chi_{\theta} = \Norm{\diag{\Layer{F}(\Matrix{x})}^{-1} \Matrix{W}_{down}^{(\theta)}}{r,\infty} \Norm{\Matrix{W}_{up}^{(\theta)} \diag{\Matrix{x}}}{\infty,r}$.
\end{theorem}
\begin{proof}
The Jacobian of $\Layer{F}$ is given by $\Matrix{J}_{\Layer{F}}(\Matrix{x}) = \Matrix{W}_{down} \diag[\big]{{\sigma'}^{\ast}(\Matrix{y})} \Matrix{W}_{up}$.
The proof then follows from Proposition~\ref{proposition:condition_numbers}, the triangle inequality, and the submultiplicativity of matrix operator norms.
\end{proof}

Note that Theorem~\ref{theorem:cond-ffn-gelu} holds for the ReLU activation function provided that $\Matrix{W}_{up} \Matrix{x}$ contains no zeros.
In this case, we select $\epsilon = 0$ and $\tau_{\epsilon} = 0$, and only the third term remains in each bound.
Furthermore, the theorem holds for any $\tau_\epsilon$ from the definition of a smooth ReLU-like function; in the limit of $\tau_\epsilon \to \infty$,
\begin{equation}
\label{eq:cond-ffn-gelu-universal}
    \cond{p,q}{\Layer{F}, \Matrix{x}} \leq \Norm{\sigma'}{C(\Real)} \Norm{\Matrix{W}_{down}}{r,q} \Norm{\Matrix{W}_{up}}{p,r} \frac{\Norm{\Matrix{x}}{p}}{\Norm{\Layer{F}(\Matrix{x})}{q}}
\end{equation}
is a universal upper bound, and similar bounds hold for other condition numbers.

In practice, the GELU activation function is typically replaced with its $\tanh$-approximation
\begin{equation}
\label{eq:gelu-tanh}
    \sigma(x) = \frac{x}{2} \left[ 1 + \tanh\left( \sqrt{\frac{2}{\pi}} (x + 0.044715 x^3) \right) \right],
\end{equation}
which is also a smooth ReLU-like function.
The following theorem explicitly concerns \eqref{eq:gelu-tanh}.

\begin{theorem}
\label{theorem:rounding-ffn-gelu-tanh}
For every $\Matrix{x} \in \Real^d$ and the $\tanh$-approximation of the GELU activation function \eqref{eq:gelu-tanh}, the value of $\Layer{F}(\Matrix{x})$ computed in FP arithmetic satisfies
\begin{equation*}
    |\fl{\Layer{F}(\Matrix{x})} - \Layer{F}(\Matrix{x})| \leq \Big( \tfrac{21}{4} \uh_{w} + \tfrac{9}{2} \uh_{q} + \tfrac{5}{4} \gamma_{MM}(d) + \gamma_{MM}(D) + \Oh{\uh_{q}^2} \Big) |\Matrix{W}_{down}| |\Matrix{W}_{up}| |\Matrix{x}|.
\end{equation*}
\end{theorem}
\begin{proof}
According to Model~\ref{model:fp_gemm}, we have
\begin{equation*}
    |\fl{\Matrix{W}_{up} \Matrix{x}} - \Matrix{W}_{up} \Matrix{x}| \leq \Big( 2 \uh_{q} + \gamma_{MM}(d) + \Oh{\uh_{q}^2} \Big) |\Matrix{W}_{up}| |\Matrix{x}|.
\end{equation*}
Let $f(y) = \sqrt{2 / \pi} (y + 0.044715 \cdot y^3)$ and $g(y) = \tfrac{1}{2} (1 + \tanh f(y))$.
Model~\ref{model:fp_arithmetic} yields
\begin{equation*}
    |\fl{f^\ast(\Matrix{W}_{up} \Matrix{x})} - f^\ast(\Matrix{W}_{up} \Matrix{x})| \leq 5 \uh_{w} |f^\ast(\Matrix{W}_{up} \Matrix{x})| + |f'^\ast(\Matrix{W}_{up} \Matrix{x})| \Big( 2 \uh_{q} + \gamma_{MM}(d) \Big) |\Matrix{W}_{up}| |\Matrix{x}| + \Oh{\uh_{q}^2},
\end{equation*}
which we use to derive
\begin{multline*}
    |\fl{\sigma^\ast(\Matrix{W}_{up} \Matrix{x})} - \sigma^\ast(\Matrix{W}_{up} \Matrix{x})| \leq 4 \uh_{w} |\sigma^\ast(\Matrix{W}_{up} \Matrix{x})| + 5 \uh_{w} \frac{|\Matrix{W}_{up} \Matrix{x}|}{2} \tanh'^{\ast}\Big(f^\ast(\Matrix{W}_{up} \Matrix{x})\Big) |f^\ast(\Matrix{W}_{up} \Matrix{x})| \\
    + \Big( 2 \uh_{q} + \gamma_{MM}(d) \Big) \Big( g^\ast(\Matrix{W}_{up} \Matrix{x}) + |\Matrix{W}_{up} \Matrix{x}| g'^\ast(\Matrix{W}_{up} \Matrix{x}) \Big) |\Matrix{W}_{up}| |\Matrix{x}| + \Oh{\uh_{q}^2}.
\end{multline*}
In the second and third terms, the function $\tanh'(y) |y|$ is bounded by $1/2$ from above for every $y \in \Real$, and $g(y) + |y| g'(y)$ is upper bounded by $8/7$.
(Both are rational upper bounds of the numerically computed maxima.)
This results in a simplified bound
\begin{equation*}
    |\fl{\sigma^\ast(\Matrix{W}_{up} \Matrix{x})} - \sigma^\ast(\Matrix{W}_{up} \Matrix{x})| \leq 4 \uh_{w} |\sigma^\ast(\Matrix{W}_{up} \Matrix{x})| + \tfrac{5}{4} \Big( \uh_{w} + 2 \uh_{q} + \gamma_{MM}(d) \Big) |\Matrix{W}_{up}| |\Matrix{x}| + \Oh{\uh_{q}^2}.
\end{equation*}
Applying Model~\ref{model:fp_gemm} again and using $|\sigma^\ast(\Matrix{W}_{up} \Matrix{x})| \leq |\Matrix{W}_{up} \Matrix{x}| \leq |\Matrix{W}_{up}| |\Matrix{x}|$, we obtain the final bound.
\end{proof}

Theorem~\ref{theorem:cond-ffn-gelu} shows that the FP computation of the (approximate) GELU FF mechanism function is forward stable.
Under our assumption on the ordering of precisions, the forward error is determined primarily by the two GEMMs and the quantisation of their operands.

\subsection{Three-layer gated feedforward mechanism}
Recall the definition of the three-layer gated FF mechanism \eqref{eq:ffn-swiglu}:
\begin{equation*}
    \Layer{FG}(\Matrix{x}) = \Matrix{W}_{down} \Big( \sigma^\ast(\Matrix{W}_{gate}\Matrix{x}) \hap \Matrix{W}_{up} \Matrix{x} \Big), \quad \Matrix{W}_{up} \in \Real^{D \times d}, \quad \Matrix{W}_{gate} \in \Real^{D \times d}, \quad \Matrix{W}_{down} \in \Real^{d \times D}.
\end{equation*}
We assume again that the gating function $\sigma$ is a smooth ReLU-like function.

\begin{theorem}
\label{theorem:cond-ffn-swiglu}
Let $\sigma$ be a smooth ReLU-like function, and let $\Matrix{x} \in \Real^d$.
For every $\epsilon > 0$ and every triplet $1 \leq p,q,r \leq \infty$, normwise condition numbers of $\Layer{FG}$ are bounded by
\begin{equation*}
    \cond{p,q}{\Layer{FG}, \Matrix{x}} \leq \Big( \epsilon (\nu_{-} + \tilde{\nu}_{-}) + \Norm{\sigma'}{C([\pm \tau_{\epsilon}])} \nu_{0} + \Norm{\sigma}{C([\pm \tau_{\epsilon}])} \tilde{\nu}_{0} + (1 + \epsilon) (\nu_{+} + \Norm{\Matrix{W}_g^{(+)} \Matrix{x}}{\infty} \tilde{\nu}_{+}) \Big) \frac{\Norm{\Matrix{x}}{p}}{\Norm{\Layer{FG}(\Matrix{x})}{q}},
\end{equation*}
where
\begin{equation*}
    \nu_{\theta} = \Norm{\Matrix{W}_{down}^{(\theta)}}{r,q} \Norm{\diag{\Matrix{W}_{up}^{(\theta)} \Matrix{x}} \Matrix{W}_{gate}^{(\theta)}}{p,r}, \quad \tilde{\nu}_{\theta} = \Norm{\Matrix{W}_{down}^{(\theta)}}{r,q} \Norm{\Matrix{W}_{up}^{(\theta)}}{p,r}
\end{equation*}
are defined for $\theta \in \{ -, 0, + \}$ based on $\Matrix{y} = \Matrix{W}_{gate} \Matrix{x}$ as
\begin{gather*}
    \Matrix{W}_{down}^{(\theta)} = \Matrix{W}_{down}(:, \Omega_{\theta}), \quad \Matrix{W}_{up}^{(\theta)} = \Matrix{W}_{up}(\Omega_{\theta}, :), \quad \Matrix{W}_{gate}^{(\theta)} = \Matrix{W}_{gate}(\Omega_{\theta}, :), \\
    \Omega_{-} = \set{i}{y_i < -\tau_{\epsilon}}, \quad \Omega_{0} = \set{i}{|y_i| \leq \tau_{\epsilon}}, \quad \Omega_{+} = \set{i}{y_i > \tau_{\epsilon}}.
\end{gather*}
Its mixed condition numbers are bounded by
\begin{gather*}
    \cond{c,q}{\Layer{F}, \Matrix{x}} \leq \Big( \epsilon (\xi_{-} + \tilde{\xi}_{-}) + \Norm{\sigma'}{C([\pm \tau_{\epsilon}])} \xi_{0} + \Norm{\sigma}{C([\pm \tau_{\epsilon}])} \tilde{\xi}_{0} + (1 + \epsilon) (\xi_{+} + \Norm{\Matrix{W}_g^{(+)} \Matrix{x}}{\infty} \tilde{\xi}_{+}) \Big) \frac{1}{\Norm{\Layer{FG}(\Matrix{x})}{q}}, \\
    \cond{p,c}{\Layer{F}, \Matrix{x}} \leq \Big( \epsilon (\eta_{-} + \tilde{\eta}_{-}) + \Norm{\sigma'}{C([\pm \tau_{\epsilon}])} \eta_{0} + \Norm{\sigma}{C([\pm \tau_{\epsilon}])} \tilde{\eta}_{0} + (1 + \epsilon) (\eta_{+} + \Norm{\Matrix{W}_g^{(+)} \Matrix{x}}{\infty} \tilde{\eta}_{+}) \Big) \Norm{\Matrix{x}}{p},
\end{gather*}
where
\begin{gather*}
    \xi_{\theta} = \Norm{\Matrix{W}_{down}^{(\theta)}}{r,q} \Norm{\diag{\Matrix{W}_{up}^{(\theta)} \Matrix{x}} \Matrix{W}_{gate}^{(\theta)} \diag{\Matrix{x}}}{\infty,r}, \quad \tilde{\xi}_{\theta} = \Norm{\Matrix{W}_{down}^{(\theta)}}{r,q} \Norm{\Matrix{W}_{up}^{(\theta)} \diag{\Matrix{x}}}{\infty,r}, \\
    \eta_{\theta} = \Norm{\diag{\Layer{FG}(\Matrix{x})}^{-1} \Matrix{W}_{down}^{(\theta)}}{r,\infty} \Norm{\diag{\Matrix{W}_{up}^{(\theta)} \Matrix{x}} \Matrix{W}_{gate}^{(\theta)}}{p,r}, \quad \tilde{\eta}_{\theta} = \Norm{\diag{\Layer{FG}(\Matrix{x})}^{-1} \Matrix{W}_{down}^{(\theta)}}{r,\infty} \Norm{\Matrix{W}_{up}^{(\theta)}}{p,r}.
\end{gather*}
The componentwise condition number is bounded by
\begin{equation*}
    \cond{c,c}{\Layer{F}, \Matrix{x}} \leq \Big( \epsilon (\chi_{-} + \tilde{\chi}_{-}) + \Norm{\sigma'}{C([\pm \tau_{\epsilon}])} \chi_{0} + \Norm{\sigma}{C([\pm \tau_{\epsilon}])} \tilde{\chi}_{0} + (1 + \epsilon) (\chi_{+} + \Norm{\Matrix{W}_g^{(+)} \Matrix{x}}{\infty} \tilde{\chi}_{+}) \Big),
\end{equation*}
where
\begin{gather*}
    \chi_{\theta} = \Norm{\diag{\Layer{FG}(\Matrix{x})}^{-1} \Matrix{W}_{down}^{(\theta)}}{r,\infty} \Norm{\diag{\Matrix{W}_{up}^{(\theta)} \Matrix{x}} \Matrix{W}_{gate}^{(\theta)} \diag{\Matrix{x}}}{\infty,r}, \\
    \tilde{\chi}_{\theta} = \Norm{\diag{\Layer{FG}(\Matrix{x})}^{-1} \Matrix{W}_{down}^{(\theta)}}{r,\infty} \Norm{\Matrix{W}_{up}^{(\theta)} \diag{\Matrix{x}}}{\infty,r}.
\end{gather*}
\end{theorem}
\begin{proof}
The Jacobian of $\Layer{FG}$ at $\Matrix{x}$ equals
\begin{equation*}
    \Matrix{J}_{\Layer{FG}}(\Matrix{x}) = \Matrix{W}_{down} \diag[\Big]{\sigma'^\ast (\Matrix{W}_{gate} \Matrix{x})} \diag{\Matrix{W}_{up} \Matrix{x}} \Matrix{W}_{gate} + \Matrix{W}_{down} \diag[\Big]{\sigma^\ast (\Matrix{W}_{gate} \Matrix{x})} \Matrix{W}_{up}.
\end{equation*}
The argument then repeats the proof of Theorem~\ref{theorem:cond-ffn-gelu}.
\end{proof}

A comment regarding the ReLU function and the limit of $\tau_\epsilon \to \infty$, similar to the one related to Theorem~\ref{theorem:cond-ffn-gelu}, applies to Theorem~\ref{theorem:cond-ffn-swiglu}.

Most practical implementations of \eqref{eq:ffn-swiglu} use the SiLU gating function $\sigma(x) = x / (1 + \exp(-x))$. 
This results in the SwiGLU mechanism \citep{shazeer2020glu}, and we analyse its forward error.

\begin{theorem}
\label{theorem:rounding-ffn-swiglu}
For every $\Matrix{x} \in \Real^d$ and the SiLU gating function, the value of $\Layer{FG}(\Matrix{x})$ computed in FP arithmetic satisfies
\begin{equation*}
    |\fl{\Layer{FG}(\Matrix{x})} - \Layer{FG}(\Matrix{x})| \leq \Big( 4 \uh_{w} + \tfrac{13}{2} \uh_{q} + \tfrac{9}{4} \gamma_{MM}(d) + \gamma_{MM}(D) + \Oh{\uh_{q}^2} \Big) |\Matrix{W}_{down}| \Big( |\Matrix{W}_{gate}| |\Matrix{x}| \hap |\Matrix{W}_{up}| |\Matrix{x}| \Big).
\end{equation*}
\end{theorem}
\begin{proof}
By Model~\ref{model:fp_gemm}, we have for $\theta \in \{ up, gate \}$:
\begin{equation*}
    |\fl{\Matrix{W}_{\theta} \Matrix{x}} - \Matrix{W}_{\theta} \Matrix{x}| \leq \Big( 2 \uh_{q} + \gamma_{MM}(d) + \Oh{\uh_{q}^2} \Big) |\Matrix{W}_{\theta}| |\Matrix{x}|.
\end{equation*}
The derivative of the SiLU function is uniformly bounded by $\Norm{\sigma'}{C(\Real)} < 5/4$, and hence
\begin{equation*}
    |\fl{\sigma^\ast(\Matrix{W}_{gate} \Matrix{x})} - \sigma^\ast(\Matrix{W}_{gate} \Matrix{x})| \leq 3 \uh_{w} |\sigma^\ast(\Matrix{W}_{gate} \Matrix{x})| + \tfrac{5}{4} \Big( 2 \uh_{q} + \gamma_{MM}(d) \Big) |\Matrix{W}_{gate}| |\Matrix{x}| + \Oh{\uh_{q}^2}.
\end{equation*}
Let $\Matrix{y} = \sigma^\ast(\Matrix{W}_{gate} \Matrix{x}) \hap \Matrix{W}_{up} \Matrix{x}$ and denote by $\hat{\Matrix{y}}$ its computed value; they satisfy
\begin{align*}
    |\hat{\Matrix{y}} - \Matrix{y}| \leq \uh_{w} |\Matrix{y}| &+ \Big( 3 \uh_{w} + 2\uh_{q} + \gamma_{MM}(d) \Big) |\sigma^\ast(\Matrix{W}_{gate} \Matrix{x})| \hap |\Matrix{W}_{up}| |\Matrix{x}| \\
    &+ \tfrac{5}{4} \Big( 2 \uh_{q} + \gamma_{MM}(d) \Big) |\Matrix{W}_{gate}| |\Matrix{x}| \hap |\Matrix{W}_{up} \Matrix{x}| + \Oh{\uh_{q}^2}.
\end{align*}
Since the absolute value of SiLU is upper bounded by the absolute value of its argument,
\begin{equation*}
    |\hat{\Matrix{y}} - \Matrix{y}| \leq \Big( 4 \uh_{w} + \tfrac{9}{2} \uh_{q} + \tfrac{9}{4} \gamma_{MM}(d) \Big) |\Matrix{W}_{gate}| |\Matrix{x}| \hap |\Matrix{W}_{up}| |\Matrix{x}| + \Oh{\uh_{q}^2}.
\end{equation*}
Finally, we use Model~\ref{model:fp_gemm} again.
\end{proof}

As with the GELU mechanism, the FP computation of the SwiGLU mechanism is forward stable with forward error dominated by three GEMMs and their quantisations errors.
\section{Analysis of softmax}
\label{sec:softmax}

\subsection{Unshifted evaluation}
The key non-linearity of the self-attention mechanism is the softmax function defined in \eqref{eq:softmax} as
\begin{equation*}
    \Layer{S}(\Matrix{x}) = \frac{1}{\sum_{i = 1}^n \exp(x_i)} \begin{bmatrix}
        \exp(x_1) & \cdots & \exp(x_n)
    \end{bmatrix}^\trans, \quad \Matrix{x} \in \Real^n.
\end{equation*}
Below, we write $\Matrix{s} = \Layer{S}(\Matrix{x})$ and denote by $s_{[i]}$ the $i$th largest entry of $\Matrix{s}$, e.g, $s_{[1]} = \Norm{\Matrix{s}}{\infty}$ and $s_{[n]} = \Norm{\Matrix{s}}{-\infty}$.

\begin{theorem}
\label{theorem:cond-softmax}
Let $\Matrix{x} \in \Real^n$.
Normwise condition numbers of $\Layer{S}$ satisfy
\begin{gather*}
    \cond{2,2}{\Layer{S}, \Matrix{x}} \leq s_{[1]} (1 - s_{[1]} + s_{[2]}) \frac{\Norm{\Matrix{x}}{2}}{\Norm{\Matrix{s}}{2}}, \\
    \cond{1,1}{\Layer{S}, \Matrix{x}} = 2 \max_{1 \leq i \leq n} s_i (1 - s_i) \Norm{\Matrix{x}}{1}, \quad \cond{\infty,\infty}{\Layer{S}, \Matrix{x}} = 2 \max_{1 \leq i \leq n} s_i (1 - s_i) \frac{\Norm{\Matrix{x}}{\infty}}{\Norm{\Matrix{s}}{\infty}}, \\
    \cond{1,\infty}{\Layer{S}, \Matrix{x}} = \max_{1 \leq i \leq n} s_i (1 - s_i) \frac{\Norm{\Matrix{x}}{1}}{\Norm{\Matrix{s}}{\infty}}, \quad \cond{\infty,1}{\Layer{S}, \Matrix{x}} = 4 \max_{\Omega \subseteq [1,n]} \bigg( \sum_{i \in \Omega} s_i \bigg) \bigg( \sum_{j \not\in \Omega} s_j \bigg) \Norm{\Matrix{x}}{\infty},
\end{gather*}
its mixed condition numbers equal
\begin{gather*}
    \cond{c,\infty}{\Layer{S}, \Matrix{x}} = \frac{1}{\Norm{\Matrix{s}}{\infty}} \max_{1 \leq i \leq n} s_i \big( \Matrix{s}^\trans |\Matrix{x}| + |x_i| (1 - 2 s_i) \big), \quad \cond{c,1}{\Layer{S}, \Matrix{x}} = 2 \sum_{i = 1}^n s_i (1 - s_i) |x_i|, \\
    \cond{c,2}{\Layer{S}, \Matrix{x}} = \frac{1}{\Norm{\Matrix{s}}{2}} \sqrt{\sum_{i = 1}^n s_i^2 \big( \Matrix{s}^\trans |\Matrix{x}| + |x_i| (1 - 2 s_i) \big)^2}, \quad \cond{2,c}{\Layer{S}, \Matrix{x}} = \sqrt{1 - 2 \Norm{\Matrix{s}}{-\infty} + \Norm{\Matrix{s}}{2}^2} \Norm{\Matrix{x}}{2}, \\
    \cond{\infty, c}{\Layer{S}, \Matrix{x}} = 2 (1 - \Norm{\Matrix{s}}{-\infty}) \Norm{\Matrix{x}}{\infty}, \quad \cond{1, c}{\Layer{S}, \Matrix{x}} = (1 - \Norm{\Matrix{s}}{-\infty}) \Norm{\Matrix{x}}{1},
\end{gather*}
and its componentwise condition number equals
\begin{equation*}
    \cond{c,c}{\Layer{S}, \Matrix{x}} = \Matrix{s}^\trans |\Matrix{x}| + \max_{1 \leq i \leq n} |x_i| (1 - 2 s_i). 
\end{equation*}
\end{theorem}
\begin{proof}
The Jacobian is $\Matrix{J}_{\Layer{S}}(\Matrix{x}) = \diag{\Matrix{s}} - \Matrix{s} \Matrix{s}^\trans$, and we express the condition numbers via Proposition~\ref{proposition:condition_numbers}.
The $\ell_2$-normwise bound follows from the bound on $\Norm{\Matrix{J}_{\Layer{S}}(\Matrix{x})}{2,2}$ proved in \cite{yudin2025pay}.
To derive $\cond{\infty,1}{\Layer{S}, \Matrix{x}}$, we denote $\Omega_{\Matrix{v}} = \{ i : v_i \geq 0 \}$.
Then
\begin{equation*}
    \Norm{\Matrix{J}_{\Layer{S}}(\Matrix{x})}{\infty,1} = \max_{\Matrix{v} \in \{-1,1\}^{n}} \sum_{i} s_i \Big| (1-s_i)v_i - \sum_{j \neq i} s_j v_j \Big| = \max_{\Matrix{v}} \bigg( 2 \sum_{i \in \Omega_{\Matrix{v}}} s_i \sum_{j \not\in \Omega_{\Matrix{v}}} s_j + 2 \sum_{i \not\in \Omega_{\Matrix{v}}} s_i \sum_{j \in \Omega_{\Matrix{v}}} s_j \bigg).
\end{equation*}
To obtain the remaining normwise condition numbers, we analyse the norms of the Jacobian:
\begin{gather*}
    \Norm{\Matrix{J}_{\Layer{S}}(\Matrix{x})}{1,1} = \Norm{\Matrix{J}_{\Layer{S}}(\Matrix{x})}{\infty,\infty} = \max_i s_i \Big( 1 - s_i + \sum_{j \neq i} s_j \Big) = 2 \max_i s_i (1 - s_i), \\
    \Norm{\Matrix{J}_{\Layer{S}}(\Matrix{x})}{1,\infty} = \max_{i,j} |\delta_{i,j} s_i - s_i s_j| = \max_i s_i \max\{ 1 - s_i, \max_{j \neq i} s_j \} = \max_i s_i (1 - s_i),
\end{gather*}
where we rely on $1 - s_i = \sum_{j \neq i} s_j$ in both equalities.
For the mixed condition number, we compute
\begin{equation*}
    \Norm{\Matrix{J}_{\Layer{S}}(\Matrix{x}) \diag{\Matrix{x}}}{\infty,\infty} = \max_i s_i \sum_{j} |\delta_{i,j} - s_j| |x_j| = \max_i s_i \Big( (1 - s_i) |x_i| + \sum_{j \neq i} s_j |x_j| \Big),
\end{equation*}
and the componentwise condition number is obtained similarly.
\end{proof}

Since the components of $\Layer{S}(\Matrix{x})$ are uniformly bounded, softmax is generally ill-conditioned when the norm of $\Matrix{x}$ is large: small relative perturbations of the input lead to substantial changes in the resulting probability distribution.
As a specific illustration, consider $\Matrix{x} = [x~x]^\trans$ and let $\hat{\Matrix{x}} = [(1 + \delta)x,~(1-\delta)x]^\trans$ be its perturbation with relative componentwise distance of $\delta > 0$.
Then
\begin{equation*}
    \Layer{S}(\Matrix{x}) = \begin{bmatrix}
        0.5 & 0.5
    \end{bmatrix}^\trans, \quad \lim_{x \to +\infty} \Layer{S}(\hat{\Matrix{x}}) = \begin{bmatrix}
        1 & 0
    \end{bmatrix}^\trans.
\end{equation*}

Note that $\max_i s_i (1 - s_i)$ and $\max_{\Omega} (\sum_{i \in \Omega} s_i) (\sum_{j \not\in \Omega} s_j)$ appearing in the normwise condition numbers are maximised when both factors are as close to $1/2$ as possible.
That is, when one $s_i$ carries about half of the probability or when the probability mass can be split into two almost equal parts, respectively.

The rounding error analysis of the softmax function was carried out in \cite{blanchard2021accurately}.
Its FP evaluation was shown to be unconditionally forward stable with a forward error,
\begin{equation*}
    |\fl{\Layer{S}(\Matrix{x})} - \Layer{S}(\Matrix{x})| \leq \big( (n+3) \uh_{w} + \Oh{\uh_{w}^2} \big) \Layer{S}(\Matrix{x}).
\end{equation*}
This bound was obtained for recursive summation, whereas on modern GPUs the normalisation constant is computed via pairwise or block summation.
To account for this discrepancy, we reuse our notation:
\begin{equation}
\label{eq:error-softmax}
    |\fl{\Layer{S}(\Matrix{x})} - \Layer{S}(\Matrix{x})| \leq \big( 3\uh_{w} + \gamma_{MM}(n,\uh_w)\big) \Layer{S}(\Matrix{x}).
\end{equation}
The proof of this bound is analogous to \cite{blanchard2021accurately}.

\subsection{Shifted evaluation}
The evaluation algorithm whose forward error was just presented computes the normalisation constant in a straightforward manner.
However, the summation of a long list of potentially large exponentials is likely to overflow in FP arithmetic.
Practical algorithms for softmax rely on its shift-invariance property $\Layer{S}(\Matrix{x} + c \One) = \Layer{S}(\Matrix{x})$ and subtract $(\max_i x_i) \One$ from $\Matrix{x}$ to ensure that every exponential in the sum is bounded by one.
Such shifted evaluation is also forward stable, though its forward error,
\begin{equation*}
    |\fl{\Layer{S}(\Matrix{x})} - \Layer{S}(\Matrix{x})| \leq \Big( \big(n+2 + 2 \max_{1 \leq i \leq n} x_i - 2 \min_{1 \leq j \leq n} x_j\big) \uh_{w} + \Oh{\uh_{w}^2} \Big) \Layer{S}(\Matrix{x}),
\end{equation*}
now depends on the range of components of $\Matrix{x}$ \citep{blanchard2021accurately}.

The shifted evaluation algorithm also necessitates a reconsideration of the sensitivity analysis.
Let $\Matrix{x}, \hat{\Matrix{x}} \in \Real^n$ attain their maxima at the last component, i.e., $x_n = \max_i x_i$ and $\hat{x}_n = \max_i \hat{x}_i$.
Then softmax is evaluated according to $\Matrix{s} = \Layer{S}(\Matrix{x} - x_n \One)$ and $\hat{\Matrix{s}} = \Layer{S}(\hat{\Matrix{x}} - \hat{x}_n \One)$.
Both shifted arguments vanish exactly at the last component, so we consider new inputs $\Matrix{y}, \hat{\Matrix{y}} \in \Real^{n-1}$ given by
\begin{equation*}
    \Matrix{y} = \begin{bmatrix}
        x_1 - x_n & \cdots & x_{n-1} - x_n
    \end{bmatrix}^\trans, \quad \hat{\Matrix{y}} = \begin{bmatrix}
        \hat{x}_1 - \hat{x}_n & \cdots & \hat{x}_{n-1} - \hat{x}_n
    \end{bmatrix}^\trans.
\end{equation*}
We then introduce an auxiliary `reduced' softmax function
\begin{equation*}
    \Layer{Z}(\Matrix{y}) = \frac{1}{1 + \sum_{i = 1}^{n-1} \exp(y_i)} \begin{bmatrix}
        \exp(y_1) & \cdots & \exp(y_{n-1})
    \end{bmatrix}^\trans
\end{equation*}
and use $\Matrix{z} = \Layer{Z}(\Matrix{y})$ to express the original softmax as
\begin{equation*}
    \Matrix{s} = \begin{bmatrix}
       \Matrix{z}^\trans & 1 - \sum_{i = 1}^{n-1} z_i
    \end{bmatrix}^\trans.
\end{equation*}
Let us estimate the perturbation in $\Matrix{s}$ induced by the perturbation in $\Matrix{y}$.
Focusing on the $\ell_1$ norm, which is twice the total-variation distance and measures the probability mass shift, the triangle inequality gives
\begin{equation*}
    \Norm{\hat{\Matrix{s}} - \Matrix{s}}{1} \leq 2 \Norm{\hat{\Matrix{z}} - \Matrix{z}}{1}.
\end{equation*}
To bound the right-hand side, we rely on the \emph{absolute} condition number $\condabs{\infty,1}{\Layer{Z}, \Matrix{y}} = \Norm{\Matrix{J}_{\Layer{Z}}(\Matrix{y})}{\infty,1}$, which measures how absolute output perturbations depend on absolute input perturbations \citep{rice1966theory}.
An analogue of Proposition~\ref{proposition:sensitivity_bound} immediately yields
\begin{equation*}
    \Norm{\hat{\Matrix{s}} - \Matrix{s}}{1} \leq 2\condabs{\infty,1}{\Layer{Z}, \Matrix{y}} \Norm{\hat{\Matrix{y}} - \Matrix{y}}{\infty} + 2C \Norm{\hat{\Matrix{y}} - \Matrix{y}}{\infty}^2 \leq 4\condabs{\infty,1}{\Layer{Z}, \Matrix{y}} \Norm{\hat{\Matrix{x}} - \Matrix{x}}{\infty} + 8C \Norm{\hat{\Matrix{x}} - \Matrix{x}}{\infty}^2.
\end{equation*}

\begin{proposition}
\label{proposition:condabs-softmax-reduced}
Let $\Matrix{y} \in \Real^{n-1}$. The absolute normwise condition number of $\Layer{Z}$ equals
\begin{equation*}
    \condabs{\infty,1}{\Layer{Z}, \Matrix{y}} = \bigg( \sum_{i = 1}^{n-1} z_i \bigg) \bigg( 1 - \sum_{i = 1}^{n-1} z_i \bigg) + 4 \max_{\Omega \subseteq [1,n-1]} \bigg( \sum_{i \in \Omega} z_i \bigg) \bigg( \sum_{j \not\in \Omega} z_j \bigg).
\end{equation*}
\end{proposition}
\begin{proof}
Similarly to $\Layer{S}$, the Jacobian is given by $\Matrix{J}_{\Layer{Z}}(\Matrix{y}) = \diag{\Matrix{z}} - \Matrix{z} \Matrix{z}^\trans$.
As in the proof of Theorem~\ref{theorem:cond-softmax},
\begin{align*}
    \Norm{\Matrix{J}_{\Layer{Z}}(\Matrix{y})}{\infty,1} &= \max_{\Matrix{v} \in \{ -1,1 \}^{n-1}} \bigg( \sum_{i \in \Omega_{\Matrix{v}}} z_i \Big( 1 + \sum_{k \not\in \Omega_{\Matrix{v}}} z_k - \sum_{j \in \Omega_{\Matrix{v}}} z_j \Big) + \sum_{i \not\in \Omega_{\Matrix{v}}} z_i \Big( 1 - \sum_{k \not\in \Omega_{\Matrix{v}}} z_k + \sum_{j \in \Omega_{\Matrix{v}}} z_j \Big) \bigg) \\
    &= \max_{\Matrix{v}} \bigg( \sum_{i = 1}^{n-1} z_i - \Big(\sum_{k \not\in \Omega_{\Matrix{v}}} z_k - \sum_{j \in \Omega_{\Matrix{v}}} z_j\Big)^2 \bigg) \\
    &= \max_{\Matrix{v}} \bigg( \sum_{i = 1}^{n-1} z_i + 4 \Big( \sum_{k \not\in \Omega_{\Matrix{v}}} z_k \Big) \Big( \sum_{j \in \Omega_{\Matrix{v}}} z_j \Big) - \Big( \sum_{i = 1}^{n-1} z_i \Big)^2 \bigg),
\end{align*}
and the only difference is that now $\sum_{i = 1}^{n-1} z_i \neq 1$.
\end{proof}

To compare the sensitivity of softmax when computed with a shifted and unshifted algorithm,\footnote{\cite{blanchard2021accurately} mention that the shift does not alter the condition numbers of softmax $\Layer{S}$. While true for its \emph{absolute} condition numbers, it was overlooked that the absolute condition number of a different function $\Layer{Z}$ determines the sensitivity of shifted softmax evaluation to input perturbations.} we consider the ratio of absolute condition numbers times the constant factor from the error bound:
\begin{equation*}
    \zeta(\Matrix{x}) = \frac{4\condabs{\infty,1}{\Layer{Z}, \Matrix{y}}}{\condabs{\infty,1}{\Layer{S}, \Matrix{x}}}.
\end{equation*}
Note that the numerator can be bounded via
\begin{equation*}
    \condabs{\infty,1}{\Layer{Z}, \Matrix{y}} \leq (1 - s_n) s_n + 4 \left( \frac{1 - s_n}{2} \right)^2 = 1 - s_n.
\end{equation*}
The denominator satisfies a bound $\condabs{\infty,1}{\Layer{S}, \Matrix{x}} \geq 4 s_n (1 - s_n)$ corresponding to $\Omega = \{ n \}$.
At the same time, for every $0 \leq c \leq 1$, there exists $\Omega$ such that $| \sum_{i \in \Omega} s_i - c | \leq s_n / 2$.
Choosing $c = 1/2$, we obtain
\begin{equation*}
    \condabs{\infty,1}{\Layer{S}, \Matrix{x}} \geq 4 \Big( \frac{1 - s_n}{2} \Big) \Big( \frac{1 + s_n}{2} \Big) = 1 - s_n^2.
\end{equation*}
Combining the three bounds together, we get a monotonically decaying envelope of the ratio
\begin{equation*}
    \zeta(\Matrix{x}) \leq \begin{cases}
        \frac{4}{1 + s_n}, & 0 < s_n \leq 1/3, \\
        \frac{1}{s_n}, & 1/3 < s_n \leq 1.
    \end{cases}
\end{equation*}
When the probability distribution $\Matrix{s}$ is highly concentrated with $s_n \approx 1$, shifted evaluation preserves the `basic' sensitivity of softmax, which actively dampens input perturbations:
\begin{equation*}
    \condabs{\infty,1}{\Layer{S}, \Matrix{x}} \approx 4\condabs{\infty,1}{\Layer{Z}, \Matrix{y}} \leq 4(1 - s_n).
\end{equation*}
In the other extreme case of flat distributions, the shift increases the sensitivity by at most a factor of 4, and this scaling factor is achieved exactly at $\lim_{n \to \infty} \zeta(\One) = 4$.
Therefore, our analysis confirms that the shifted evaluation of softmax introduces no ill-conditioning.

\subsection{Concentrated distributions and attention sinks}
The forward-error bounds developed in \cite{blanchard2021accurately} for the unshifted and shifted evaluation of softmax do not impose any assumptions on its input.
Meanwhile, pre-trained LLMs typically generate highly concentrated softmax probability distributions, where almost the entire probability mass resides in a small fraction of the entries and the majority of exponentials fall below the absorption threshold.
We analyse the unshifted algorithm in this case, using recursive summation and assuming that the maximum input entry is the first one; overflows and underflows are not taken into account.

\begin{theorem}
\label{theorem:rounding-softmax-absorption}
Let $\Matrix{x} \in \Real^{n}$ satisfy $x_1 = \max_{1 \leq i \leq n} x_i$, and let
\begin{equation*}
    \Omega = \set[\Big]{1 \leq i \leq n}{\exp(x_i - x_1) < \frac{\uh_w}{2} \frac{1 - \uh_w}{1 + \uh_w}}.
\end{equation*}
The value of $\Matrix{s} = \Layer{S}(\Matrix{x})$ computed in FP arithmetic satisfies
\begin{equation*}
    |\fl{\Matrix{s}} - \Matrix{s}| \leq \Big( (n - |\Omega| + 2) \uh_w + \sum_{i \in \Omega} s_i + \Oh{\uh_w^2} \Big) \Matrix{s}.
\end{equation*}
\end{theorem}
\begin{proof}
Let $\hat{p}_0 = 0$. Define $p_{k} = \hat{p}_{k-1} + \fl{\exp x_{k}}$ and $\hat{p}_{k} = \fl{p_{k}}$.
First of all, note that since $\fl{\exp x_{k}} > 0$ by Model~\ref{model:fp_arithmetic}, we have by the monotonicity properties of rounding that
\begin{equation*}
    \hat{p}_k \geq \cdots \geq \hat{p}_1 = \fl{\exp x_{1}} \geq (1 - \uh_w) \exp(x_{1}) > 0.
\end{equation*}
When $k \not\in \Omega$, Model~\ref{model:fp_arithmetic} yields $|\hat{p}_{k} - p_k| \leq \uh_w (|\hat{p}_{k-1}| + |\fl{\exp x_{k}}|) = \uh_w p_k$.
Suppose that $k \in \Omega$, then
\begin{equation*}
    \fl{\exp x_{k}} \leq (1 + \uh_w) \exp(x_{k}) < \frac{\uh_w}{2} (1 - \uh_w) \exp(x_1) \leq \frac{\uh_w}{2} \fl{\exp x_1} \leq \frac{\uh_w}{2} \hat{p}_{k-1}.
\end{equation*}
This means that $\fl{\exp x_{k}}$ falls below the absorption threshold, and hence $\hat{p}_k = \hat{p}_{k-1}$.
Consequently, the recursive summation of the normalisation constant of softmax satisfies
\begin{equation*}
    \hat{p}_n = \fl[\Big]{\sum_{k = 1}^{n} \exp x_k} = \fl[\Big]{\sum_{k \not\in \Omega} \exp x_k},
\end{equation*}
and thus $|\hat{p}_n - \sum_{k = 1}^{n} \exp(x_k)|$ is bounded by
\begin{align*}
    \Big( (n - |\Omega|) \uh_w + \Oh{\uh_w^2} \Big) \sum_{k \not\in \Omega} \exp(x_k) + \sum_{i \in \Omega} \exp(x_i) \leq \Big( (n - |\Omega|) \uh_w + \sum_{i \in \Omega} s_i + \Oh{\uh_w^2} \Big) \sum_{k = 1}^{n} \exp(x_k).
\end{align*}
It remains to compute $\exp(x_k) / \hat{p}_n$ in FP arithmetic, resulting in the final rounding error bound.
\end{proof}

When all of the exponentials $\exp(x_i - x_1)$ for $i \in \Omega$ are close to $\uh_w / 2$, the rounding-error bound in Theorem~\ref{theorem:rounding-softmax-absorption} reduces to
\begin{equation*}
    |\fl{\Matrix{s}} - \Matrix{s}| \lesssim \Big( (n - |\Omega|) \uh_w + \sum_{i \in \Omega} s_i \Big) \Matrix{s} \leq \Big( (n - |\Omega|) \uh_w + \sum_{i \in \Omega} \exp(x_i - x_1) \Big) \Matrix{s} \lesssim \Big( (n - |\Omega|) \uh_w + \frac{|\Omega|}{2} \uh_w \Big) \Matrix{s},
\end{equation*}
effectively recovering \cite[Theorem~3.4]{blanchard2021accurately}.
Our bound leads to improvement when a stronger inequality $\exp(x_i - x_1) \ll \uh_w$ holds.

A modification of Theorem~\ref{theorem:rounding-softmax-absorption} to the shifted algorithm follows analogously, introducing the gap $x_1 - \min_j x_j$ into the bound as in \cite[Theorem~4.3]{blanchard2021accurately}.
Suppose that $x_i = x_{\min}$ for all $i \in \Omega$ and $\exp(x_{\min} - x_1) = \epsilon \uh_w$.
Then the relative error of the shifted algorithm is of order
\begin{equation*}
    (n - |\Omega| + x_1 - x_{\min}) \uh_w + \sum_{i \in \Omega} s_i \leq \Big(n - |\Omega| + \ln(\uh_w^{-1}) + \ln(\epsilon^{-1}) + |\Omega| \epsilon \Big) \uh_w.
\end{equation*}
In long-context inference of modern LLMs, it is typical to have, for example, $n = 10^6$ and $n - \Omega = 10^2$ with $\epsilon = 10^{-10}$.
When single precision is used, the parentheses approximately equal $140$, guaranteeing the loss of accuracy of only two orders of magnitude for recursive summation.

The assumption $x_1 = \max_{1 \leq i \leq n} x_i$ is not particularly stringent in long-context inference.
It ensures that the running sum becomes sufficiently large early on, allowing absorptions to occur.
If the maximum is attained at index $k$, the error bound would still hold with a restricted absorption index set $\Omega \cap \{ i \geq k \}$.
Crucially, in the context of attention, massive pre-trained LLMs exhibit an \emph{attention-sink} phenomenon \citep{xiao2024efficient}, whereby an attention head systematically routes most of the probability mass to the first few tokens in the sequence.
Thus, our theorem naturally captures the behaviour of attention sinks.

In the absence of an attention sink, the maximum is typically attained at $k \approx n$, because queries generally attend more strongly to recent keys.
In this case, we can reverse the order of summation to achieve the same improvement in the rounding error bound.

Finally, note that Theorem~\ref{theorem:rounding-softmax-absorption} relies on the specific accumulation ordering of recursive summation.
For a different summation algorithm, the absorption dynamics would need to be studied separately.
\section{Analysis of self-attention mechanism}
\label{sec:attention}

\subsection{Condition numbers}
Recall the definition of causal multi-head self-attention with RoPE given in \eqref{eq:sh-attention}-\eqref{eq:mh-attention} as
\begin{gather*}
    \Layer{A}(\Matrix{X}) = \Matrix{W}_O \begin{bmatrix}
        \Matrix{a}_{1,1} & \cdots & \Matrix{a}_{1,N} \\
        \vdots & \ddots & \vdots \\
        \Matrix{a}_{n_{head},1} & \cdots & \Matrix{a}_{n_{head},N}
    \end{bmatrix} \in \Real^{d \times N}, \quad \Matrix{X} \in \Real^{d \times N}, \\
    \Matrix{a}_{h,n} = \begin{bmatrix}
        \Matrix{v}_{h,1} & \cdots & \Matrix{v}_{h,n}
    \end{bmatrix} \Layer{S}(\Matrix{b}_{h,n}) \in \Real^{d_{head}}, \quad \Matrix{b}_{h,n} = \frac{1}{\sqrt{d_{head}}}
    \begin{bmatrix}
        \Matrix{R}_{1}\Matrix{k}_{h,1} & \cdots & \Matrix{R}_{n}\Matrix{k}_{h,n}
    \end{bmatrix}^\trans \Matrix{R}_n \Matrix{q}_{h,n} \in \Real^n,
\end{gather*}
where $\Matrix{R}_t \in \Real^{d_{head} \times d_{head}}$ are orthogonal and $\Matrix{v}_{h,t} \in \Real^{d_{head}}$ are obtained according to
\begin{equation*}
    \begin{bmatrix}
        \Matrix{V}_{1} \\
        \vdots \\
        \Matrix{V}_{n_{head}}
    \end{bmatrix} = \begin{bmatrix}
        \Matrix{v}_{1,1} & \cdots & \Matrix{v}_{1,N} \\
        \vdots & \ddots & \vdots \\
        \Matrix{v}_{n_{head},1} & \cdots & \Matrix{v}_{n_{head},N} \\
    \end{bmatrix} = \begin{bmatrix}
        \Matrix{W}_{V,1} \\
        \vdots \\
        \Matrix{W}_{V,n_{head}}
    \end{bmatrix} \Matrix{X} \in \Real^{d \times N},
\end{equation*}
and likewise for $\Matrix{q}_{h,t}, \Matrix{k}_{h,t} \in \Real^{d_{head}}$.
To analyse the conditioning of $\Layer{A}$, we begin by computing the Jacobian of every individual $\Matrix{a}_{h,n}$.
We have $\Matrix{J}_{\Matrix{a}_{h,n}}(\Matrix{X}) = [\Matrix{\Matrix{J}}_{h,n}~0] \in \Real^{d_{head} \times dN}$ with
\begin{multline*}
    \Matrix{\Matrix{J}}_{h,n} = \Matrix{s}_{h,n}^\trans \krp \Matrix{W}_{V,h} + \tfrac{1}{\sqrt{d_{head}}} \Matrix{V}_{h,[1,n]} \Matrix{J}_{\Layer{S}}(\Matrix{b}_{h,n}) \Big( (\Matrix{I}_n \krp \Matrix{q}_{h,n}^\trans \Matrix{R}_n^\trans) \diag{\Matrix{R}_1, \ldots, \Matrix{R}_n} (\Matrix{I}_n \krp \Matrix{W}_{K,h}) \\
    + \begin{bmatrix}
        \Matrix{R}_{1}\Matrix{k}_{h,1} & \cdots & \Matrix{R}_{n}\Matrix{k}_{h,n}
    \end{bmatrix}^\trans \Matrix{R}_n \Matrix{W}_{Q,h} (\Matrix{e}_n^\trans \krp \Matrix{I}_d) \Big) \in \Real^{d_{head} \times dn},
\end{multline*}
where $\Matrix{s}_{h,n} = \Layer{S}(\Matrix{b}_{h,n})$.
Vectorising the input and output of $\Layer{A}$, we get its block lower triangular Jacobian
\begin{equation*}
    \Matrix{J}_{\Layer{A}}(\Matrix{X}) = (\Matrix{I}_N \krp \Matrix{W}_{O}) 
    \begin{bmatrix}
        \Matrix{J}_{\Matrix{a}_{1,1}}(\Matrix{X})^\trans & \cdots & \Matrix{J}_{\Matrix{a}_{n_{head},1}}(\Matrix{X})^\trans & \Matrix{J}_{\Matrix{a}_{1,2}}(\Matrix{X})^\trans & \cdots & \Matrix{J}_{\Matrix{a}_{n_{head},N}}(\Matrix{X})^\trans
    \end{bmatrix}^\trans \in \Real^{dN \times dN}.
\end{equation*}

We bound absolute normwise condition numbers of causal self-attention for $\ell_1 \to \ell_p$ operator norms applied to input and output matrix perturbations.\footnote{Theorem~\ref{theorem:condabs-attention} below and its proof extend verbatim to non-causal self-attention, grouped-query and multi-query modifications of attention. We assume unshifted evaluation of softmax.}
By Proposition~\ref{proposition:condition_numbers}, this amounts to computing the norm of the Fr\'echet derivative, or equivalently the $(1,p) \to (1,p)$ norm of the Jacobian with
\begin{equation*}
    \Norm{\Vect{\Matrix{X}}}{(1,p)} = \Norm{\Matrix{X}}{1,p} = \max_{1 \leq n \leq N} \Norm{\Matrix{x}_n}{p}.
\end{equation*}

\begin{lemma}
\label{lemma:1p-norm}
Let $1 \leq p, q \leq \infty$.
\begin{enumerate}
    \item If $\Matrix{B} \in \Real^{d_{head} \times d}$ and $\Matrix{U} \in \Real^{m \times N}$ then $\Norm{\Matrix{U} \krp \Matrix{B}}{(1,p),(1,q)} = \Norm{\Matrix{U}}{\infty,\infty} \Norm{\Matrix{B}}{p,q}$.
    \item If $\Matrix{B}_1, \ldots, \Matrix{B}_m \in \Real^{d_{head} \times dN}$ then $\Norm{[\Matrix{B}_1^\trans~\cdots~\Matrix{B}_m^\trans]^\trans}{(1,p),(1,q)} = \max_{1 \leq i \leq m} \Norm{\Matrix{B}_i}{(1,p),q}$.
    \item If $\Matrix{B}_1, \ldots, \Matrix{B}_N \in \Real^{k \times d}$ then $\Norm{\diag{\Matrix{B}_1, \ldots, \Matrix{B}_N}}{(1,p),(1,q)} = \max_{1 \leq n \leq N} \Norm{\Matrix{B}_n}{p,q}$.
\end{enumerate}
\end{lemma}
\begin{proof}
By the definition of operator norms, we maximise $\Norm{(\Matrix{U} \krp \Matrix{B}) \Vect{\Matrix{Y}}}{(1,q)}$ over matrices $\Matrix{Y} \in \Real^{d \times N}$ with $\Norm{\Matrix{Y}}{1,p} = \max_{1 \leq n \leq N} \Norm{\Matrix{y}_n}{p} = 1$.
It holds that
\begin{equation*}
    \Norm{(\Matrix{U} \krp \Matrix{B}) \Vect{\Matrix{Y}}}{(1,q)} = \max_{1 \leq i \leq m} \Norm[\Big]{\sum_{n} u_{i,n} \Matrix{B} \Matrix{y}_n}{q} \leq \max_{i} \sum_{n} |u_{i,n}| \Norm{\Matrix{B} \Matrix{y}_n}{q} \leq \max_{i} \sum_{n} |u_{i,n}| \Norm{\Matrix{B}}{p,q}.
\end{equation*}
Likewise for the second statement, we obtain
\begin{equation*}
    \Norm{[\Matrix{B}_1^\trans~\cdots~\Matrix{B}_m^\trans]^\trans \Vect{\Matrix{Y}}}{(1,q)} = \max_{1 \leq i \leq m} \Norm{\Matrix{B}_i \Vect{\Matrix{Y}}}{q} \leq \max_i \Norm{\Matrix{B}_i}{(1,p),q}.
\end{equation*}
And analogously for the third statement:
\begin{equation*}
    \Norm{\diag{\Matrix{B}_1, \ldots, \Matrix{B}_N} \Vect{\Matrix{Y}}}{(1,q)} = \max_{1 \leq n \leq N} \Norm{\Matrix{B}_n \Matrix{y}_n}{q} \leq \max_{n} \Norm{\Matrix{B}_n}{p,q}.
\end{equation*}
All inequalities can be attained with a properly chosen $\Matrix{Y}$.
\end{proof}

We consider two cases: a general input matrix and a matrix with a \emph{featurewise} massive outlier.
The latter is a typical scenario in massive LLMs, where massive outliers across individual tokens (columns) are localised at the same feature (row).

\begin{theorem}
\label{theorem:condabs-attention}
Let $\Matrix{X} \in \Real^{d \times N}$.
\begin{enumerate}
    \item Let $1 \leq p,q \leq \infty$.
    The absolute normwise condition number of $\Layer{A}$ satisfies
    \begin{equation*}
        \condabs{(1,p),(1,q)}{\Layer{A}, \Matrix{X}} \leq \Norm{\Matrix{W}_O}{(1,q),q} \max_{1 \leq h \leq n_{head}} \tilde{\kappa}_{h,p,q}(\Layer{A}), \quad \tilde{\kappa}_{h,p,q}(\Layer{A}) = \max_{1 \leq n \leq N} \Norm{\Matrix{J}_{h,n}(\Matrix{X})}{(1,p),q}.
    \end{equation*}
    Denote $\tilde{\kappa}_{h}(\Layer{S}) = \max_{1 \leq n \leq N} \Norm{\Matrix{J}_{\Layer{S}}(\Matrix{b}_{h,n})}{\infty,1}$.
    Then
    \begin{equation*}
        \tilde{\kappa}_{h,p,q}(\Layer{A}) \leq \Norm{\Matrix{W}_{V,h}}{p,q} \Big( 1 + \tilde{\kappa}_{h}(\Layer{S}) M_{p} \Norm{\Matrix{X}}{1,p}^2 \Big), \quad M_{p} = 2 d_{head}^{-1/2} \Norm{\Matrix{W}_{Q,h}}{p,2} \Norm{\Matrix{W}_{K,h}}{p,2}.
    \end{equation*}
    \item Denote by $\Matrix{f}_i \in \Real^N$ the $i$th row of $\Matrix{X}$.
    If $\Norm{\Matrix{f}_{i}}{\infty} \leq \vartheta < \Norm{\Matrix{f}_1}{\infty}$ for $i \geq 2$, then
    \begin{gather*}
        \tilde{\kappa}_{h,\infty,\infty}(\Layer{A}) \leq \Norm{\Matrix{W}_{V,h}}{\infty,\infty} + \tilde{\kappa}_{h}(\Layer{S}) \Big( \Norm{\Matrix{w}_{V,h,1}}{\infty} \Norm{\Matrix{X}}{1,\infty} + \Norm{\Matrix{W}_{V,h}}{\infty,\infty} \vartheta \Big) \Big( \mu_\infty \Norm{\Matrix{X}}{1,\infty} + M_\infty \vartheta \Big), \\
        \mu_\infty = d_{head}^{-1/2} (\Norm{\Matrix{W}_{K,h}}{\infty,2} \Norm{\Matrix{w}_{Q,h,1}}{2} + \Norm{\Matrix{W}_{Q,h}}{\infty,2} \Norm{\Matrix{w}_{K,h,1}}{2}).
    \end{gather*}
    \item In the absence of RoPE, the constants can be improved to
    \begin{gather*}
        M_{p} = 2 d_{head}^{-1/2} \Norm{\Matrix{W}_{Q,h}^\trans \Matrix{W}_{K,h}}{p,p^*}, \quad \tfrac{1}{p} + \tfrac{1}{p^*} = 1, \\
        \mu_\infty = d_{head}^{-1/2} (\Norm{\Matrix{W}_{K,h}^\trans \Matrix{w}_{Q,h,1}}{1} + \Norm{\Matrix{W}_{Q,h}^\trans \Matrix{w}_{K,h,1}}{1}).
    \end{gather*}
\end{enumerate}
\end{theorem}
\begin{proof}
By Lemma~\ref{lemma:1p-norm} and the definition of $(1,p)$ norms,
\begin{equation*}
    \Norm{\Matrix{J}_{\Layer{A}}(\Matrix{X})}{(1,p),(1,q)} \leq \Norm{\Matrix{W}_O}{(1,q),q} \max_{1 \leq h \leq n_{head}, 1 \leq n \leq N} \Norm{\Matrix{J}_{\Matrix{a}_{h,n}}(\Matrix{X})}{(1,p),q} = \Norm{\Matrix{W}_O}{(1,q),q} \max_{h,n} \Norm{\Matrix{J}_{h,n}(\Matrix{X})}{(1,p),q}.
\end{equation*}
We bound $\Norm{\Matrix{J}_{h,n}(\Matrix{X})}{(1,p),q}$ using the triangle inequality.
First, we have by Lemma~\ref{lemma:1p-norm} that
\begin{equation*}
    \Norm{\Matrix{s}_{h,n}^\trans \krp \Matrix{W}_{V,h}}{(1,p),q} = \Norm{\Matrix{s}_{h,n}^\trans}{\infty,\infty} \Norm{\Matrix{W}_{V,h}}{p,q} = \Norm{\Matrix{s}_{h,n}}{1} \Norm{\Matrix{W}_{V,h}}{p,q} = \Norm{\Matrix{W}_{V,h}}{p,q}.
\end{equation*}
We refer to Theorem~\ref{theorem:cond-softmax} to bound
\begin{equation*}
    \Norm{\Matrix{V}_{h,[1,n]} \Matrix{J}_{\Layer{S}}(\Matrix{b}_{h,n})}{\infty,q} \leq \Norm{\Matrix{V}_{h,[1,n]}}{1,q} \Norm{\Matrix{J}_{\Layer{S}}(\Matrix{b}_{h,n})}{\infty,1} \leq \tilde{\kappa}_{h}(\Layer{S}) \Norm{\Matrix{W}_{V,h}}{p,q} \Norm{\Matrix{X}}{1,p}.
\end{equation*}
Next, according to Lemma~\ref{lemma:1p-norm} and the structure of RoPE matrices,
\begin{align*}
    \Norm{\diag{\Matrix{q}_{h,n}^\trans \Matrix{R}_1^\trans \Matrix{R}_{n} \Matrix{W}_{K,h}, \ldots, \Matrix{q}_{h,n}^\trans \Matrix{R}_n^\trans \Matrix{R}_{n} \Matrix{W}_{K,h}}}{(1,p),\infty} &= \max_t \Norm{\Matrix{q}_{h,n}^\trans \Matrix{R}_t^\trans \Matrix{R}_{n} \Matrix{W}_{K,h}}{p,\infty} \\
    &\leq \Norm{\Matrix{q}_{h,n}}{2} \Norm{\Matrix{W}_{K,h}}{p,2} \\
    &\leq \Norm{\Matrix{W}_{Q,h}}{p,2} \Norm{\Matrix{x}_n}{p} \Norm{\Matrix{W}_{K,h}}{p,2}.
\end{align*}
Finally, with the help of Lemma~\ref{lemma:1p-norm}, we get
\begin{align*}
    \Norm[\big]{\begin{bmatrix}
        \Matrix{R}_{1}\Matrix{k}_{h,1} & \cdots & \Matrix{R}_{n}\Matrix{k}_{h,n}
    \end{bmatrix}^\trans \Matrix{R}_n \Matrix{W}_{Q,h} (\Matrix{e}_n^\trans \krp \Matrix{I}_d)}{(1,p),\infty} &\leq \Norm{\Matrix{K}_{h,[1,n]}}{1,2} \Norm{\Matrix{W}_{Q,h} (\Matrix{e}_n^\trans \krp \Matrix{I}_d)}{(1,p),2} \\
    &\leq \Norm{\Matrix{W}_{K,h}}{p,2} \Norm{\Matrix{X}}{1,p} \Norm{\Matrix{W}_{Q,h} (\Matrix{e}_n^\trans \krp \Matrix{I}_d)}{(1,p),2} \\
    &= \Norm{\Matrix{W}_{K,h}}{p,2} \Norm{\Matrix{X}}{1,p} \Norm{\Matrix{e}_n^\trans \krp \Matrix{W}_{Q,h}}{(1,p),2} \\
    &= \Norm{\Matrix{W}_{K,h}}{p,2} \Norm{\Matrix{X}}{1,p} \Norm{\Matrix{W}_{Q,h}}{p,2}.
\end{align*}
Combining these partial bounds yields the final result.

To obtain a bound in the specific case of a featurewise massive outlier, it suffices to incorporate the following triangle-inequality bounds in the derivation:
\begin{gather*}
    \Norm{\Matrix{V}_{h,[1,n]}}{1,\infty} \leq \Norm{\Matrix{w}_{V,h,1}}{\infty} \Norm{\Matrix{X}}{1,\infty} + \Norm{\Matrix{W}_{V,h}}{\infty,\infty} \vartheta, \\
    \Norm{\Matrix{q}_{h,n}}{2} \leq \Norm{\Matrix{w}_{Q,h,1}}{2} \Norm{\Matrix{X}}{1,\infty} + \Norm{\Matrix{W}_{Q,h}}{\infty,2} \vartheta, \\
    \Norm{\Matrix{K}_{h,[1,n]}}{1,2} \leq \Norm{\Matrix{w}_{K,h,1}}{2} \Norm{\Matrix{X}}{1,\infty} + \Norm{\Matrix{W}_{K,h}}{\infty,2} \vartheta.
\end{gather*}

Without RoPE, the product $\Matrix{W}_{K,h}^\trans \Matrix{W}_{Q,h}$ need not be broken apart, leading to smaller constants.
\end{proof}

We highlight two properties of the bounds in Theorem~\ref{theorem:condabs-attention}.
First, they are independent of the internal structure of orthogonal RoPE matrices because $\ell_p \to \ell_2$ operator norms were chosen for the key and query weight matrices.

Second, the bounds are independent of the sequence length $N$ because $\Matrix{X} \in \Real^{d \times N}$ enters them with the $\ell_1 \to \ell_p$ norm (i.e., the maximum columnwise $\ell_p$ norm), while the weight matrices are structurally independent of $N$.
The key to achieving sequence-length independence lies in using the $\ell_\infty \to \ell_1$ absolute condition number of softmax; note also that it is uniformly bounded by one.

Previously, a bound on the local Lipschitz constant (i.e., absolute condition number) of mean-field self-attention that is independent of $N$ was derived in \cite{castin2024smooth}.
In the discrete case, their bound corresponds to $p = q = 2$; however, it contains a multiplicative factor $\exp \Oh{\Norm{\Matrix{X}}{1,2}^2}$, whereas our bound is strictly quadratic in the norm of $\Matrix{X}$.

We note that after independently developing Theorem~\ref{theorem:condabs-attention} during the revision of our manuscript, we became aware of concurrent work \cite{emadi2026exact} establishing the special $p = q = 2$ case without RoPE.

Let us compare the general and specialised bounds from Theorem~\ref{theorem:condabs-attention} for $p = q = \infty$ in the `extreme' case of $\vartheta = 0$; the latter can be rewritten as
\begin{equation*}
    \tilde{\kappa}_{h,\infty,\infty}(\Layer{A}) \leq \Norm{\Matrix{W}_{V,h}}{\infty,\infty} \Big( 1 + \lambda \tilde{\kappa}_{h}(\Layer{S}) M_{\infty} \Norm{\Matrix{X}}{1,\infty}^2 \Big)
\end{equation*}
with an improvement factor $\lambda$ bounded by
\begin{equation*}
    \lambda \leq \frac{1}{2} \frac{\Norm{\Matrix{W}_{V,h}}{1,\infty}}{\Norm{\Matrix{W}_{V,h}}{\infty,\infty}} \bigg( \frac{\Norm{\Matrix{W}_{Q,h}}{1,2}}{\Norm{\Matrix{W}_{Q,h}}{\infty,2}} + \frac{\Norm{\Matrix{W}_{K,h}}{1,2}}{\Norm{\Matrix{W}_{K,h}}{\infty,2}} \bigg).
\end{equation*}
The right-hand side ranges from $1 / d^2$ to $1$, quantifying the degree to which the presence of a massive outlier reduces the condition number of an attention head.
Note also that $\Norm{\Matrix{W}_O}{(1,\infty),\infty} = \Norm{\Matrix{W}_O}{\infty,\infty}$, which simplifies the bound on the condition number of multi-head attention.

\subsection{Forward error}
To bound the forward rounding error of self-attention, we need to describe the RoPE matrices $\{ \Matrix{R}_{n} \}_{n \in \Z}$ in more detail.
Let $\tau \in \N$ be large, and let $d_{head}$ be even.
The matrix $\Matrix{R}_n \in \Real^{d_{head} \times d_{head}}$ is defined as
\begin{equation*}
    \Matrix{R}_n = \diag[\Big]{\Matrix{R}_n^{(0)}, \ldots, \Matrix{R}_n^{(d_{head}/2-1)}}, \quad \Matrix{R}_n^{(k)} = \begin{bmatrix}
        \cos(n \omega_k) & -\sin(n \omega_k) \\
        \sin(n \omega_k) &  \cos(n \omega_k)
    \end{bmatrix}, \quad \omega_k = \tau^{-2k/d_{head}}.
\end{equation*}
Thus, every RoPE matrix is a block-diagonal rotation matrix.
In standard implementations, the RoPE matrices are precomputed in high precision for a sufficiently wide range of $n$.
To maintain clarity in the forthcoming analysis, we neglect the rounding errors introduced at the generation of $\Matrix{R}_n$.

\begin{theorem}
\label{theorem:rounding-attn}
Let $\Matrix{X} \in \Real^{d \times N}$ and denote
\begin{equation*}
    \Delta_h = \Big(10 \uh_{q} + 5\uh_{a} + \gamma_{MM}(d_{head}) + 2\gamma_{MM}(d) \Big) 2 d_{head}^{-1/2} \Norm{|\Matrix{W}_{K,h}|}{\infty,2} \Norm{|\Matrix{W}_{Q,h}|}{\infty,2}
\end{equation*}
for every $1 \leq h \leq n_{head}$.
The value of $\Layer{A}(\Matrix{X})$ computed in FP arithmetic satisfies
\begin{multline*}
    \Norm{\fl{\Layer{A}(\Matrix{X})} - \Layer{A}(\Matrix{X})}{1,\infty} \leq \Big( 6 \uh_{q} + 2\gamma_{MM}(d) + \gamma_{MM}(N) + 3 \uh_{w} + \gamma_{MM}(N, \uh_w) \\
    + \max_h \tilde{\kappa}_{h}(\Layer{S}) \Delta_h \Norm{\Matrix{X}}{1,\infty}^2 + \Oh[\big]{\uh_q^2 + \uh_q^2 \Norm{\Matrix{X}}{1,\infty}^2} \Big) \Norm{\Matrix{W}_O}{\infty,\infty} \max_h \Norm{\Matrix{W}_{V,h}}{\infty,\infty} \Norm{\Matrix{X}}{1,\infty},
\end{multline*}
where $\tilde{\kappa}_{h}(\Layer{S})$ is defined in Theorem~\ref{theorem:condabs-attention}.
\end{theorem}
\begin{proof}
According to Model~\ref{model:fp_gemm}, it holds for every $h$ and $n$ that
\begin{equation*}
    |\fl{\Matrix{v}_{h,n}} - \Matrix{v}_{h,n}| \leq \Big( 2 \uh_{q} + \gamma_{MM}(d) + \Oh{\uh_{q}^2} \Big) |\Matrix{W}_{V,h}| |\Matrix{x}_n|,
\end{equation*}
and similarly for $\Matrix{k}_{h,n}$ and $\Matrix{q}_{h,n}$.
The application of RoPE matrices is done block by block and results in
\begin{equation*}
    |\fl{\Matrix{R}_{n} \Matrix{q}_{h,n}} - \Matrix{R}_n \Matrix{q}_{h,n}| \leq \Big( 4 \uh_{q} + 2 \uh_a + \gamma_{MM}(d) + \Oh{\uh_{q}^2} \Big) |\Matrix{R}_{n}| |\Matrix{W}_{Q,h}| |\Matrix{x}_n|.
\end{equation*}
and likewise for $\Matrix{R}_t \Matrix{k}_{h,t}$ for $1 \leq t \leq n$.
The rounding error at the $t$th entry of $\Matrix{b}_{h,n}$ is then bounded by
\begin{equation*}
    \Big(10 \uh_{q} + 5\uh_{a} + \gamma_{MM}(d_{head}) + 2\gamma_{MM}(d) + \Oh{\uh_{q}^2}\Big) \frac{\big( |\Matrix{R}_{t}| |\Matrix{W}_{K,h}| |\Matrix{x}_t| \big)^\trans |\Matrix{R}_{n}| |\Matrix{W}_{Q,h}| |\Matrix{x}_n|}{\sqrt{d_{head}}}.
\end{equation*}
Since $\Norm{||\Matrix{R}_t|}{2,2} \leq \sqrt{2}$, this quantity is in turn bounded by $\Delta_h \Norm{\Matrix{X}}{1,\infty}^2 + \Oh{\uh_q^2 \Norm{\Matrix{X}}{1,\infty}^2}$.
Combining \eqref{eq:error-softmax} with triangle inequality and Proposition~\ref{proposition:sensitivity_bound}, we obtain for $\Matrix{s}_{h,n} = \Layer{S}(\Matrix{b}_{h,n})$ that
\begin{equation*}
    \Norm{\fl{\Matrix{s}_{h,n}} - \Matrix{s}_{h,n}}{1} \leq 3 \uh_{w} + \gamma_{MM}(n,\uh_w) + \Norm{\Matrix{J}_{\Layer{S}}(\Matrix{b}_{h,n})}{\infty,1} \Delta_h \Norm{\Matrix{X}}{1,\infty}^2 + \Oh{\uh_{w}^2 + \uh_q^2 \Norm{\Matrix{X}}{1,\infty}^2}.
\end{equation*}
It follows by triangle inequality that $\Norm{\fl{\Matrix{a}_{h,n}} - \Matrix{a}_{h,n}}{\infty}$ is bounded by
\begin{multline*}
    \Big( 4 \uh_{q} + \gamma_{MM}(n) + \gamma_{MM}(d) + 3 \uh_{w} + \gamma_{MM}(n,\uh_w) \\
    + \tilde{\kappa}_{h}(\Layer{S}) \Delta_h \Norm{\Matrix{X}}{1,\infty}^2 + \Oh[\big]{\uh_q^2 + \uh_q^2 \Norm{\Matrix{X}}{1,\infty}^2} \Big) \Norm{\Matrix{W}_{V,h}}{\infty,\infty} \Norm{\Matrix{X}}{1,\infty},
\end{multline*}
which eventually leads to the final error bound after another application of Model~\ref{model:fp_gemm}.
\end{proof}

In Theorem~\ref{theorem:rounding-attn}, we assumed that softmax is evaluated with the unshifted algorithm.
Furthermore, as stated in Section~\ref{sec:background}, we assumed that every GEMM uses the same quantisation precision, and hence that the aggregated quantisation errors $10 \uh_{q}$ and $6 \uh_{q}$ could be tightened with a more fine-grained analysis.
Nevertheless, the structure of our bound already isolates the corresponding amplification factors.

In the long-context scenario with large $N$, the error is dominated by the FP evaluation of softmax and the subsequent matrix product with the value matrix; both worst-case terms result from summations, and so their scaling with respect to $N$ could be reduced with probabilistic analysis.
Crucially, however, the softmax probability distributions generated by pre-trained LLMs are typically concentrated, so that the vast majority of summands in both computations fall below the absorption threshold.
Therefore, the large $N$ could be effectively replaced by a much smaller number (Theorem~\ref{theorem:rounding-softmax-absorption}), justifying the empirical success of long-context inference.

Note that the bound can be tightened when $\Matrix{X}$ has a featurewise massive outlier (cf. Theorem~\ref{theorem:condabs-attention}).
\section{Analysis of the entire transformer architecture}
\label{sec:transformer}

\subsection{Theoretical analysis}
\label{subsec:transformer-theory}
To obtain a readable bound on the forward rounding error for the inference of a deep transformer $\Layer{T}$, we need to aggregate the output errors due to input perturbations and FP evaluation for all of its constituent parts.
For definiteness, let us focus on the architecture \eqref{eq:transformer-RN} with layer-normalisation function $\Layer{RN}$ \eqref{eq:rms-layernorm} and approximate GELU FF mechanism \eqref{eq:gelu-tanh}.

\begin{proposition}
\label{proposition:rounding-ff-block}
Let $\Matrix{x}, \Matrix{\hat{x}} \in \Real^d$.
\begin{enumerate}
    \item The value of $\Layer{RN}(\Matrix{\hat{x}})$ computed in FP arithmetic satisfies
    \begin{gather*}
        \Norm{\fl{\Layer{RN}(\Matrix{\hat{x}})} - \Layer{RN}(\Matrix{x})}{\infty} \leq \delta_{\Layer{RN}} \Norm{\Layer{RN}(\Matrix{x})}{\infty} + (1 + \delta_{\Layer{RN}}) \condabs{\infty,\infty}{\Layer{RN}, \Matrix{x}} \Norm{\Matrix{\hat{x}} - \Matrix{x}}{\infty} + \Oh{\Norm{\Matrix{\hat{x}} - \Matrix{x}}{\infty}^2},\\
        \delta_{\Layer{RN}} = \tfrac{7}{2} \uh_w + \tfrac{1}{2} \gamma_{MM}(d, \uh_w) + \Oh{\uh_w^2}.
    \end{gather*}
    \item The value of $\Layer{F}(\Matrix{\hat{x}})$ computed in FP arithmetic satisfies
    \begin{gather*}
        \Norm{\fl{\Layer{F}(\Matrix{\hat{x}})} - \Layer{F}(\Matrix{x})}{\infty} \leq \delta_{\Layer{F}} \omega_{u,d} \Norm{\Matrix{x}}{\infty} + (\tfrac{8}{7} + \delta_{\Layer{F}}) \omega_{u,d} \Norm{\Matrix{\hat{x}} - \Matrix{x}}{\infty} + \Oh{\Norm{\Matrix{\hat{x}} - \Matrix{x}}{\infty}^2}, \\
        \delta_{\Layer{F}} = \tfrac{21}{4} \uh_{w} + \tfrac{9}{2} \uh_{q} + \tfrac{5}{4} \gamma_{MM}(d) + \gamma_{MM}(D) + \Oh{\uh_{q}^2}, \quad \omega_{u,d} = \Norm{\Matrix{W}_{down}}{\infty,\infty} \Norm{\Matrix{W}_{up}}{\infty,\infty}.
    \end{gather*}
    \item The value of $\Layer{F}(\Layer{RN}(\Matrix{\hat{x}}))$ computed in FP arithmetic satisfies
    \begin{align*}
        \Norm{\fl{\Layer{F}(\Layer{RN}(\Matrix{\hat{x}}))} - \Layer{F}(\Layer{RN}(\Matrix{x}))}{\infty} &\leq \big( \delta_{\Layer{F}} + \delta_{\Layer{RN}} (\tfrac{8}{7} + \delta_{\Layer{F}}) \big) \omega_{u,d} \Norm{\Layer{RN}(\Matrix{x})}{\infty} \\
        &+ (1 + \delta_{\Layer{RN}}) (\tfrac{8}{7} + \delta_{\Layer{F}}) \condabs{\infty,\infty}{\Layer{RN},\Matrix{x}} \omega_{u,d} \Norm{\Matrix{\hat{x}} - \Matrix{x}}{\infty} + \Oh{\Norm{\Matrix{\hat{x}} - \Matrix{x}}{\infty}^2}.
    \end{align*}
    \item The value of $\Layer{B}_{\Layer{F}}(\Matrix{\hat{x}}) = \Matrix{\hat{x}} + \Layer{F}(\Layer{RN}(\Matrix{\hat{x}}))$ computed in FP arithmetic satisfies
    \begin{multline*}
        \Norm{\fl{\Layer{B}_{\Layer{F}}(\Matrix{\hat{x}})} - \Layer{B}_{\Layer{F}}(\Matrix{x})}{\infty} \leq \uh_r \Norm{\Matrix{x}}{\infty} + \Big[ \uh_r + (1 + \uh_r) \big( \delta_{\Layer{F}} + \delta_{\Layer{RN}} (\tfrac{8}{7} + \delta_{\Layer{F}}) \big) \Big] \omega_{u,d} \Norm{\Layer{RN}(\Matrix{x})}{\infty} \\
        + (1 + \uh_r) \Big[ 1 + (1 + \delta_{\Layer{RN}}) (\tfrac{8}{7} + \delta_{\Layer{F}}) \omega_{u,d} \condabs{\infty,\infty}{\Layer{RN},\Matrix{x}} \Big] \Norm{\Matrix{\hat{x}} - \Matrix{x}}{\infty} + \Oh{\Norm{\Matrix{\hat{x}} - \Matrix{x}}{\infty}^2}.
    \end{multline*}
\end{enumerate}
\end{proposition}
\begin{proof}
By triangle inequality and Theorem~\ref{theorem:rounding-both-layernorm},
\begin{align*}
    |\fl{\Layer{RN}(\Matrix{\hat{x}})} - \Layer{RN}(\Matrix{x})| &\leq |\fl{\Layer{RN}(\Matrix{\hat{x}})} - \Layer{RN}(\Matrix{\hat{x}})| + |\Layer{RN}(\Matrix{\hat{x}}) - \Layer{RN}(\Matrix{x})| \\
    &\leq \delta_{\Layer{RN}} |\Layer{RN}(\Matrix{\hat{x}})| + |\Layer{RN}(\Matrix{\hat{x}}) - \Layer{RN}(\Matrix{x})| \\
    &\leq \delta_{\Layer{RN}} |\Layer{RN}(\Matrix{x})| + (1 + \delta_{\Layer{RN}}) |\Matrix{J}_{\Layer{RN}}(\Matrix{x}) (\Matrix{\hat{x}} - \Matrix{x})| + \Oh{\Norm{\Matrix{\hat{x}} - \Matrix{x}}{\infty}^2},
\end{align*}
and we take the $\ell_\infty$ norm of both sides.
Next, we invoke Theorem~\ref{theorem:cond-ffn-gelu} and \eqref{eq:cond-ffn-gelu-universal} to get
\begin{align*}
    \Norm{\fl{\Layer{F}(\Matrix{\hat{x}})} - \Layer{F}(\Matrix{x})}{\infty} \leq \delta_{\Layer{F}} \omega_{u,d} \Norm{\Matrix{\hat{x}}}{\infty} + \Norm{\sigma'}{C(\Real)} \omega_{u,d} \Norm{\Matrix{\hat{x}} - \Matrix{x}}{\infty} + \Oh{\Norm{\Matrix{\hat{x}} - \Matrix{x}}{\infty}^2},
\end{align*}
where we use the invariance of the $\ell_\infty \to \ell_\infty$ norm to the signs of entries,
and note that $\Norm{\sigma'}{C(\Real)} < 8/7$ for the approximate GELU \eqref{eq:gelu-tanh}.
The third statement is obtained by substituting the first bound into the second one.
For the final bound, consider
\begin{equation*}
    |\fl{\Layer{B}_{\Layer{F}}(\Matrix{\hat{x}})} - \Layer{B}_{\Layer{F}}(\Matrix{x})| \leq \uh_{r} \big(|\Matrix{x}| + |\Layer{F}(\Layer{RN}(\Matrix{x}))|\big) + (1 + \uh_r) |\Matrix{\hat{x}} - \Matrix{x}| + (1 + \uh_r) |\fl{\Layer{F}(\Layer{RN}(\Matrix{\hat{x}}))} - \Layer{F}(\Layer{RN}(\Matrix{x}))|,
\end{equation*}
take the $\ell_\infty$ norm, and substitute the previous bounds.
\end{proof}

\begin{proposition}
\label{proposition:rounding-attn-block}
Let $\Matrix{X}, \Matrix{\hat{X}} \in \Real^{d \times N}$.
\begin{enumerate}
    \item The value of $\Layer{A}(\Matrix{\hat{X}})$ computed in FP arithmetic satisfies
    \begin{align*}
        \Norm{\fl{\Layer{A}(\Matrix{\hat{X}})} - \Layer{A}(\Matrix{X})}{1, \infty} &\leq \delta_{\Layer{OVS}} \omega_{o,v} \Norm{\Matrix{X}}{1,\infty} + \delta_{\Layer{KQ}} \omega_{o,v} \omega_{k,q} \Norm{\Matrix{X}}{1,\infty}^3 \\
        &+ \big( 1 + \delta_{\Layer{OVS}} + (1 + 3 \delta_{\Layer{KQ}}) \omega_{k,q} \Norm{\Matrix{X}}{1,\infty}^2 \big) \omega_{o,v} \Norm{\Matrix{\hat{X}} - \Matrix{X}}{1,\infty} + \Oh{\Norm{\Matrix{\hat{X}} - \Matrix{X}}{1,\infty}^2},
    \end{align*}
    where
    \begin{gather*}
        \delta_{\Layer{OVS}} = 6 \uh_{q} + 2\gamma_{MM}(d) + \gamma_{MM}(N) + 3 \uh_{w} + \gamma_{MM}(N, \uh_w) + \Oh{\uh_q^2}, \\
        \delta_{\Layer{KQ}} = 10 \uh_{q} + 5\uh_{a} + \gamma_{MM}(d_{head}) + 2\gamma_{MM}(d) + \Oh{\uh_q^2}, \\
        \omega_{o,v} = \Norm{\Matrix{W}_O}{\infty,\infty} \max_{1 \leq h \leq n_{head}} \Norm{\Matrix{W}_{V,h}}{\infty,\infty}, \quad \omega_{k,q} = 2 d_{head}^{-1/2} \max_{1 \leq h \leq n_{head}} \Norm{|\Matrix{W}_{K,h}|}{\infty,2} \Norm{|\Matrix{W}_{Q,h}|}{\infty,2}.
    \end{gather*}
    \item The value of $\Layer{A}(\Layer{RN}^*(\Matrix{\hat{X}}))$ computed in FP arithmetic satisfies
    \begin{multline*}
        \Norm{\fl{\Layer{A}(\Layer{RN}^*(\Matrix{\hat{X}}))} - \Layer{A}(\Layer{RN}^*(\Matrix{X}))}{1,\infty} \leq \big( \delta_{\Layer{OVS}} + \delta_{\Layer{RN}} (1 + \delta_{\Layer{OVS}}) \big) \omega_{o,v} \Norm{\Matrix{Y}}{1,\infty} \\
        + \big( \delta_{\Layer{KQ}} + \delta_{\Layer{RN}} (1 + 3\delta_{\Layer{KQ}}) \big) \omega_{o,v} \omega_{k,q} \Norm{\Matrix{Y}}{1,\infty}^3 \\
        + \big( 1 + \delta_{\Layer{OVS}} + (1 + 3 \delta_{\Layer{KQ}}) \omega_{k,q} \Norm{\Matrix{Y}}{1,\infty}^2 \big) (1 + \delta_{\Layer{RN}}) \Tilde{\kappa}(\Layer{RN},\Matrix{X}) \omega_{o,v} \Norm{\Matrix{\hat{X}} - \Matrix{X}}{1,\infty} + \Oh{\Norm{\Matrix{\hat{X}} - \Matrix{X}}{1,\infty}^2}.
    \end{multline*}
    where $\Matrix{Y} = \Layer{RN}^*(\Matrix{X})$ and $\Tilde{\kappa}(\Layer{RN},\Matrix{X}) = \max_{1 \leq n \leq N} \condabs{\infty,\infty}{\Layer{RN},\Matrix{x}_n}$.
    \item The value of $\Layer{B}_{\Layer{A}}(\Matrix{\hat{X}}) = \Matrix{\hat{X}} + \Layer{A}(\Layer{RN}^*(\Matrix{\hat{X}}))$ computed in FP arithmetic satisfies
    \begin{multline*}
        \Norm{\fl{\Layer{B}_{\Layer{A}}(\Matrix{\hat{X}})} - \Layer{B}_{\Layer{A}}(\Matrix{X})}{1,\infty} \leq \uh_r \Norm{\Matrix{X}}{1,\infty} + \Big[ \uh_r + (1 + \uh_r) \big( \delta_{\Layer{OVS}} + \delta_{\Layer{RN}} (1 + \delta_{\Layer{OVS}}) \big) \Big] \omega_{o,v} \Norm{\Matrix{Y}}{1,\infty} \\
        + (1 + \uh_r) \big( \delta_{\Layer{KQ}} + \delta_{\Layer{RN}} (1 + 3\delta_{\Layer{KQ}}) \big) \omega_{o,v} \omega_{k,q} \Norm{\Matrix{Y}}{1,\infty}^3 \\
        + (1 + \uh_r) \Big[1 + \big( 1 + \delta_{\Layer{OVS}} + (1 + 3 \delta_{\Layer{KQ}}) \omega_{k,q} \Norm{\Matrix{Y}}{1,\infty}^2 \big) (1 + \delta_{\Layer{RN}}) \Tilde{\kappa}(\Layer{RN},\Matrix{X}) \omega_{o,v} \Big] \Norm{\Matrix{\hat{X}} - \Matrix{X}}{1,\infty} + \Oh{\Norm{\Matrix{\hat{X}} - \Matrix{X}}{1,\infty}^2}.
    \end{multline*}
\end{enumerate}
\end{proposition}
\begin{proof}
The proof proceeds in completely analogy with Proposition~\ref{proposition:rounding-ff-block}, relying on Theorems~\ref{theorem:condabs-attention} and~\ref{theorem:rounding-attn}.
In addition, we bound $\tilde{\kappa}_h(\Layer{S})$ with one from above by Theorem~\ref{theorem:cond-softmax}.
\end{proof}

Let us assume that the attention and FF residual blocks in \eqref{eq:transformer-RN} are identical for all $1 \leq l \leq L$ and that the same gain and stabilisation parameters are used in the two instances of $\Layer{RN}$.
For every $1 \leq m \leq 2L$, denote by $\Matrix{X}_m \in \Real^{d \times N}$ the exact output of the $m$th block in \eqref{eq:transformer-RN}, so that $\Matrix{X}_{2L} = \Layer{T}(\Matrix{X})$, and let $\Matrix{X}_0 = \Matrix{X}$.
Note that $\max_{0 \leq m \leq 2L - 1} \Norm{\Layer{RN}^*(\Matrix{X}_m)}{1,\infty} \leq \Norm{\Matrix{g}}{\infty}$.
Furthermore, it readily holds that
\begin{equation*}
    \Norm{\Matrix{X}_m}{1,\infty} \leq \Norm{\Matrix{X}_0}{1,\infty} + \Big( \Big\lceil \frac{m}{2} \Big\rceil \omega_{o,v} + \Big\lfloor \frac{m}{2} \Big\rfloor \omega_{u,d} \Big) \Norm{\Matrix{g}}{\infty},
\end{equation*}
and hence the norm of the residual stream grows at most linearly with depth.
Meanwhile, to bound the forward error of transformer inference, we also require a \emph{lower} bound on the residual stream in order to be able to bound $\tilde{\kappa}(\Layer{RN}, \Matrix{X}_m)$ from above.
To this end, we introduce the following polynomial model
\begin{equation}
\label{eq:residual-growth-model}
    \min_{1 \leq n \leq N} \Norm{\Matrix{x}_{m,n}}{2} \geq c m^{\upsilon}, \quad c > 0, \quad 0 \leq \upsilon \leq 1, \quad m \geq 1,
\end{equation}
which ensures that every column of $\Matrix{X}_m$ stays non-zero and perhaps grows in norm with depth.
Assuming that \eqref{eq:residual-growth-model} holds, we refer to the bound \eqref{eq:cond-rn-universal-bound} to obtain
\begin{equation*}
    \tilde{\kappa}(\Layer{RN}, \Matrix{X}_m) = \max_n \condabs{\infty,\infty}{\Layer{RN}, \Matrix{x}_{m,n}} \leq \frac{\sqrt{d} + 1}{2} \Norm{\Matrix{g}}{\infty} \max_n \frac{1}{\sqrt{\Norm{\Matrix{x}_{k,n}}{2}^2 + \epsilon}} \leq \frac{\sqrt{d} + 1}{2c} \Norm{\Matrix{g}}{\infty} m^{-\upsilon}.
\end{equation*}

\begin{theorem}
\label{theorem:rounding-deep}
Let $\Matrix{X}_0 \in \Real^{d \times N}$.
Assume that \eqref{eq:residual-growth-model} holds and denote
\begin{gather*}
    \Delta_{\Layer{A}} = (\uh_r + \delta_{\Layer{OVS}} + \delta_{\Layer{RN}}) \Norm{\Matrix{g}}{\infty} + (\delta_{\Layer{KQ}} + \delta_{\Layer{RN}}) \omega_{k,q} \Norm{\Matrix{g}}{\infty}^3, \quad \Delta_{\Layer{F}} = (\uh_r + \delta_{\Layer{F}} + \tfrac{8}{7} \delta_{\Layer{RN}}) \Norm{\Matrix{g}}{\infty}, \\
    C_{o,v} = \frac{\sqrt{d} + 1}{2} \frac{\Norm{\Matrix{g}}{\infty}}{c} (1 + \omega_{k,q} \Norm{\Matrix{g}}{\infty}^2), \quad C_{u,d} = \frac{8}{7} \frac{\sqrt{d} + 1}{2} \frac{\Norm{\Matrix{g}}{\infty}}{c}.
\end{gather*}
Furthermore, for $1 \leq m \leq 2L$, denote 
\begin{gather*}
    \Lambda_m = 1 + m^{-\upsilon} \begin{cases}
        C_{o,v} \omega_{o,v}, & m~\text{is odd}, \\
        C_{u,d} \omega_{u,d}, & m~\text{is even},
    \end{cases} \\
    \Psi_m = \uh_r \Norm{\Matrix{X}_0}{1,\infty} + \uh_r \Big( \Big\lceil \frac{m-1}{2} \Big\rceil \omega_{o,v} + \Big\lfloor \frac{m-1}{2} \Big\rfloor \omega_{u,d} \Big) \Norm{\Matrix{g}}{\infty} + \begin{cases}
        \Delta_{\Layer{A}} \omega_{o,v}, & m~\text{is odd}, \\
        \Delta_{\Layer{F}} \omega_{u,d}, & m~\text{is even}.
    \end{cases}
\end{gather*}
The value of $\Layer{T}(\Matrix{X}_0)$ computed in FP arithmetic satisfies
\begin{equation*}
    \Norm{\fl{\Layer{T}(\Matrix{X}_0)} - \Layer{T}(\Matrix{X}_0)}{1,\infty} \leq \sum_{m = 1}^{2L} \Big( \prod_{l = m + 1}^{2L} \Lambda_l \Big) \Psi_m + \Oh{\uh_q^2}.
\end{equation*}
\end{theorem}
\begin{proof}
Let $\Matrix{\hat{X}}_m$ be the value of $\Matrix{X}_m$ computed in FP arithmetic.
Propositions~\ref{proposition:rounding-ff-block} and~\ref{proposition:rounding-attn-block} yield
\begin{equation*}
    \Norm{\Matrix{\hat{X}}_m - \Matrix{X}_m}{1, \infty} \leq \Psi_m + \Lambda_m \Norm{\Matrix{\hat{X}}_{m-1} - \Matrix{X}_{m-1}}{1,\infty} + \Oh{\uh_q^2}, \quad \Matrix{\hat{X}}_0 = \Matrix{X}_0.
\end{equation*}
Unrolling the recursion, we get the final bound at $m = 2L$.
\end{proof}

We shall now study this global forward-error bound by explicitly taking into account the dependence of $\Lambda_m$ and $\Psi_m$ on $m$.
Let $\chi = \max\{ C_{o,v} \omega_{o,v}, C_{u,d} \omega_{u,d} \}$ and consider the product:
\begin{equation*}
    \prod_{l = m + 1}^{2L} \Lambda_i \leq \prod_{l = m + 1}^{2L} (1 + \chi l^{-\upsilon}) \leq \exp\Big( \chi \sum_{l = m + 1}^{2L} l^{-\upsilon} \Big) \leq \exp \Big( \chi \int_{m}^{2L} x^{-\upsilon} dx \Big),
\end{equation*}
where we use the monotonicity of $x^{-\upsilon}$ in the last inequality.
Consequently, it holds that
\begin{equation*}
    \prod_{l = m + 1}^{2L} \Lambda_l \leq \begin{cases}
        \exp\Big( \frac{\chi}{1 - \upsilon} \big[ (2L)^{1 - \upsilon} - m^{1 - \upsilon} \big] \Big), & \upsilon \neq 1, \\
        \Big( \frac{2L}{m} \Big)^{\chi}, & \upsilon = 1.
    \end{cases}
\end{equation*}

Let us begin with the $\upsilon = 0$ regime, whereby the residual stream does not necessarily grow.
Then $\prod_{l = m + 1}^{2L} \Lambda_l \leq \exp( \chi (2L - m))$, and with $\Psi_m \leq m \psi + \phi$, Theorem~\ref{theorem:rounding-deep} guarantees for $\chi > 0$ that
\begin{equation*}
    \Norm{\Matrix{\hat{X}}_{2L} - \Matrix{X}_{2L}}{1, \infty} \leq e^{2 \chi L} \sum_{m = 1}^{2L} e^{-m \chi} (m \psi + \phi) \leq e^{2 \chi L} \sum_{m = 1}^{\infty} e^{-m \chi} (m \psi + \phi)
\end{equation*}
grows as $\exp(2 \chi L)$.
To counter the exponential growth of the error, it has become a practical rule of thumb to scale the weights in accordance with depth.
The scaling law $\chi = \hat{\chi} / (2L)$ ensures that the error is bounded and applies to any growth rate of the residual stream, and hence is theoretically universal.
However, such a reduction in weight magnitude impedes training and is avoided in practice in favour of relying on the natural growth of the residual stream.

Let $\upsilon = 0.5$, which is typical for transformers at the early stages of training:
\begin{equation*}
    \Norm{\Matrix{\hat{X}}_{2L} - \Matrix{X}_{2L}}{1, \infty} \leq e^{2 \chi \sqrt{2L}} \sum_{m = 1}^{2L} e^{-2 \sqrt{m} \chi} (m \psi + \phi) \leq e^{2 \chi \sqrt{2L}} \sum_{m = 1}^{\infty} e^{-2 \sqrt{m} \chi} (m \psi + \phi).
\end{equation*}
The growth is now sub-exponential and can be tamed with the requirement $\chi = \hat{\chi} / \sqrt{2L}$.
This scaling law was first introduced for the training of GPT-2 \citep{radford2019language} and has become the standard initialisation technique for modern LLMs (Table~\ref{tab:LLM-sizes}).

At $\upsilon = 1$, the lower and upper bounds on the residual stream agree and $\Matrix{X}_m$ gets consistently pushed in a single direction.
(This behaviour is correlated with the emergence of featurewise massive outliers.)
In this case, the amplification of local errors becomes polynomial rather than (sub-)exponential, yielding
\begin{equation*}
    \Norm{\Matrix{\hat{X}}_{2L} - \Matrix{X}_{2L}}{1, \infty} \leq (2L)^\chi \sum_{m = 1}^{2L} m^{-\chi} (m \psi + \phi).
\end{equation*}
We bound the right-hand side by estimating the sums $\sum_m m^{1 - \chi}$ and $\sum_m m^{-\chi}$ with integrals:\footnote{We use $\sum_{m = 1}^{2L} m^p \leq \int_{1}^{2L+1} x^p dx$ when $p \geq 0$ and $\sum_{m = 1}^{2L} m^p \leq 1 + \int_{1}^{2L} x^p dx$ when $p < 0$.}
\begin{equation*}
    \Norm{\Matrix{\hat{X}}_{2L} - \Matrix{X}_{2L}}{1, \infty} \leq \begin{cases}
        (2L)^\chi (\frac{\chi - 1}{\chi - 2} \psi + \frac{\chi}{\chi - 1} \phi), & \chi > 2, \\
        (2L)^2 \big[(1 + \ln(2L)) \psi + (2 - \frac{1}{2L}) \phi\big], & \chi = 2, \\
        \big[(2L)^\chi + \frac{(2L)^2 - (2L)^\chi}{2 - \chi}\big] \psi + \frac{\chi (2L)^{\chi} - 2L}{\chi - 1} \phi, & 1 < \chi < 2, \\
        (2L)^2 \psi + (2L) (1 + \ln(2L)) \phi, & \chi = 1, \\
        (2L + 1)^2 (\frac{1}{2 - \chi} \psi + \frac{2}{1 - \chi} \phi), & \chi < 1.
    \end{cases}
\end{equation*}
Therefore, the error bound grows polynomially with depth; in particular, quadratically when $\chi < 2$.
The case of $\upsilon = 1$ does not lead to a scaling law as the threshold value of $\chi$ is independent of $L$.
At the same time, the growth rate of our error bound cannot be further improved no matter how small $\chi$ is.

Let us examine the constant $c$ from the residual-stream bound, to which the parameter $\chi$ is inversely proportional.
Consider a uniform scaling of the residual-projection weights $\Matrix{W}_O$ and $\Matrix{W}_{down}$ by a scalar $\xi$.
If the residual stream responds linearly $(c \mapsto \xi c)$, the scaling factor perfectly cancels within $\chi$---and hence the absolute forward error scales by $\xi$ through the $\Psi_m$ term, leaving the relative error unchanged.
Therefore, Theorem~\ref{theorem:rounding-deep} suggests that a necessary condition for the scaling to impact the numerical stability of inference is that it must force a qualitative transition in the dynamic of the residual stream.

\subsection{Computational experiments}
To compare the bound in Theorem~\ref{theorem:rounding-deep} with the empirical behaviour of the rounding error in transformer inference, we select the GPT-2 XL model (see \citep{radford2019language} and Appendix~\ref{appendix:llm_comparison}).
Specifically, we consider two instances of the model: with pre-trained and randomly initialised weights.
For the latter, every weight matrix $\Matrix{W} \in \Real^{d_1 \times d_2}$ is populated with independent and identically distributed random Gaussian entries with zero mean and standard deviation of $1 / \sqrt{d_1}$; additionally, the matrices $\Matrix{W}_O$ and $\Matrix{W}_{down}$ are further scaled by $1 / \sqrt{2L}$, the gain of layer normalisation is set to $\One$, and its bias to zero.
Note that standard GPT-2 inference is performed in FP32.\footnote{The code used for the experiments is available at \url{https://github.com/sbudzinskiy/lowprec-llm-inference}.}

\begin{figure}[t]
\centering
\includegraphics[width=0.8\linewidth]{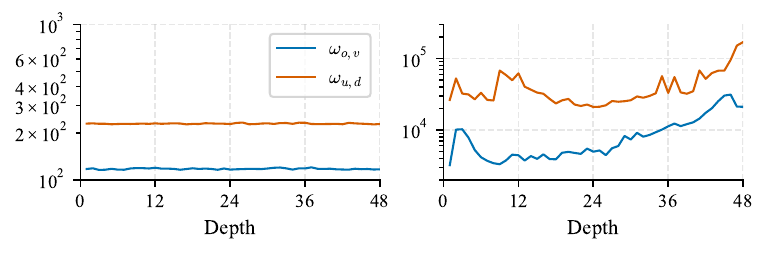}
\caption{Comparison of $\omega_{o,v}$ and $\omega_{u,d}$ from Theorem~\ref{theorem:rounding-attn} for the GPT-2 XL architecture with randomly initialised weights (left) and pre-trained weights (right).}
\label{fig:omega_comparison}
\end{figure}

First, we compute the values of $\omega_{o,v}$ and $\omega_{u,d}$, which now differ across blocks, for the two GPT-2 instances.
The results in Figure~\ref{fig:omega_comparison} show that they are \emph{large}, making our parameter $\chi \gg 1$, and leading to extremely fast amplification of the error in our bounds regardless of the residual-stream growth rate $\upsilon$.

Note that the values of $\omega_{o,v}$ and $\omega_{u,d}$ for the pre-trained model are $1.5$--$2$ orders of magnitude larger than those of the randomly initialised model.
Since pre-trained models tend to operate in the massive-outlier regime, the error bound in Theorem~\ref{theorem:rounding-attn} can be controlled more tightly based on
\begin{equation*}
    \tilde{\omega}_{u,d} = \Norm{\Matrix{W}_{down}}{\infty,\infty} \Norm{\Matrix{W}_{up}}{1,\infty}, \quad \tilde{\omega}_{o,v} = \Norm{\Matrix{W}_O}{\infty,\infty} \max_{1 \leq h \leq n_{head}} \Norm{\Matrix{W}_{V,h}}{1,\infty},
\end{equation*}
by taking the columns of $\Matrix{W}_{up}$ and $\Matrix{W}_{V,h}$ corresponding to the outlier (cf. Theorem~\ref{theorem:condabs-attention}).
Assuming the absolute values of the weights are approximately uniform, we have $\omega_{u,d} \approx d \tilde{\omega}_{u,d}$ and $\omega_{o,v} \approx d_{head} \tilde{\omega}_{o,v}$, bringing the initialised and pre-trained models closer together.
Still, the empirically observed values in Figure~\ref{fig:omega_comparison} render the worst-case upper bound of Theorem~\ref{theorem:rounding-attn} overly pessimistic and practically vacuous.

\begin{figure}[t]
\centering
\includegraphics[width=0.8\linewidth]{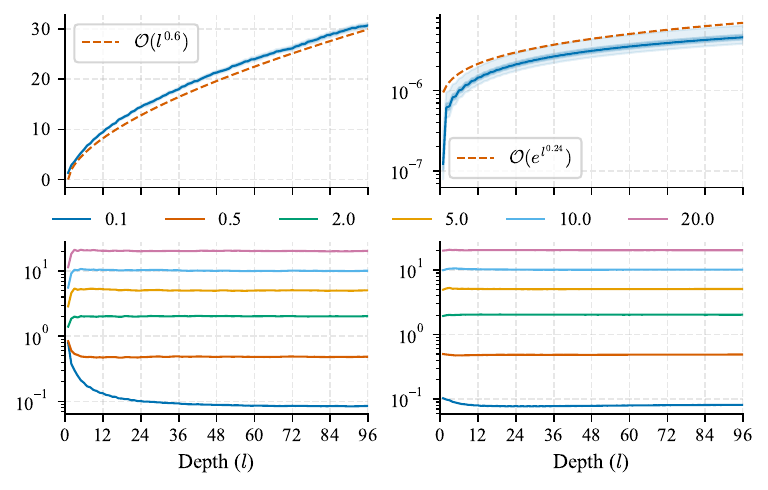}
\caption{Inference dynamics for GPT-2 XL with randomly initialised weights.
Left: magnitude of the residual stream according to \eqref{eq:residual-growth-model}.
Right: absolute $\ell_1 \to \ell_{\infty}$ normwise rounding error.
Top: median, $75$th, and $99$th percentiles over 100 sequences of 1024 tokens.
Bottom: ratio of medians for a model subjected to residual-projection weight scaling relative to the original model.
}
\label{fig:residual_random}
\end{figure}

\begin{figure}[h]
\centering
\includegraphics[width=0.8\linewidth]{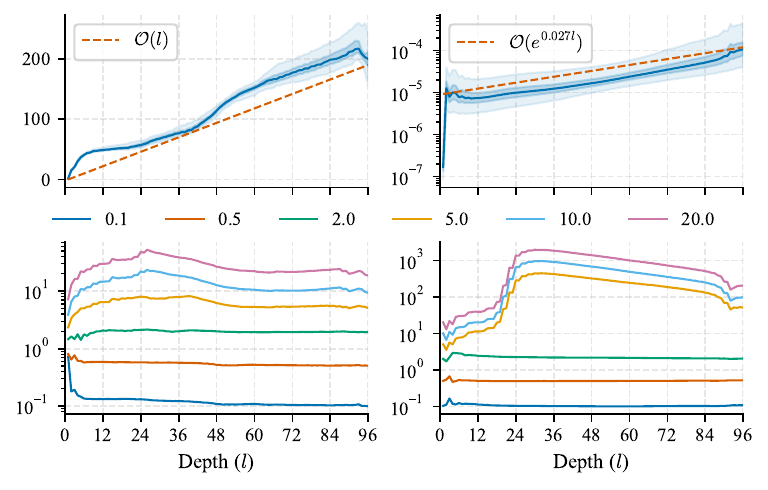}
\caption{Inference dynamics for GPT-2 XL with pre-trained weights.
Left: magnitude of the residual stream according to \eqref{eq:residual-growth-model}.
Right: absolute $\ell_1 \to \ell_{\infty}$ normwise rounding error.
Top: median, $75$th, and $99$th percentiles over 100 sequences of 1024 tokens.
Bottom: ratio of medians for a model subjected to residual-projection weight scaling relative to the original model.
}
\label{fig:residual_trained}
\end{figure}

To evaluate the actual rounding errors of standard GPT-2 inference, we use an FP64 implementation of the model as the ground truth and gather statistics over 100 sequences of 1024 tokens each from the OpenWebText\footnote{\url{https://huggingface.co/datasets/Skylion007/openwebtext}} dataset.
We present the results for the randomly initialised model in Figure~\ref{fig:residual_random} and for the pre-trained model in Figure~\ref{fig:residual_trained}.
In both figures, the top row depicts the growth of the residual stream (left) and of the absolute $\ell_\infty$-norm error (right).
For the initialised model, Theorem~\ref{theorem:rounding-attn} correctly predicts the sub-exponential nature of the error growth, even though the empirical error accumulates at a much milder rate of $\Oh{\exp(l^{0.24})}$ than the theoretical bound of $\Oh{\exp(\chi l^{0.4})}$.
In contrast, the pre-trained model exhibits slow exponential growth of the absolute error; the discrepancy with the polynomial-rate prediction in Theorem~\ref{theorem:rounding-attn} likely stems from the fact that the magnitude of pre-trained weights varies with depth, contrary to the theoretical assumptions of our theorem.

The bottom row in Figures~\ref{fig:residual_random} and~\ref{fig:residual_trained} presents the impact of residual-projection weight scaling, where we multiply $\Matrix{W}_O$ and $\Matrix{W}_{down}$ by a scalar $\xi$ and plot the ratios
\begin{equation*}
    \frac{\min_{1 \leq n \leq N} \Norm{\Matrix{x}_{l,n}^{(\xi)}}{2}}{\min_{1 \leq n \leq N} \Norm{\Matrix{x}_{l,n}^{(1)}}{2}}, \quad \frac{\Norm{\Matrix{\hat{X}}_{l}^{(\xi)} - \Matrix{X}_{l}^{(\xi)}}{1, \infty}}{\Norm{\Matrix{\hat{X}}_{l}^{(1)} - \Matrix{X}_{l}^{(1)}}{1, \infty}}, \quad 1 \leq l \leq 2L,
\end{equation*}
for different values of $\xi$.
In accordance with the structural properties of the error bound in Theorem~\ref{theorem:rounding-attn}, the absolute error scales linearly with $\xi$ for the randomly initialised model.
In contrast, the pre-trained model exhibits a qualitative transition of the residual-stream dynamic between $\xi = 2$ and $\xi = 5$, leading to a non-linear response of the absolute error.
Once the scaling exceeds $\xi = 5$, the model settles into a new regime, and we observe a return to strictly proportional scaling of the absolute error.

\begin{figure}[h]
\centering
\includegraphics[width=0.8\linewidth]{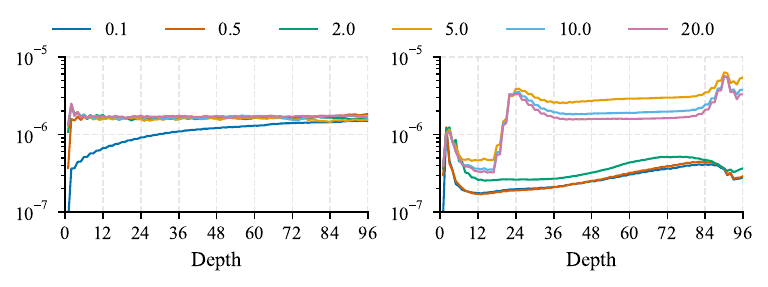}
\caption{Inference dynamics of the relative $\ell_1 \to \ell_{\infty}$ normwise rounding error for GPT-2 XL with randomly initialised weights (left) and pre-trained weights (right).
The median error is computed over 100 sequences of 1024 tokens for models subjected to residual-projection weight scaling.
}
\label{fig:rel_err}
\end{figure}

The results in Figure~\ref{fig:rel_err} validate our theoretical hypothesis formulated at the end of Subsection~\ref{subsec:transformer-theory}: residual-projection weight scaling preserves the dynamics of the relative rounding error of inference, unless the scaling causes the underlying residual stream to bifurcate into a qualitatively distinct regime.

\section*{Acknowledgements}
We are grateful to the referees for motivating us to significantly deepen the analysis.

\section*{Funding}
This work was carried out in the framework of a research project funded by Huawei Technologies Ltd.

\bibliographystyle{abbrvnat}
\bibliography{common/bib}

\appendix
\renewcommand{\theHsection}{A\arabic{section}}

\section{Comparison of transformer architectures across generations of LLMs}
\label{appendix:llm_comparison}

Architectural details presented in the tables below are based on technical reports and verified against open-source implementations (Hugging Face).
A diverse selection of open-source models is sorted by release date.
Note that Falcon 2 \citep{malartic2024falcon2} is a parallel architecture.
We describe the choice of layer normalisation (Table~\ref{tab:LLM-LN}), FF mechanism (Table~\ref{tab:LLM-FFN}), and attention mechanism (Table~\ref{tab:LLM-ATTN}).
In Table~\ref{tab:LLM-sizes}, the model sizes are presented; note that some models violate the standard design principles described in Section~\ref{sec:background}.
Additionally, we present in Table~\ref{tab:LLM-sizes} the weight-scaling used at initialisation.

\begin{table}[th]
\caption{Layer normalisation in LLMs}
\label{tab:LLM-LN}
\begin{tabular}{@{}lcccc@{}}
\toprule
Model & Variant & Type & Gain & Bias \\
\midrule
Vanilla \citep{vaswani2017attention}  & $\Layer{LN}$ & Post & + & + \\
BERT \citep{devlin2019bert}           & $\Layer{LN}$ & Post & + & + \\
GPT-2 \citep{radford2019language}     & $\Layer{LN}$ & Pre  & + & + \\
OPT \citep{zhang2022opt}              & $\Layer{LN}$ & Pre  & + & + \\
BLOOM \citep{scao2022bloom}           & $\Layer{LN}$ & Pre  & + & + \\
Mixtral 8x7B \citep{jiang2024mixtral} & $\Layer{RN}$ & Pre  & + & - \\
Llama 3 \citep{grattafiori2024llama}  & $\Layer{RN}$ & Pre  & + & - \\
Falcon 2 \citep{malartic2024falcon2}  & $\Layer{LN}$ & Pre  & + & + \\
Gemma 2 \citep{riviere2024gemma}      & $\Layer{RN}$ & Pre \& Post-Sublayer & + & - \\
Qwen2.5 \citep{yang2024qwen25}        & $\Layer{RN}$ & Pre  & + & - \\
OLMo 2 \citep{walsh20252}             & $\Layer{RN}$ & Post-Sublayer & + & - \\
DeepSeek-v3 \citep{liu2024deepseek}   & $\Layer{RN}$ & Pre  & + & - \\
Ministral 3 \citep{liu2026mistral}    & $\Layer{RN}$ & Pre  & + & - \\
\botrule
\end{tabular}
\end{table}

\begin{table}[th]
\caption{Feedforward mechanism in LLMs}
\label{tab:LLM-FFN}
\begin{tabular}{@{}lccc@{}}
\toprule
Model & Variant & Type & Bias \\
\midrule
Vanilla \citep{vaswani2017attention}  & ReLU   & Dense & + \\
BERT \citep{devlin2019bert}           & GELU   & Dense & + \\
GPT-2 \citep{radford2019language}     & GELU   & Dense & + \\
OPT \citep{zhang2022opt}              & ReLU   & Dense & + \\
BLOOM \citep{scao2022bloom}           & GELU   & Dense & + \\
Mixtral 8x7B \citep{jiang2024mixtral} & SwiGLU & MoE   & - \\
Llama 3 \citep{grattafiori2024llama}  & SwiGLU & Dense & - \\
Falcon 2 \citep{malartic2024falcon2}  & SwiGLU & Dense & - \\
Gemma 2 \citep{riviere2024gemma}      & GeGLU  & Dense & - \\
Qwen2.5 \citep{yang2024qwen25}        & SwiGLU & Dense & - \\
OLMo 2 \citep{walsh20252}             & SwiGLU & Dense & - \\
DeepSeek-v3 \citep{liu2024deepseek}   & SwiGLU & MoE   & - \\
Ministral 3 \citep{liu2026mistral}    & SwiGLU & Dense & - \\
\botrule
\end{tabular}
\end{table}

\begin{table}[th]
\caption{Attention mechanism in LLMs}
\label{tab:LLM-ATTN}
\begin{tabular}{@{}lccc@{}}
\toprule
Model & Variant & Pos. Enc. & Bias  \\
\midrule
Vanilla \citep{vaswani2017attention}  & Standard      & Abs.  & + \\
BERT \citep{devlin2019bert}           & Standard      & Abs.  & + \\
GPT-2 \citep{radford2019language}     & Standard      & Abs.  & + \\
OPT \citep{zhang2022opt}              & Standard      & Abs.  & + \\
BLOOM \citep{scao2022bloom}           & Standard      & ALiBi & + \\
Mixtral 8x7B \citep{jiang2024mixtral} & Grouped-Query & RoPE  & - \\
Llama 3 \citep{grattafiori2024llama}  & Grouped-Query & RoPE  & - \\
Falcon 2 \citep{malartic2024falcon2}  & Grouped-Query & RoPE  & - \\
Gemma 2 \citep{riviere2024gemma}      & Grouped-Query & RoPE  & - \\
Qwen2.5 \citep{yang2024qwen25}        & Grouped-Query & RoPE  & + \\
OLMo 2 \citep{walsh20252}             & Grouped-Query & RoPE  & - \\
DeepSeek-v3 \citep{liu2024deepseek}   & Latent        & RoPE  & - \\
Ministral 3 \citep{liu2026mistral}    & Grouped-Query & RoPE  & - \\
\botrule
\end{tabular}
\end{table}

\begin{table}[th]
\caption{Sizes and initialisation scaling of LLMs}
\label{tab:LLM-sizes}
\begin{tabular}{@{}lcccccc@{}}
\toprule
Model & $d$ & $D$ & $n_{head}$ & $d_{head}$ & $2L$ & Scaling of $\Matrix{W}_O$ and $\Matrix{W}_{down}$ \\
\midrule
Vanilla \citep{vaswani2017attention}  & 1024  & 4096  & 16  & 64  & 12  & None \\
BERT \citep{devlin2019bert}           & 1024  & 4096  & 16  & 64  & 48  & None \\
GPT-2 \citep{radford2019language}     & 1600  & 6400  & 25  & 64  & 96  & $1 / \sqrt{2L}$ \\
OPT \citep{zhang2022opt}              & 12288 & 49152 & 96  & 128 & 192 & $1 / \sqrt{2L}$ \\
BLOOM \citep{scao2022bloom}           & 14336 & 57344 & 112 & 128 & 140 & $1 / \sqrt{2L}$ \\
Mixtral 8x7B \citep{jiang2024mixtral} & 4096  & 14336 & 32  & 128 & 64  & $1 / \sqrt{2L}$ \\
Llama 3 \citep{grattafiori2024llama}  & 16384 & 53248 & 128 & 128 & 252 & $1 / \sqrt{2L}$ \\
Falcon 2 \citep{malartic2024falcon2}  & 4096  & 16384 & 32  & 128 & 120 & $1 / \sqrt{2L}$ \\
Gemma 2 \citep{riviere2024gemma}      & 4608  & 36864 & 32  & 128 & 92  & $1 / \sqrt{2L}$ \\
Qwen2.5 \citep{yang2024qwen25}        & 8192  & 29568 & 64  & 128 & 160 & $1 / \sqrt{2L}$ \\
OLMo 2 \citep{walsh20252}             & 5120  & 27648 & 40  & 128 & 128 & $1 / \sqrt{2L}$ \\
DeepSeek-v3 \citep{liu2024deepseek}   & 7168  & 2048  & 128 & 128 & 122 & $1 / \sqrt{2L}$ \\
Ministral 3 \citep{liu2026mistral}    & 5120  & 16384 & 32  & 128 & 80  & $1 / \sqrt{2L}$ \\
\botrule
\end{tabular}
\end{table}

\section{Layer normalisation: Condition number comparison in Subsection~\ref{subsec:layernorm_comparison}}
\label{appendix:layernorm-cond-comparison}

The expressions of condition numbers below are derived based on Theorems~\ref{theorem:cond-rms-layernorm} and~\ref{theorem:cond-layernorm}.
Recall that we set $\Matrix{g} = g \One$ for the two massive-outlier examples.
We mark with $\dagger$ the formulas where the row-maximum in the definition of the corresponding condition number is attained at the outlier index $i = 1$.

\subsection{Massive outlier with zero-variance background}
The exact condition numbers of the $\Layer{RN}$ normalisation function \eqref{eq:rms-layernorm} are
\begin{gather*}
    \cond{\infty, \infty}{\Layer{RN}, \Matrix{x}} = 1 + \begin{cases}
        \frac{(d-3) \alpha^2 + |\alpha|}{1 + (d-1) \alpha^2 + \epsilon}, \quad & |\alpha| \leq \tfrac{1}{d-3}, \\
        \frac{(d-1) |\alpha| - 1}{1 + (d-1) \alpha^2 + \epsilon}, \quad & \text{otherwise},^\dagger
    \end{cases}\\
    \cond{\infty,c}{\Layer{RN},\Matrix{x}} = \frac{1}{|\alpha|} + \frac{1 + (d-3) |\alpha|}{1 + (d-1) \alpha^2 + \epsilon}, \\
    \cond{c,\infty}{\Layer{RN},\Matrix{x}} = \begin{cases}
        2 |\alpha| \frac{1 + (d-2) \alpha^2 + \epsilon/2}{1 + (d-1) \alpha^2 + \epsilon}, & |\alpha| < \frac{1}{d-2}~\text{and}~\epsilon < 2 |\alpha| (1 - (d-2) |\alpha|), \\
        \frac{2(d-1) \alpha^2 + \epsilon}{1 + (d-1) \alpha^2 + \epsilon}, & \text{otherwise},^\dagger
    \end{cases}\\
    \cond{c,c}{\Layer{RN}, \Matrix{x}} = 2 \frac{1 + (d-2) \alpha^2 + \epsilon/2}{1 + (d-1) \alpha^2 + \epsilon}.
\end{gather*}
The exact condition numbers of the $\Layer{LN}$ normalisation function \eqref{eq:layernorm} are
\begin{gather*}
    \cond{\infty,\infty}{\Layer{LN}, \Matrix{x}} = \frac{2d}{(d-1)(1-\alpha)} \frac{(d-2)(1-\alpha)^2 + (d-1)\epsilon}{(d-1)(1-\alpha)^2 + d\epsilon}, \\
    \cond{\infty,c}{\Layer{LN}, \Matrix{x}} = (d-1) \cond{\infty,\infty}{\Layer{LN}, \Matrix{x}}, \\
    \cond{c,\infty}{\Layer{LN}, \Matrix{x}} = \frac{d}{(d-1)(1-\alpha)} \begin{cases}
        \frac{\epsilon + |\alpha| \big[ 2(d-2)(1-\alpha)^2 + (2d-3) \epsilon \big]}{(d-1)(1-\alpha)^2 + d\epsilon}, & \epsilon \leq \frac{2 |\alpha| (1 - \alpha)^2}{1 - |\alpha|}, \\
        \frac{(d-1) (|\alpha| + 1) \epsilon}{(d-1)(1-\alpha)^2 + d\epsilon}, & \text{otherwise},^\dagger
    \end{cases} \\
    \cond{c,c}{\Layer{LN}, \Matrix{x}} = \frac{d}{1-\alpha} \frac{\epsilon + |\alpha| \big[ 2(d-2)(1-\alpha)^2 + (2d-3) \epsilon \big]}{(d-1)(1-\alpha)^2 + d\epsilon}.
\end{gather*}

\subsection{Massive outlier with zero-mean background}
The condition numbers of $\Layer{RN}$ are identical to the respective formulas from the zero-variance example.
To present the condition numbers of $\Layer{LN}$, we introduce auxiliary
\begin{gather*}
    \epsilon_\ast(\alpha) = (d-1) (\alpha - \alpha^2), \\
    P_{\infty,\infty}^{(1)}(\alpha) = d(d^2 - 4d + 1)\alpha^2 - d(d^2 - 2d - 3)\alpha + 2d(d-2) + d(d-1)\epsilon, \\
    P_{\infty,\infty}^{(2)}(\alpha) = d(d^2 - 6d + 3)\alpha^2 - d(d^2 - 4d + 7)\alpha + 2d(d-2) + d(d-3)\epsilon, \\
    P_{\infty,\infty}^{(3)}(\alpha) = (d-1)\alpha^2 - (d-5)\alpha + \epsilon, \\
    P_{c,\infty}(\alpha) = (2d^2 - 7d + 3)\alpha^3 - (2d^2 - 4d - 2)\alpha^2 + [3d - 5 + \epsilon(2d-3)]\alpha - \epsilon(d-2).
\end{gather*}
Then
\begin{gather*}
    \cond{\infty,\infty}{\Layer{LN}, \Matrix{x}} = \begin{cases}
        1 + \frac{(d-1) (d\alpha-1)}{d-1 + d(d-1)\alpha^2 + d\epsilon}, & \epsilon \leq \epsilon_\ast(\alpha)~\text{and}~P_{\infty,\infty}^{(1)}(\alpha) < 0~\text{and}~P_{\infty,\infty}^{(2)}(\alpha) < 0,^\dagger \\
        \frac{2(d-2)}{d-1} + \frac{2(d\alpha + 1)(d - 2 -d\alpha)}{(d-1) (d-1 + d(d-1)\alpha^2 + d\epsilon)}, & \epsilon \leq \epsilon_\ast(\alpha)~\text{and}~P_{\infty,\infty}^{(2)}(\alpha) \geq 0~\text{and}~P_{\infty,\infty}^{(3)}(\alpha) < 0, \\
        2 - \frac{2(d\alpha - 1)^2}{(d-1) (d-1 + d(d-1)\alpha^2 + d\epsilon)}, & \text{otherwise},
    \end{cases}\\
    \cond{\infty,c}{\Layer{LN}, \Matrix{x}} = 
        \left| \frac{2(d-1)}{1-d\alpha} - \frac{2(1-d\alpha)}{d-1 + d(d-1)\alpha^2 + d\epsilon} \right|, \\
    \cond{c,\infty}{\Layer{LN}, \Matrix{x}} = \begin{cases}
        (1 + \alpha) \left( 1 - \frac{d-1}{d-1 + d(d-1)\alpha^2 + d\epsilon} \right), & \epsilon > \epsilon_\ast(\alpha),^\dagger \\
        1 - \frac{(d-1)(1 - d\alpha^2)}{d-1 + d(d-1)\alpha^2 + d\epsilon}, & \epsilon \leq \epsilon_\ast(\alpha)~\text{and}~P_{c,\infty}(\alpha) < 0,^\dagger \\
        \frac{1 + \alpha(2d-3)}{d-1} + \frac{(d\alpha-1)(d-1 + 2\alpha + d(d-3)\alpha^2)}{(d-1)(d-1 + d(d-1)\alpha^2 + d\epsilon)}, & \text{otherwise},
    \end{cases}\\
    \cond{c,c}{\Layer{LN}, \Matrix{x}} = \left| \frac{1 + \alpha(2d-3)}{1 - d\alpha} - \frac{d-1 + 2\alpha + d(d-3)\alpha^2}{d-1 + d(d-1)\alpha^2 + d\epsilon} \right|.
\end{gather*}

\end{document}